\documentclass{amsart}

\usepackage{amssymb,latexsym,mathrsfs,amsmath,tikz,tikz-cd}
\usepackage{amsthm}
\usepackage{graphicx}
\usepackage{pstricks}
\usepackage{graphics}
\usepackage{psfrag}
\usepackage{pstool}
\usepackage{tikz}
\usepackage{caption}
\usepackage{subcaption}
\usepackage{etex}
\usepackage[all,2cell]{xy}
\SelectTips{eu}{12}

\newcommand{\cell}{\mathcal{C}}

\newcommand{\Kh}{\mathrm{Kh}}

\newcommand{\wt}{\widetilde}
\newcommand{\Z}{\mathbb{Z}}

\newcommand{\bd}{\mathbf{d}}

\newcommand{\cC}{\mathscr{C}}
\newcommand{\cS}{\mathscr{S}}
\newcommand{\cL}{\mathscr{L}}

\newcommand{\cw}{\mathcal{C}}
\newcommand{\M}{\mathcal{M}}
\newcommand{\X}{\mathcal{X}}

\newcommand{\R}{\mathbb{R}}

\newcommand{\E}{\mathbb{E}}

\newcommand{\CP}{\mathbb{C}\mathbf{P}}

\newcommand{\SO}{\mathrm{SO}}

\newcommand{\SL}{\mathfrak{sl}}

\newcommand{\Cat}{\mathscr{C}}
\newcommand{\gr}[1]{|{#1}|}
\newcommand{\id}{\mathrm{id}}
\newcommand{\im}{\mathrm{im}}

\newcommand{\Hom}{\mathrm{Hom}}

\newcommand{\RP}{\mathbb{R}\mathbf{P}}
\newcommand{\PP}{{}^+_+}
\newcommand{\PM}{{}^-_+}
\newcommand{\MP}{{}^+_-}
\newcommand{\MM}{{}^-_-}

\DeclareMathOperator{\Sq}{Sq}

\DeclareMathOperator{\Ob}{Ob}

\newtheorem{lemma}{Lemma}[section]

\newtheorem{proposition}[lemma]{Proposition}
\newtheorem{theorem}[lemma]{Theorem}

\theoremstyle{definition}
\newtheorem{remark}[lemma]{Remark}
\newtheorem{definition}[lemma]{Definition}
\newtheorem{example}[lemma]{Example}

\begin{document}
\parindent0em
\setlength\parskip{.1cm}

\title[Morse moves in flow categories]{Morse moves in flow categories}

\author[Dan Jones]{Dan Jones}
\address{Department of Mathematical Sciences\\ Durham University}
\email{daniel.jones@durham.ac.uk}

\author[Andrew Lobb]{Andrew Lobb}
\address{Department of Mathematical Sciences\\ Durham University}
\email{andrew.lobb@durham.ac.uk}

\author[Dirk Sch\"utz]{Dirk Sch\"utz}
\address{Department of Mathematical Sciences\\ Durham University}
\email{dirk.schuetz@durham.ac.uk}

\thanks{AL and DS were both supported by EPSRC grant EP/M000389/1, DJ was supported by an EPSRC graduate studentship.}

\begin{abstract}
We pursue the analogy of a framed flow category with the flow data of a Morse function.  In classical Morse theory, Morse functions can sometimes be locally altered and simplified by the Morse moves.  These moves include the \emph{Whitney trick} which removes two oppositely framed flowlines between critical points of adjacent index and \emph{handle cancellation} which removes two critical points connected by a single flowline.

A \emph{framed flow category} is a way of encoding flow data such as that which may arise from the flowlines of a Morse function or of a Floer functional.  The Cohen-Jones-Segal construction associates a stable homotopy type to a framed flow category whose cohomology is designed to recover the corresponding Morse or Floer cohomology.  We obtain analogues of the Whitney trick and of handle cancellation for framed flow categories: in this new setting these are moves that can be performed to simplify a framed flow category without changing the associated stable homotopy type.

These moves often enable one to compute by hand the stable homotopy type associated to a framed flow category.  We apply this in the setting of the Lipshitz-Sarkar stable homotopy type (corresponding to Khovanov cohomology) and the stable homotopy type of a matched diagram due to the authors (corresponding to $\SL_n$ Khovanov-Rozansky cohomology).
\end{abstract}

\maketitle

\tableofcontents

\section{Introduction}
\label{sec:introduction}
Given a Morse-Smale function $f: M \rightarrow \R$ on a compact Riemannian manifold $M$, it is well-known that there is a handle decomposition of $M$ corresponding to $f$.

Suppose that this handle decomposition has a handle $h^i$ of index $i$ and a handle $h^{i+1}$ of index $i+1$.  If the attaching sphere of $h^{i+1}$ intersects the belt sphere of $h^i$ in exactly one point then one can obtain a new handle decomposition of $M$ in which $h^i$ and $h^{i+1}$ are omitted but all other handles remain with suitably adjusted attaching maps.  In the Morse theory picture, intersections of the attaching sphere with the belt sphere correspond to flowlines between the critical points $p^i$ and $p^{i+1}$ which give rise to $h^i$ and $h^{i+1}$ respectively.  If there is just a single such flowline then the Morse function may be modified in a neighbourhood of that flowline so that both critical points $p^i$ and $p^{i+1}$ are removed.  This process of modifying the handle decomposition or the Morse function is known as \emph{handle cancellation}.

Suppose now that the attaching sphere of $h^{i+1}$ intersects the belt sphere of $h^i$ in more than one point, and in particular in two points $x^+$ and $x^-$ which have opposite sign.  These correspond to two flowlines between $p^i$ and $p^{i+1}$ which have opposite `framing'.  One would like to `cancel' $x^+$ and $x^-$ against each other and thus reduce the total number of intersection points by two.  In contrast to handle cancellation, there are now topological conditions that need to be satisfied before one can be sure that one can achieve this: in particular we need to be in a situation with a large enough dimension and a large enough degree of connectedness.  The process by which one can cancel such pairs of intersection points (or such pairs of flowlines) is known as the \emph{Whitney trick}.  It is the Whitney trick's failure in general in low dimensions that leads, for example, to the complexity of simply-connected smooth 4-manifold topology.

In this paper we extend the idea of handle cancellation and the Whitney trick to \emph{framed flow categories}.  A framed flow category can be thought of as a way of encoding the flow data that might arise from a Morse function or a Floer functional.  Associated to framed flow category $\cC$ is a stable homotopy type $| \cC |$ by a construction due to Cohen-Jones-Segal \cite{CJS}.  The cohomology of $|\cC|$ is designed to recover the Morse or Floer cohomology of the input.

\begin{figure}[ht]
\centerline{
{
\includegraphics[height=0.8in,width=1in]{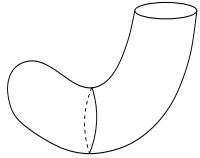}
}}
\caption{Part of a surface with three critical points with respect to the height function.  Taking the Morse-Smale metric to be restriction of the Euclidean metric, we see that there is exactly one flowline between the highest and the middle critical point, while there are two flowlines of opposite sign between the middle and the lowest critical point.}
\label{fig_sock}
\end{figure}

Roughly speaking, a \emph{flow category} $\cC$ consists of a finite number of $\Z$-graded objects where one thinks of the objects as being critical points of a Floer functional and the $\Z$-grading as being an absolute Maslov index.  Then the space of morphisms from an object of index $i$ to an object of index $j$ is a $(i-j-1)$-dimensional compact manifold-with-corners which one thinks of as being a space of flowlines between two critical points.  A \emph{framed} flow category further refines this notion.

\begin{example}
	\label{example_introduction_CP2}
	The cup product structure on cohomology allows one to distinguish between the spaces $X_1 = S^2 \vee S^4$ and $X_2 = \CP^2$.  In fact, even up to (based) stable homotopy equivalence, the spaces $X_1$ and $X_2$ are not the same.  The cup product is not a stable operation so it cannot now distinguish them.  Rather they can be distinguished by the observation that the former has a trivial second Steenrod square (a stable cohomology operation) while for latter it is non-trivial.
	
	One could ask what is the simplest framed flow category $\cC_i$ that gives rise to $X_i = |\cC_i|$ for $i = 1,2$?
	
	In both cases, one needs at least one object in each of the cohomological degrees $2$ and $4$ to generate the cohomology (we are interested in the \emph{reduced} cohomology and we are working up to \emph{based} (de-)suspension).  Let us suppose then that $\cC_1$ and $\cC_2$ each have just two objects which we shall call $p^2$ and $p^4$.
	
	Since there are no objects of degree $3$ it follows that the space of morphisms from $p^4$ to $p^2$ is a compact boundaryless 1-manifold (the absence of degree $3$ objects should be thought of as a lack of critical points at which flowlines from $p^4$ to $p^2$ can break).  Hence the morphism space is a disjoint union of circles in both cases.  How these cases differ will essentially be in the \emph{framings} of the circles.  Different choices here lead to either $X_1$ or $X_2$.  What these choices are is discussed in detail in Subsection \ref{subsec:chang_easy_examples}.
\end{example}

Our main results are the construction of moves analogous to the Whitney trick and to handle cancellation.  These appear as Theorems \ref{thm:whitney_equivalence} and \ref{cancelthm}.  The content of these theorems is in the construction of a framed flow category $\cC_W$ (respectively $\cC_H$) from a framed flow category $\cC$ whose $0$-dimensional morphism spaces suggest the possibility of performing a Whitney trick (resp.~handle cancellation).  More specifically:

Suppose that $\cC$ is framed flow category with two objects $x$ and $y$ of index differing by $1$, such that the morphism space between them contains two morphisms of differing sign.  Then we construct a framed flow category $\cC_W$ with object set $\Ob(\cC_W) = \Ob(\cC)$ and such that the morphism space between $x$ and $y$ has the same signed count but contains two fewer morphisms.  We show that we have
\begin{align*}
(\mbox{Whitney trick}) & \hspace{2.5cm} &
|\cC_W| \simeq |\cC| {\rm .} & \hspace{5cm} &
\end{align*}
On the other hand, suppose that $\cC$ is a framed flow category with two objects $x$ and $y$ of index differing by $1$, with exactly one morphism between them.  Then we construct a framed flow category $\cC_H$ with $\Ob(\cC_H) = \Ob(\cC) \setminus \{x,y\}$, such that
\begin{align*}
(\mbox{Handle cancellation}) & \hspace{1.65cm} &
|\cC_H| \simeq |\cC| {\rm .} & \hspace{5.05cm} &
\end{align*}

\begin{remark}
	\label{rem:handle_slide_comment}
	In the proof of the $h$-cobordism theorem \cite{milnor} in high dimensions, three Morse moves are used.  The third Morse move corresponds to the operation of handlesliding.  This move also has an analogue in the setting of framed flow categories, but we intend rather to discuss it in a later paper when we have a new application for it.  For now we note here that using the moves on a framed flow category $\cC$ one can treat them as operations simplifying the CW cochain complex of $|\cC|$.  Indeed, the Whitney trick ensures that the absolute count of the $0$-dimensional moduli spaces matches the relevant component of the differential, handle cancellation acts as Gauss elimination, and handle sliding acts as base change.  In this way one could find a framed flow category representative $\cS$ of any finite free cochain complex $C$ with $H^*(C) = H^*(|\cC|)$ (and furthermore in which the $0$-dimensional moduli spaces $\cS$ are determined by the differential of $C$) such that $|\cS| \simeq |\cC|$.
\end{remark}

Recently, Lipshitz-Sarkar \cite{LipSarKhov} have constructed a framed flow category $\cL^{\Kh}(D)$ associated to an oriented link diagram $D$.  The associated stable homotopy type $\X^{\Kh}(D) := |\cL^{\Kh}(D)|$ is invariant under the Reidemeister moves and its bigraded cohomology (graded cohomologically and with respect to a splitting of $\X^{\Kh}(D)$ as a wedge sum along a second \emph{quantum} grading) is exactly Khovanov cohomology \cite{kh1}.  The authors \cite{JLS} have associated a framed flow category $\cL^n(D)$ (and associated stable homotopy type $\X^n(D) = |\cL^n(D)|$) to an oriented link diagram $D$ (with a choice of decomposition into elementary tangles) and an integer $n \geq 2$.  In the case $n=2$ we showed that $\X^2 (D) = \X^{\Kh}(D)$ (up to a choice of bigrading normalization).  For $n > 2$ and $D$ a \emph{matched diagram}, the bigraded cohomology of $\X^n(D)$ is $\SL_n$ Khovanov-Rozansky cohomology \cite{khr1}.

Computations of these stable homotopy types has been performed so far essentially by computation of cohomology operations (in particular the first and second Steenrod squares).  With the two moves on framed categories corresponding to the Whitney trick and to handle cancellation, we are able to work by hand at the level of the framed flow category, reducing the number of objects and the complexity of the morphism spaces.  We use these two moves in examples at the end of the paper, each time reducing the complexity of a framed flow category until it is essentially as simple as possible and the associated stable homotopy type can be seen directly without, for example, direct computation of stable cohomology operations.

\subsection{Plan of the paper}
\label{subsec:plan}
We start by giving a brief overview of framed flow categories in Subsection \ref{subsec:flow_cats}.  Then in Subsection \ref{sec:plus-minus} (respectively \ref{sec:pushmap})  we discuss how to define the framed flow category $\cC_W$ (resp.~$\cC_H$) arising from performing the Whitney trick (resp.~handle cancellation) on a framed flow category $\cC$.  We show that the Cohen-Jones-Segal construction gives spaces for which there is a stable homotopy equivalence $|\cC_W| \simeq |\cC|$ (resp.~$|\cC_H| \simeq |\cC|$).

Then in Section \ref{sec:framings} we determine how the Whitney trick and handle cancellation affect the framings on the $1$-dimensional moduli spaces.  These framings may give rise to non-trivial topology in the associated stable homotopy type (exhibited for example in a non-trivial second Steenrod square).  In principle this could be done for moduli spaces of even higher dimension which may provide a way to detect unusual stable homotopy types, such as those that are invisible to stable cohomology operations.

We apply this in Section \ref{sec:examples} to the computation by hand of three stable homotopy types.  In particular, in Subsection \ref{subsec:pretzel_link} we consider the framed flow category $\cL^3(P)$ where $P$ is the pretzel link $P(2,-2,-2)$.  By successive application of the two moves we reduce the flow category in quantum degree $-6$ to two objects as in Example \ref{example_introduction_CP2}.  The framings on the circle moduli spaces then imply that there is a $\CP^2$ in $\X^3(P)$.  In Subsection \ref{subsec:torus34}, we do something similar to quantum degree $11$ of $\cL^{\Kh}(T_{3,4})$ where $T_{3,4}$ is the $(3,4)$ torus knot in the form of the pretzel knot diagram $P(-2,3,3)$.  In this case we reduce to three objects and considerations of framings then shows that there is an $\RP^5 / \RP^2$ in $\X^{\Kh}(T_{3,4})$.  Finally in Subsection \ref{subsec:trefoils} we consider the Lipshitz-Sarkar stable homotopy type of the disjoint union of two trefoils, and in this case we find an $\RP^2 \wedge \RP^2$ as a wedge summand as predicted by \cite[Thm.1]{LawLipSar}.

\section{Morse moves in framed flow categories}
\label{sec:morse_and_flow}

\subsection{Framed flow categories}
\label{subsec:flow_cats}
To define flow categories, we need a sharpening of the concept of smooth manifolds with corners. We will give a somewhat shortened presentation here, for more details see \cite{ich}, \cite{Laures}, \cite{LipSarKhov}, or \cite{JLS}.

\begin{definition} \label{euclidcorners}
 Let $n$ be a non-negative integer and let $\mathbf{d}=(d_0,\ldots,d_n)$ be an $(n+1)$-tuple of non-negative integers. Define
\[
 \E^\mathbf{d}=\R^{d_0}\times [0,\infty) \times \R^{d_1}\times [0,\infty) \times \cdots \times [0,\infty) \times \R^{d_n}.
\]
Furthermore, if $0\leq a < b \leq n+1$, we denote $\E^\mathbf{d}[a:b]=\E^{(d_a,\ldots,d_{b-1})}$ and set $\mathbf{d}_{a:b}=d_a+\cdots+d_{b-1}$. Also, let
\[
 \partial_i \E^\mathbf{d} = \R^{d_0}\times \cdots \times \R^{d_{i-1}} \times \{0\} \times \R^{d_i} \times \cdots \times \R^{d_n}
\]
for $i\in \{1,\ldots,n\}$.

If $J\subset \{1,\ldots,n\}$ is a non-empty subset, let 
\[
 \partial_J \E^\mathbf{d} = \bigcup_{j\in J} \partial_i \E^\mathbf{d}
\]
and $p_J\colon \E^\mathbf{d} \to [0,\infty)^{|J|}$ be the projection such that $p_J|_{\partial_J\E^\mathbf{d}}$ is constant.
\end{definition}

\begin{definition}\label{def:corner_man}
 Let $n$ be a non-negative integer and let $\mathbf{d}=(d_0,\ldots,d_n)$ be an $(n+1)$-tuple of non-negative integers. A smooth $\langle n\rangle$-manifold $M^m$ is a smooth manifold with corners together with an immersion $\imath\colon M \looparrowright \E^\mathbf{d}$ such that
\begin{enumerate}
 \item corner points of codimension $l$ in $M$ are sent to corner points of codimension $l$ in $\E^\mathbf{d}$ for all $0\leq l \leq n$;
 \item if $x\in M$ has a chart neighborhood $[0,\infty)^l\times \R^{m-l}$ with $x$ corresponding to $0\in [0,\infty)^l\times \R^{m-l}$, there is $J\subset \{1,\ldots,m\}$ with $|J|=l$, $\imath(x)\in \partial_J\E^\mathbf{d}$ and the embedding is orthogonal to $\partial_J\E^\mathbf{d}$ at $x$.
\end{enumerate}
For $i=1,\ldots,n$ define
\[
 \partial_iM = \imath^{-1}(\partial_{i}\E^\mathbf{d}).
\]
\end{definition}

The immersions can be improved to embeddings by stabilizing $\mathbf{d}$, and immersions (resp.\ embeddings) are referred to as \em neat \em if they satisfy the conditions of Definition \ref{def:corner_man}.

\begin{definition}
 A \em framed flow category \em consists of a category $\Cat$ with finitely many objects $\Ob=\Ob(\Cat)$, a function $\gr{\cdot}\colon \Ob \to \Z$, called the \em grading\em, an $(n+1)$-tuple of non-negative integers $\mathbf{d}=(d_k,\ldots,d_{n+k})$ and a collection $\varphi_\cdot$ of immersions satisfying the following:
\begin{enumerate}
 \item $k = \min\{|x|\,:\, x \in \Ob(\Cat)\}$ and $n = \max\{|x|\,:\, x \in \Ob(\Cat)\}-k$.
 \item $\Hom(x,x)=\{\id\}$ for all $x\in \Ob$, and for $x\not=y \in \Ob$, $\Hom(x,y)$ is a smooth, compact $(\gr{x}-\gr{y}-1)$-dimensional $\langle \gr{x}-\gr{y}-1\rangle$-manifold which we denote by $\mathcal{M}(x,y)$, and whose immersions are functions $\imath_{x,y}\colon \M(x,y) \to \E^\mathbf{d}[|y|:|x|]$.
 \item For $x,y,z\in \Ob$ with $\gr{z}-\gr{y}=m$, the composition map
$$\circ\colon \mathcal{M}(z,y) \times \mathcal{M}(x,z) \to \mathcal{M}(x,y)$$
is an embedding into $\partial_m\mathcal{M}(x,y)$. Furthermore,
\begin{align*}
 \circ^{-1}(\partial_i \mathcal{M}(x,y))&=\left\{ \begin{array}{lr}
                                             \partial_i \mathcal{M}(z,y)\times \mathcal{M}(x,z) & \mbox{for }i<m \\
                                             \mathcal{M}(z,y)\times \partial_{i-m}\mathcal{M}(x,z) & \mbox{for }i>m
                                            \end{array}
\right.
\end{align*}
and
\[
 i_{x,y}(p\circ q) = (i_{z,y}(p),0,i_{x,z}(q)).
\]

\item For $x\not= y\in \Ob$, $\circ$ induces a diffeomorphism
\[
 \partial_i\mathcal{M}(x,y) \cong \coprod_{z,\,\gr{z}=\gr{y}+i} \mathcal{M}(z,y) \times \mathcal{M}(x,z).
\]
\item The immersions $\imath_{x,y}$ for $x,y\in \Ob(\Cat)$ extend to immersions
\[
 \varphi_{x,y} \colon \M(x,y)\times [-\varepsilon,\varepsilon]^{\mathbf{d}_{|y|:|x|}} \looparrowright \E^\mathbf{d}[|y|:|x|]
\]
which satisfy
\begin{multline*}
 \varphi(x,y)(p\circ q,t_1,\ldots,t_{\mathbf{d}_{|y|,|x|}}) = \\
 (\varphi_{z,y}(p,t_1,\ldots,t_{\mathbf{d}_{|y|:|z|}}),0,\varphi_{x,z}(q,t_{\mathbf{d}_{|y|:|z|}+1},\ldots,t_{\mathbf{d}_{|y|:|x|}}))
\end{multline*}
for all $p\in \M(z,y)$, $q\in \M(x,z)$ where $z\in \Ob(\Cat)$.
\end{enumerate}
The manifold $\mathcal{M}(x,y)$ is called the \em moduli space from $x$ to $y$\em, and we also set $\mathcal{M}(x,x)=\emptyset$.
\end{definition}

A flow category is basically obtained by dropping the immersions. Note that the $\varphi_{x,y}$ are codimension $0$ immersions, and we therefore think of them as framings. Again we can obtain embeddings by stabilization.

In \cite{CJS} a stable homotopy type $|\Cat|$ is associated to a framed flow category $\Cat$. We quickly recall the construction in the form given by \cite{LipSarKhov}.

\begin{definition} \label{cwcomplex}
Let $\Cat$ be a framed flow category embedded into $\E^\mathbf{d}$ for some ${\bf d}=(d_k,\ldots,d_{k+n})$. For an arbitrary object $a$ in ${\rm Ob}(\Cat)$ of degree $i$, recall that for each object $b$ in ${\rm Ob}(\Cat)$ of degree $j<i$, we have the embedding
\[
\varphi_{a,b} : \mathcal{M}(a,b) \times [-\varepsilon, \varepsilon]^{\mathbf{d}_{j:i}} \rightarrow [-R,R]^{d_j} \times [0,R] \times \cdots \times [0,R] \times [-R,R]^{d_{i-1}}
\]
where $R$ is chosen to be large enough that all moduli spaces $\mathcal{M}(a,b)$ can be embedded in this way. The CW complex $|\Cat|$ consists of one $0$-cell (the basepoint) and one $(d_k + \cdots + d_{n+k-1} - k + i)$-cell $\mathcal{C}(a)$ for every object $a$ with $|a|=i$ defined as
\[
[0,R] \times [-R,R]^{d_k} \times \cdots \times [-R,R]^{d_{i-1}} \times \{ 0 \} \times [- \varepsilon , \varepsilon]^{d_i} \times \{ 0 \} \times \cdots \times \{0\} \times [-\varepsilon , \varepsilon]^{d_{n+k-1}}.
\]

Each cell $\mathcal{C}(a)$ is considered a subset of a different copy of  $[0,\infty) \times \E^\mathbf{d}$. The embedding $\varphi$ can be used to identify particular subsets 
\begin{equation} \label{product_in_boundary_cell}
\mathcal{M}(a,b) \times \mathcal{C}(b) \cong \mathcal{C}_b(a) \subset \partial_n \mathcal{C}(a)
\end{equation}
in the following way:
\begin{align*}
\mathcal{C}_b(a) = & [0,R] \times [-R,R]^{d_k} \times \cdots \times [-R,R]^{d_{j-1}} \times \{ 0 \} \times
\varphi_{a,b} \big( \mathcal{M}(a,b) \times [-\varepsilon, \varepsilon]^{\mathbf{d}_{j:i}} \big) \times \\
& \{ 0 \} \times [-\varepsilon, \varepsilon]^{d_m} \times \cdots \times \{0\} \times [-\varepsilon , \varepsilon]^{d_{A-1}} \subset \partial \mathcal{C}(a).
\end{align*}

It will be useful to introduce notation for this identification by letting
\begin{equation} \label{gamma_identification}
\Gamma_{a,b}\colon \mathcal{M}(a,b) \times \mathcal{C}(b) \rightarrow \partial_j \mathcal{C}(a)
\end{equation}
be the identification $\mathcal{M}(a,b) \times \mathcal{C}(b) \cong \mathcal{C}_b(a)$. Let
\begin{equation} \label{C_shift}
C = d_k + \cdots + d_{n+k-1} -k.
\end{equation}
Then the attaching map for each cell $\partial \mathcal{C}(a) \rightarrow |\Cat|^{(C+j-1)}$ is defined via the Thom construction for each embedding into $\partial \mathcal{C}(a)$ simultaneously. That is, for each subset $\mathcal{M}(a,b) \times \mathcal{C}(b) \cong \mathcal{C}_b(a) \subset \partial \mathcal{C}(a)$, the attaching map projects to $\mathcal{C}(b)$ (which carries trivialisation information), and sends the rest of the boundary $\partial \mathcal{C}(a) \setminus \bigcup_{b} \mathcal{C}_b(a)$ to the basepoint.
\end{definition}

The independence of the stable homotopy type of $|\Cat|$ on the various choices is discussed in \cite[\S 3]{LipSarKhov}.

\subsection{The Whitney trick in framed flow categories}
\label{sec:plus-minus}
Let $(\cC, \varphi, \imath)$ be a framed flow category containing objects $x$ and $y$ with $|x| = i$ and $|y| = i-1$, and such that $\M(x,y)$ includes two points, $P$ and $M$, with opposite framings.  We shall define a new framed flow category, written $\cC_W$, such that $|\cC_W| \simeq |\cC|$.

\begin{figure}
\centerline{
{
\psfrag{prod}{$[0,1) \times [0,1) \setminus \{(0,0)\}$}
\psfrag{=}{$=$}
\psfrag{prodmod}{$([0,1) \times [0,1) \setminus \{(0,0)\}) \sqcup ([0,1) \times [0,1) \setminus \{(0,0)\}) / \sim_0$}
\includegraphics[height=2.9in,width=3.2in]{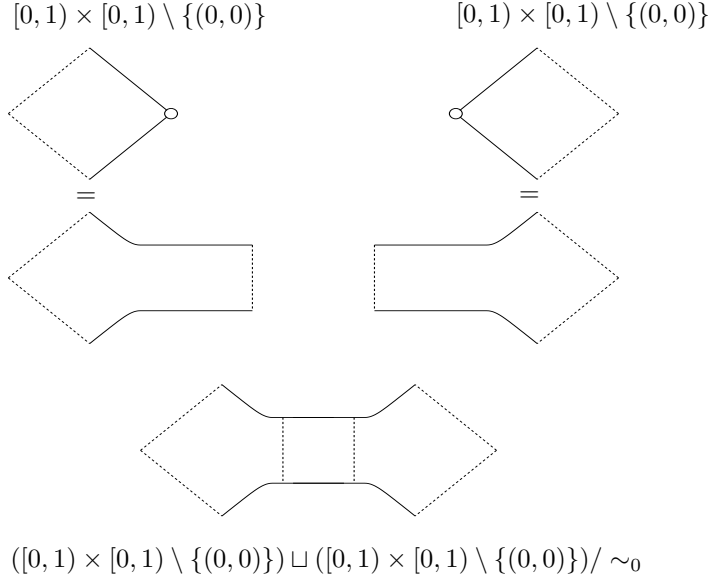}
}}
\caption{We show how to glue together two copies of $[0,1) \times [0,1) \setminus \{(0,0)\}$ by the orientation-reversing gluing equivalence relation $\sim_0$.  The dotted line represents the open boundary.}
\label{fig:2dimglue}
\end{figure}

\begin{definition}
\label{def:CH+-}
With $\cC$ as above, we define the object set of $\cC_W$ by $\Ob(\cC_W) = \{ \bar{a} : a \in \Ob(\cC) \}$.  We now give the moduli spaces of $\cC_W$.

\begin{enumerate}
\item $\M(\bar{x}, \bar{y}) = \M(x,y) \setminus \{ P , M \}$.

\item If $a \in \Ob(\cC)$ is such that $\M(a,x) \not= \phi$ then we have
\[ \{ P , M \} \times \M(a,x) \subset \partial \M(a,y) {\rm .}\]

\noindent Let $\{ P , M \} \times [0,1) \times \M(a,x)$ be a collar neighbourhood of this subset and write $(P, t, p) \sim (M , 1/2 - t, p)$ for $0 < t < 1/2$ and all $p \in \M(a,x)$.  Now we define
\[ \M(\bar{a}, \bar{y}) = (\M(a,y) \setminus \{ P , M \} \times \M(a,x))/\sim {\rm .}\]

\item Similarly, if $b \in \Ob(\cC)$ is such that $\M(y,b) \not= \phi$ then we have
\[ \M(y,b) \times \{ P , M \} \subset \partial \M(x,b) {\rm .}\]

\noindent Let $\M(y,b) \times [0,1) \times \{ P , M \}$ be a collar neighbourhood of this subset and write $(p, t, P) \sim ( p, 1/2 - t , M)$ for $0 < t < 1/2$ and all $p \in \M(y,b)$.  And we define
\[ \M(\bar{x}, \bar{b}) = (\M(x,b) \setminus \M(y,b) \times \{ P , M \})/\sim {\rm .}\]

\item If $a,b \in \Ob(\cC)$ and both $\M(a,x) \not= \phi$ and $\M(y,b) \not= \phi$ then we have
\[ \M(y,b) \times \{P,M\} \times \M(a,x) \subset \partial \M(a,b) {\rm .} \]

\noindent Let $\M(y,b) \times [0,1) \times \{P,M\} \times [0,1) \times \M(a,x)$ be a neighbourhood of this subset in $\M(a,b)$.  Choose this neighbourhood such that $\M(y,b) \times \{0\} \times \{P,M\} \times [0,1) \times \M(a,x)$ is a collar neighbourhood of $\M(y,b)\times \{P,M\}\times \M(a,x)$ in $\M(y,b) \times \M(a,y)$ and $\M(y,b) \times[0,1) \times \{P,M\} \times \{0\}\times \M(a,x)$ is a collar neighbourhood of $\M(y,b)\times \{P,M\}\times \M(a,x)$ in $\M(x,b) \times \M(a,x)$.

Now, using the equivalence relation $\sim_0$ given in Figure \ref{fig:2dimglue}, we define the equivalence $\sim$ on $\M(y,b) \times[0,1) \times \{P,M\} \times[0,1) \times \M(a,x)$ by requiring
\[ (p,r, P,s, q) \sim (p,t,Q,u,q) \]

\noindent for $r,s,t,u \in [0,1)$, all $p \in \M(y,b)$, and all $q \in \M(a,x)$ if and only if
\[ (r,s) \sim_0 (t,u) {\rm .} \]

Then we define
\[ \M(\bar{a}, \bar{b}) = (\M(a,b) \setminus \M(y,b) \times \{P,M\} \times \M(a,x)) /\sim{\rm .}\]

\noindent Note that the equivalence relation $\sim_0$ may be chosen to be compatible with (2) and (3).
\item In all other cases we define $\M(\bar{a}, \bar{b}) = \M(a,b)$.
\end{enumerate}
\end{definition}

Clearly this defines a flow category $\cC_W$.

Now suppose that $\cC$ comes with a framed embedding $(\cC, \imath, \varphi)$ into the Euclidean space $\E^\bd$.  After possibly a stabilization and an isotopy, we may assume that $\imath_{x,y}$ takes $P$ and $M$ to the points $(-1,0,\ldots,0)$ and $(1,0,\ldots,0)$ in $\R^{d_i}$ respectively.  Furthermore we may assume that for all $a,b \in \Ob(\cC)$ such that $|a| = i$ and $|b| = i-1$ we have $\imath_{a,b}^{-1} (\R \times \{ (0,\ldots,0) \}) \subseteq \{ P , Q \}$.

Also, thinking of the framings of $P$ and $M$ as ordered $d_i$-tuples of orthonormal vectors, we may assume that the framings of $P$ and $M$ differ only in the first vector, and these vectors are $(1,0,\ldots,0)$ and $(-1,0,\ldots,0)$ for $P$ and $M$ respectively.

Now, collar neighbourhoods are embedded transversely to the boundaries of $\E^\bd$.  So if $a \in \Ob(\cC)$ such that $\M(a,x) \not= 0$ then we may assume that the embedding $ \imath_{a,y}|_{\{ P , M \} \times [0,1) \times \M(a,x)}$ of the collar neighbourhood of $\{ P , M \} \times \M(a,x)$ described in Definition \ref{def:CH+-} satisfies
\begin{eqnarray*}
\imath_{a,y}|_{\{ P , M \} \times [0,1) \times \M(a,x)} (P,t,p) &=& ((-1,0,\ldots,0),t,\imath_{a,x}(p)) \\
{\rm and} \,\,\, \imath_{a,y}|_{\{ P , M \} \times [0,1) \times \M(a,x)} (M,t,p) &=& ((1,0,\ldots,0),t,\imath_{a,x}(p))
\end{eqnarray*}
\noindent for all $t \in [0,1)$ and all $p\in \M(a,x)$ where the image lies in
\[ \E^\bd[i-1,|a|] = \R^{d_i}\times [0,\infty) \times \E^\bd[i, |a|] {\rm .} \]

\noindent Furthermore, we may assume that the framing of this collar neighbourhood is given by the product framing of $\{P,M\} \times \M(a,x)$ (via the identification of normal bundles using the Euclidean inner product).

Similarly, we may assume for $b \in \Ob(\cC)$ with $\M(y,b) \not= \phi$ that we have
\begin{eqnarray*}
\imath_{x,b}|_{\M(y,b) \times [0,1) \times \{ P , M \}} (p,t,P) &=& (\imath_{y,b}(p),t,(-1,0,\ldots,0)) \\
{\rm and} \,\,\, \imath_{x,b}|_{\M(y,b) \times [0,1) \times \{ P , M \}} (p,t,Q) &=& (\imath_{y,b}(p),t,(1,0,\ldots,0))
\end{eqnarray*}
\noindent for all $t \in [0,1)$ and all $p\in \M(y,b)$, and that the framing is given by the product framing of $\M(y,b) \times \{P,M\}$.

Finally if $a,b \in \Ob(\cC)$ and both $\M(a,x) \not= \phi$ and $\M(y,b) \not= \phi$, then we may assume that the embedding $\imath_{a,b}$, in the neighbourhood of $\M(y,b) \times \{P,M\} \times \M(a,x)$ given in Definition \ref{def:CH+-}, satisfies
\begin{align*}
\imath_{a,b}|_{\M(y,b) \times [0,1) \times  \{P,M\} \times [0,1) \times \M(a,x)}(p,r, P,s, q) =\\
 (\imath_{y,b}(p),r,(-1,0,\ldots,0),s,\imath_{a,x}(q)) {\rm ,} \\
\imath_{a,b}|_{\M(y,b) \times [0,1) \times \{P,M\} \times [0,1) \times \M(a,x)}(p,r, M,s, q) =\\
(\imath_{y,b}(p),r,(1,0,\ldots,0),s,\imath_{a,x}(q))
\end{align*}

\noindent where the image lies in
\[ \E^\bd[|b|,|a|] = \E^\bd[|b|,i-1] \times [0,\infty) \times \R^{d_i} \times [0,\infty) \times \E^\bd[i, |a|] {\rm .} \]

\noindent And we can assume that the framing on this neighbourhood is given by the product framing on $\M(y,b) \times \{P,M\} \times \M(a,x)$.

\begin{figure}
\centerline{
{
\psfrag{prod}{$[0,1) \times [0,1)$}
\psfrag{=}{$=$}
\psfrag{prodmod}{$[0,1) \times [0,1) \sqcup [0,1) \times [0,1) / \sim_0$}
\includegraphics[height=3.5in,width=2.5in]{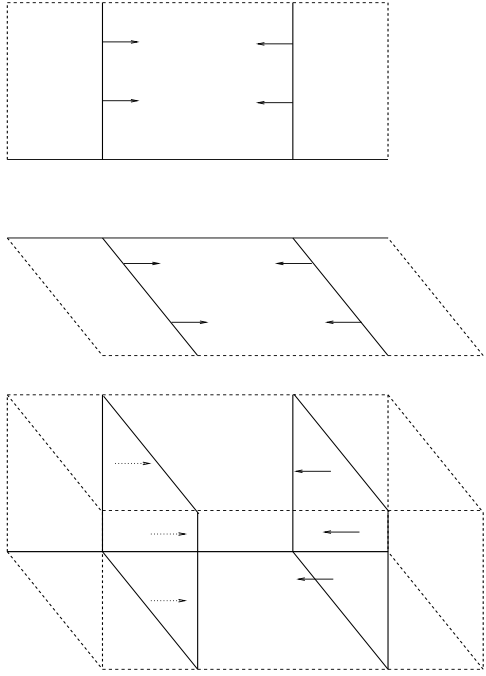}
}}
\caption{The three diagrams illustrate the framed embeddings of collar neighbourhoods of $\{ P , M \} \times \M(a,x)$, $\M(y,b) \times \{P,M\}$, and $\M(y,b) \times \{P,M\} \times \M(a,x)$ respectively.  In the first two cases we have projected to the product of the first coordinate of $\R^{d_i} = \E^\bd[i-1,i]$ and the relevant $[0,\infty)$ factor of $\E^\bd$, and in the third case we have projected to the product of the first coordinate of $\R^{d_i}$ and the two relevant $[0,\infty)$ factors.  In the omitted factors the embeddings are given by the maps $\imath_{a,x}$, $\imath_{y,b}$, and $(\imath_{y,b}, \imath_{a,x})$.  The arrows give the vector that corresponds to the first coordinate of the framing of $\{P,M\}$.  All subsequent vectors in the framing of $P$ agree with those of $Q$.}
\label{fig:precan+-}
\end{figure}

We have illustrated the embeddings of these collar neighbourhoods in Figure \ref{fig:precan+-}, where we have included only the interesting factors of $\E^\bd$.

\begin{figure}
\centerline{
{
\psfrag{prod}{$[0,1) \times [0,1)$}
\psfrag{=}{$=$}
\psfrag{prodmod}{$[0,1) \times [0,1) \sqcup [0,1) \times [0,1) / \sim_0$}
\includegraphics[height=3.5in,width=2.5in]{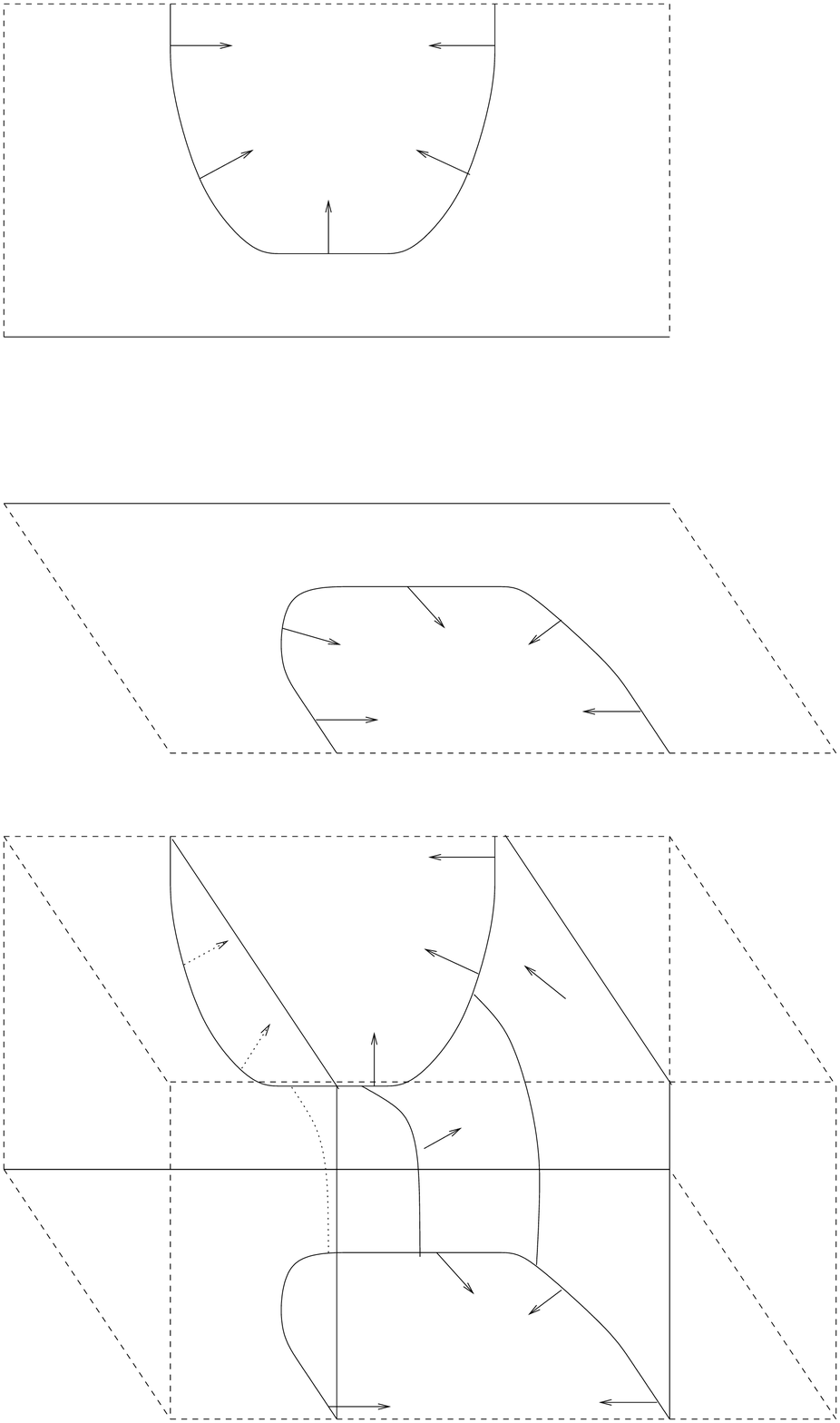}
}}
\caption{The three diagrams illustrate the framed embeddings of open subsets of $\M(\bar{a},\bar{y})$, $\M(\bar{x}, \bar{b})$, and $\M(\bar{a},\bar{b})$ containing the subsets in which moduli spaces of $\cC$ have been glued together.  The coordinates projected to agree with those of Figure \ref{fig:precan+-}.  In the coordinates that have not been shown, and outside the regions shown, the embeddings and the framings are inherited from $(\cC, \imath, \varphi)$.}
\label{fig:can+-}
\end{figure}

\begin{definition}
\label{def:CH+-embedframe}
We give an embedding $\bar{\imath}$ and framing $\bar{\varphi}$ of $\cC_W$ in the Euclidean space $\E^\bd$.

Firstly, in case (1) of Definition \ref{def:CH+-}, the embedding and framing of $\M(\bar{x},\bar{y})$ is defined by restriction of $\imath_{x,y}$ and of $\varphi$.  In case (5) of Definition \ref{def:CH+-}, the embeddings and framings of $\cC_W$ agree with those of $\cC$.

In cases (2), (3), and (4) of Definition \ref{def:CH+-}, we define the framing and embedding of $\cC_W$ to differ only from those of $\cC$ in a small neighbourhood of the gluing regions given in those cases.  How the framings and embeddings differ is described in Figure~\ref{fig:can+-}.
\end{definition}

It remains to conclude that the stable homotopy type associated to $(\cC_W, \bar{\imath}, \bar{\varphi})$ agrees with that associated to $(\cC,\imath,\varphi)$.  In fact, we can say a little more.

\begin{proposition}
\label{CH+-=C}
We write $|\cC|$ and $|\cC_W|$ respectively for the CW-complexes (and not just the stable homotopy type of those complexes) determined by applying the construction due to Cohen-Jones-Segal to the framed flow categories $(\cC, \imath, \varphi)$ and $(\cC_W, \bar{\imath}, \bar{\varphi})$.  Then we write $f_a$ and $f_{\bar{a}}$ respectively for the attaching maps of the cells $\cell(a)$ and $\cell(\bar{a})$ for all $a \in \Ob(\cC)$.  

For each $t \in [0,1]$ there exist maps
\[ F_{a,t}: \{ t \} \times \partial \cell(a) \rightarrow Y^{|a| - 1}_t \]

\noindent where $Y^i_t$ is defined inductively for increasing $i$ by
\[ Y^i_t = Y^{i-1}_t \cup_{F_{a,t}} \{t\} \times \cell(a) \]

where the union is taken over all cells $\cell(a)$ of dimension $i$ and $Y^0_t = \{ {\rm pt} \}$.

These $F_{a,t}$ are such that the maps
\[ F_a : [0,1] \times \partial \cell(a) \rightarrow Y^{|a|-1} : (t,x) \mapsto F_{a,t}(x) \]

are continuous where $Y^i$ is defined inductively for increasing $i$ by
\[ Y^i = Y^{i-1} \cup_{F_a} [0,1] \times \cell(a) \]

and $Y^0 = \{pt\}$.

Furthermore, identifying $\cell(a)$ with $\{0\} \times \cell(a)$ and $\cell(\bar{a})$ with $\{1\} \times \cell(a)$ we have that $F_{a,0} = f_a$ and $F_{a,1} = f_{\bar{a}}$.
\end{proposition}

It is then easy to see that $|\cC| = \cup_i Y^i_0$ and $|\cC_W| = \cup_i Y^i_1$ are both subsets of $\cup_i Y^i$, and that this larger space deformation retracts onto both $|\cC|$ and onto $|\cC_W|$.  Hence the following is immediate.

\begin{theorem}
	\label{thm:whitney_equivalence}
The spaces $|\cC|$ and $|\cC_W|$ from Proposition \ref{CH+-=C} are homotopy equivalent.\qed
\end{theorem}

\begin{figure}
\centerline{
{
\psfrag{prod}{$[0,1) \times [0,1)$}
\psfrag{=}{$=$}
\psfrag{prodmod}{$[0,1) \times [0,1) \sqcup [0,1) \times [0,1) / \sim_0$}
\includegraphics[height=3in,width=5in]{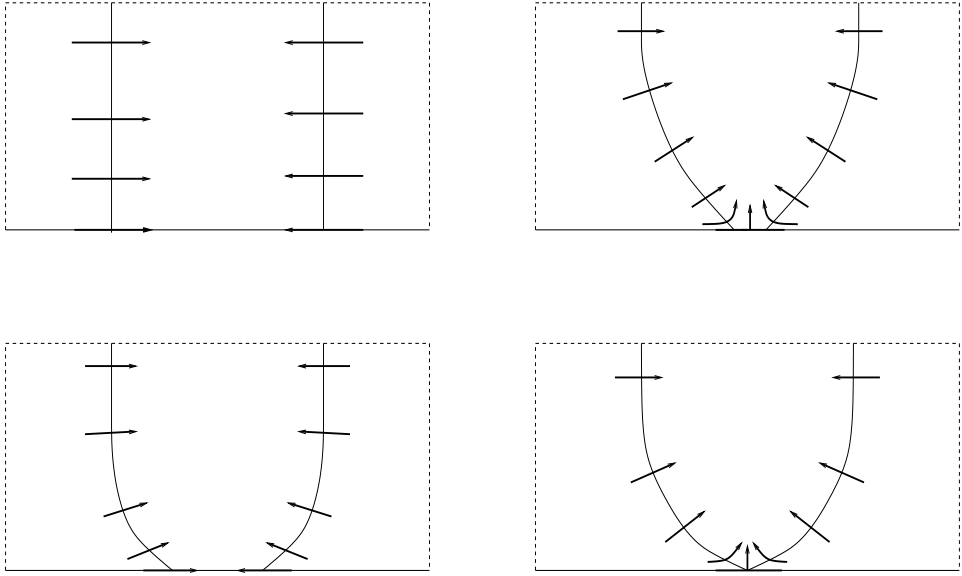}
}}
\caption{We show the sequence of attaching maps $F_{a,t}$ from Proposition \ref{CH+-=C} as for $0 \leq t \leq 1/2$.}
\label{fig:C+-mash}
\end{figure}

\begin{proof}[Proof of Proposition \ref{CH+-=C}.]  We wish to see that we can continuously deform the attaching maps of the complex $|\cC|$ to arrive at the attaching maps of the complex $|\cC_W|$.

The attaching maps of the cells of $|\cC|$ are determined by the framed embedding $(\cC, \imath, \varphi)$.  Ideally we would like a smooth deformation through framed embeddings to arrive at the framed embedding $(\cC, \bar{\imath}, \bar{\varphi})$, but since $\cC \not= \cC_W$ this cannot be achieved.  Instead we pass from $(\cC, \imath, \varphi)$ to $(\cC_W, \bar{\imath}, \bar{\varphi})$ through the deformations in Figure \ref{fig:C+-mash}.

We have illustrated in Figure $\ref{fig:C+-mash}$ the attaching map $f_{a,t}|_{f_{a,t}^{-1}(\cell(y))}$ thought of as a map to $\cell(y)$ for $0 \leq t \leq 1/2$ (the reader should be able to fill in the pictures for $1/2 \leq t \leq 1$ herself) where $\M(a,x) \not= \phi$.  The attaching maps are constructed in the usual way: via a framing and then applying the Thom construction, with the exception illustrated in the third and fourth diagrams of Figure \ref{fig:C+-mash}.  In all diagrams the thick arrows represent (a vector in) the framing, and in fact are the fibres that are going to wrap once over the first coordinate of $[-\epsilon,\epsilon]^{d_i}$ thought of as a factor of $\cell(y)$.  In the third and fourth diagram there is one thick arrow with a single head and a double tail, while all other arrows are homeomorphic to the interval $[-\epsilon, \epsilon]$.  These singular fibres again map to the first coordinate of $[-\epsilon,\epsilon]^{d_i}$, with the maps determined by continuity of $f_{a,t}$.  In the factors not illustrated the framings and embeddings of course do not change with $t$.

The attaching maps $f_{x,t}|_{f_{x,t}^{-1}(\cell(y))}$ thought of as maps to $\cell(y)$ are obtained from Figure \ref{fig:C+-mash} by just looking at the horizontal boundaries of the diagram.  The attaching maps $f_{x,t}|_{f_{a,t}^{-1}(\cell(b))}$ for those $b$ with $\M(y,b) \not= \phi$ are obtained by rotating the diagrams by $\pi/2$.  The attaching maps $f_{a,t}|_{f_{a,t}^{-1}(\cell(b))}$ for $a,b$ with $\M(a,x) \not= \phi$ and $\M(y,b) \not= \phi$ are swept out by rotating the diagrams in \ref{fig:C+-mash} by $\pi/2$.  Finally, for all other $a,b$ the maps $f_{a,t}|_{f_{a,t}^{-1}(\cell(b))}$ do not vary with $t$.
\end{proof}

\subsection{Handle cancellation in framed flow categories}
\label{sec:pushmap}
Throughout this subsection, let $\Cat$ denote a framed flow category $(\Cat,\imath,\varphi)$ with two of its objects having a one-point moduli space between them, $\mathcal{M}(x,y) = \ast$.  Let $|x|=i$ and $|y| = i-1$.  We shall show that the space arising from the handle-cancelled framed flow category $\cC_H$ is stably homotopy equivalent to the space arising from the original framed flow category $\cC$.  This appears as Theorem \ref{cancelthm} but before we can state it we need to define, embed, and frame $\cC_H$.  We denote the CW complex associated to a framed flow category $\cS$ by the Cohen-Jones-Segal construction (as opposed to its stable homotopy class) by $|\cS|$ and we term this the \emph{realisation} of $\cS$.

\begin{definition} \label{cancelledcat}
Denote by $\Cat_H$ the flow category whose object set is given by
\[ \Ob(\cC_H) = \{ \bar{a} : a \in (\Ob(\cC) \setminus \{ x,y \} ) \} \]
and whose moduli spaces are given by
\[
 \mathcal{M}(\bar{a},\bar{b}) = \mathcal{M}(a,b)\cup_f \big( \mathcal{M}(x,b)\times \mathcal{M}(a,y) \big)
\]
where $f$ identifies the subsets 
\[
\mathcal{M}(x,b)\times \mathcal{M}(a,x)\cup \mathcal{M}(y,b)\times \mathcal{M}(a,y) \subset \mathcal{M}(a,b)
\]
and
\begin{multline*}
 \mathcal{M}(x,b)\times (\mathcal{M}(x,y)\times \mathcal{M}(a,x)) \cup (\mathcal{M}(y,b)\times \mathcal{M}(x,y))\times \mathcal{M}(a,y) \\ \subset \mathcal{M}(x,b)\times \mathcal{M}(a,y).
\end{multline*}
We call $\Cat_H$ the {\it cancelled category} (relative to $x$ and $y$) of $\Cat$.
\end{definition}

It follows from the existence of collar neighbourhoods for $\langle n\rangle$-manifolds, see \cite[Lemma 2.1.6]{Laures}, that $\mathcal{M}(\bar{a},\bar{b})$ is a $(|a|-|b|-1)$-dimensional $\langle |a|-|b|-1 \rangle$-manifold, and that $\Cat_H$ is a flow category, with object grading inherited from $\cC$.

For the framed flow category $\Cat$, we must provide framed neat embeddings of $\Cat_H$, so that we can form $|\cC_H|$.  Recall that $\mathcal{C}(x)$ is a $(C+i)$-cell (for some $C>>0$) and a single copy of $\mathcal{C}(y) \cong \mathcal{M}(x,y) \times \mathcal{C}(y)$ is identified with a subset on the boundary of $\cell(x)$ (in fact on the $(i-1)$-face) via
\begin{align*}
\mathcal{C}_y(x) = & [0,R] \times [-R,R]^{d_{B}} \times \cdots \times [-R,R]^{d_{i-2}} \times \{ 0 \} \times \\
& \imath_{x,y} \big( \mathcal{M}(x,y) \times [-\varepsilon, \varepsilon]^{d_{i-1}} \big) \times \{ 0 \} \times [-\varepsilon, \varepsilon]^{d_i} \times \cdots \times \{0\} \times [-\varepsilon , \varepsilon]^{d_{A-1}} \\
& \subset \partial_{i-1} \mathcal{C}(x){\rm .}
\end{align*}
Note that we can assume $\imath_{x,y}$ embeds the point $\mathcal{M}(x,y) = \{\ast\}$ as $\imath_{x,y}(\ast) = 0$ in $\R^{d_{i-1}}$.  The framing of that point gives a homeomorphism between $\mathcal{C}_y(x)$ and the $(C+i-1)$-cell $\mathcal{C}(y)$.

Choose a homeomorphism $f: \partial \cell(x)\setminus {\rm int}(\cell_y(x)) \rightarrow \cell_y(x)$ which is the identity on $\partial \cell_y(x)$ and is smooth on all boundaries of $\cell(x)$ of codimension $\geq 1$.

Let
\[ A_f = (\partial \cell(x)\setminus {\rm int}(\cell_y(x))) \times [0,1] \sqcup \cell_y(x) / \sim \]
\noindent (where $\sim$ is defined by $(p,s) \sim (p,t)$ for all $p \in \partial \cell_y(x)$ and $s,t \in [0,1]$ and $(q,1) \sim f(q)$ for all $q \in \partial \cell(x)\setminus {\rm int}(\cell_y(x))$)
be a quotient of the mapping cylinder of $f$.

Note that the boundary of $A_f$ is naturally identified with $\partial \cell(x)$ and choose a identification, smooth on the interior, $A_f = \cell(x)$ respecting this on the boundary.

\begin{definition}
\label{pushmap}
For $t \in [0,1]$, define the map $\Psi_t :  \partial \cell(x)\setminus {\rm int}(\cell_y(x)) \rightarrow \cell(x)$ by composing the inclusion $(\partial \cell(x)\setminus {\rm int}(\cell_y(x))) \times \{t\} \subset \partial \cell(x)\setminus {\rm int}(\cell_y(x)) \times [0,1] \sqcup \cell_y(x)$ with the quotient $\sim$.
\end{definition}

Recall from Definition \ref{cancelledcat} that $\mathcal{M}(\bar{a},\bar{b})$ is formed by gluing the two spaces $\mathcal{M}(a,b)$ and $\mathcal{M}(x,b)\times \mathcal{M}(a,y)$ along their common boundaries.  Thus, in order to embed $\M(\bar{a}, \bar{b})$, we shall define an embedding for each of these spaces separately, and emphasise how the gluing works.  The former of the two is embedded with its original embedding from $(\Cat, \imath , \varphi)$, while a framed embedding $\Gamma_{x,b \times a,y}$ of the product moduli spaces $\mathcal{M}(x,b) \times \mathcal{M}(a,y)$ is described in Lemma \ref{embedding1} along with a description of the gluing.

A framed embedding $\Gamma_{\bar{a},\bar{b}}$ of the moduli spaces $\mathcal{M}(\bar{a},\bar{b})$ is then described in Lemma \ref{embedding2}.  Finally, some alteration is needed to ensure that these embeddings are \emph{neat} embeddings, and this is described in Lemma \ref{embedding3}.

In the remainder of this section let $a,b \in \Ob{\cC} \setminus \{ x,y \}$ with $|a| = m > n = |b|$.  We shall write
\[
\Gamma_{a,b} : \mathcal{M}(a,b) \times \mathcal{C}(b) \cong \mathcal{C}_b(a) \hookrightarrow \partial \mathcal{C}(a)
\]
for the inclusions.

\begin{lemma} \label{embedding1}
There is an embedding
\[
\Gamma_{x,b \times a,y} : \mathcal{M}(x,b) \times \mathcal{M}(a,y) \times \mathcal{C}(b) \rightarrow \partial \mathcal{C}(a) {\rm .}
\]
Moreover, this embedding can be defined to agree with $\Gamma_{a,b}$ on the boundary subset $\big( \mathcal{M}(y,b) \times \mathcal{M}(x,y) \big) \times \mathcal{M}(a,y) \cup \mathcal{M}(x,b) \times \big( \mathcal{M}(x,y) \times \mathcal{M}(a,x) \big)$.
\end{lemma}

\begin{proof}
Consider each $\mathcal{M}(x,b)$ embedded into Euclidean space as
\begin{align*}
\imath_{x,b}: \mathcal{M}(x,b) & \times [-\varepsilon, \varepsilon]^{d_n + \cdots + d_{i-1}} \rightarrow \\
& [-R,R]^{d_n} \times [0,R] \times \cdots \times [0,R] \times [-R,R]^{d_{i-1}} {\rm .}
\end{align*}
Varying $t$ in $[0,1]$ provides an interval of framed embedded subspaces $\Psi_t |_{ \mathcal{C}_b(x)} ( \mathcal{C}_b(x) ) $ inside $\mathcal{C}(x)$, and in particular we now consider the framed embedded subspace $\Psi_1 |_{ \mathcal{C}_b(x)} (\mathcal{C}_b(x)) $ inside $\mathcal{C}_y(x)$ (see Figure \ref{fig:collapsed_to_y}). This provides an embedding of $\mathcal{M}(x,b) \times \mathcal{C}(b)$ into
\begin{align*}
& [0,R] \times [-R,R]^{d_{B}} \times \cdots \times [-R,R]^{d_{i-2}} \times \{ 0 \} \times \imath_{x,y} \big( \mathcal{M}(x,y) \times [-\varepsilon, \varepsilon]^{d_{i-1}} \big) \\
& \times \{ 0 \} \times [-\varepsilon, \varepsilon]^{d_i} \times \cdots \times \{0\} \times [-\varepsilon , \varepsilon]^{d_{A-1}} = \mathcal{C}_y(x) {\rm .}
\end{align*}
given by
\[
\Psi_1 |_{ \mathcal{C}_b(x)} \circ \Gamma_{x,b} : \mathcal{M}(x,b) \times \mathcal{C}(b) \rightarrow \mathcal{C}_y(x) {\rm .}
\]

Now, abusing notation slightly, let $\Gamma_{x,y}^{-1}:\mathcal{C}_y(x) \rightarrow \mathcal{M}(x,y) \times \mathcal{C}(y)$ be the obvious homeomorphism, so that 
\[
\Gamma_{x,y}^{-1} \circ \Psi_1 |_{ \mathcal{C}_b(x)} \circ \Gamma_{x,b} : \mathcal{M}(x,b) \times \mathcal{C}(b) \rightarrow \mathcal{C}(y)
\]
provides an embedding of $\mathcal{M}(x,b) \times \mathcal{C}(b)$ into a one-point product of $\mathcal{C}(y)$. Next, to embed the product moduli space $\mathcal{M}(x,b) \times \mathcal{M}(a,y) \times \mathcal{C}(b)$, consider the identification
\[
\Gamma_{a,y} : \mathcal{M}(a,y) \times \mathcal{C}(y) \rightarrow \mathcal{C}_y(a) \subset \partial_{i-1} \mathcal{C}(a)
\]
with $\mathcal{M}(x,b) \times \mathcal{C}(b) $ embedded into the $\mathcal{C}(y)$ component via $\Gamma_{x,y}^{-1} \circ \Psi_1 |_{ \mathcal{C}_b(x)} \circ \Gamma_{x,b}$. Then define the embedding
\begin{equation} \label{eq:product_embedding}
\Gamma^{'}_{x,b \times a,y} : \mathcal{M}(x,b) \times \mathcal{M}(a,y) \times \mathcal{C}(b) \rightarrow \partial \mathcal{C}(a)
\end{equation}
by $\Gamma^{'}_{x,b \times a,y} ( p,q, \delta) = \Gamma_{a,y} \big( q, \Gamma_{x,y}^{-1} \circ \Psi_1 |_{ \mathcal{C}_b(x)} \circ \Gamma_{x,b}(p, \delta) \big)$ (see Figure \ref{fig:cancel_pic2}).

Notice that the embeddings $\Gamma^{'}_{x,b \times a,y}$ and $\Gamma_{a,b}$ agree on $\big( \mathcal{M}(y,b) \times \mathcal{M}(x,y) \big) \times \mathcal{M}(a,y)$ since $\Psi_1$ is the identity there, but on $\mathcal{M}(x,b) \times \big( \mathcal{M}(x,y) \times \mathcal{M}(a,x) \big)$ they disagree.  We shall next fix this by adding a collar
\[
\mathcal{M}(x,b) \times \mathcal{M}(a,x) \times [0,1] \times \mathcal{C}(b)
\]
which we glue onto $\mathcal{M}(x,b) \times \mathcal{M}(a,y) \times \mathcal{C}(b)$ along $\mathcal{M}(x,b) \times \mathcal{M}(a,x) \times \{ 1 \} \times \mathcal{C}(b)$ in the obvious way.  Note of course that this results in the same space and we shall abuse notation by referring to the space \emph{with}  the collar also as $\mathcal{M}(x,b) \times \mathcal{M}(a,y) \times \mathcal{C}(b)$.

Consider the embedding
\begin{equation} \label{eq:product_embedding_with_gluing}
\Gamma_{x,b \times a,y} : \mathcal{M}(x,b) \times \mathcal{M}(a,y) \times \mathcal{C}(b) \rightarrow \partial \mathcal{C}(a)
\end{equation}
that is defined as $\Gamma^{'}_{x,b \times a,y}$ away from the collar, and defined as $\Gamma_{x,b \times a,y} (p,q,t, \delta) = \Gamma_{a,x} \big( q, \Psi_t |_{\mathcal{C}_b(x)}( \Gamma_{x,b} (p, \delta) ) \big)$ on points $(p,q,t, \delta)$ within the collar. Varying $t$ from $0$ to $1$ has the effect of tracing from the embedding $\Gamma_{a,b}$ (when $t=0$) to the embedding $\Gamma_{x,b \times a,y}^{'}$ (when $t=1$).
We need to check that the map on the collar $\mathcal{M}(x,b) \times \mathcal{M}(a,x) \times [0,1]$ does not affect the intersection with the collar neighbourhood $\mathcal{M}(y,b) \times \mathcal{M}(a,y) \times [0,1]$, which is $\mathcal{M}(y,b) \times \mathcal{M}(a,x) \times [0,1]^2$.  In fact, since $\Gamma_{x,b}$ sends the boundary subset $\mathcal{M}(y,b) \times \mathcal{M}(x,y) \times \mathcal{C}(b)$ to $\mathcal{C}_b(x) \cap \mathcal{C}_y(x) \subset \partial \mathcal{C}_y(x)$, $\Psi_t$ has no effect on this particular collar. 
Hence, $\Gamma_{x,b \times a,y}$ provides an embedding satisfying the required properties.
\end{proof}

Since the embeddings defined in the proof of Lemma \ref{embedding1} may seem a little abstract, let us consider an example.

\begin{example} \label{example_cancel_1}
Let $\Cat_{{\rm Ex}}$ be a framed flow category with ${\rm Ob}(\Cat_{{\rm Ex}}) = \{a,x,c,y,b \}$ such that $|a|=2$, $|x| = |c| = 1$, and $|y|=|b|=0$. Here is an illustration of $\Cat_{{\rm Ex}}$
\begin{center}
\begin{tikzpicture}
\node [left] (a)at (0,0) {$a$};
\node [left] (x)at (0,-1.5) {$x$};
\node [left] (y)at (0,-3) {$y$};
\node [right] (c)at (2,-1.5) {$c$};
\node [right] (b)at (2,-3) {$b$};

\draw[-] (0,0) -- (0,-1.5) node[blue, left,pos=0.5]{$p_1$};
\draw[-] (0,0) -- (2,-1.5) node[blue, right,pos=0.5]{$p_2$};
\draw[-] (2,-1.5) -- (0,-3) node[blue, right,pos=0.35]{$q_2$};
\draw[draw=white,double=black,very thick] (0,-1.5) -- (2,-3) node[blue, left,pos=0.35]{$q_1$};
\draw[-] (0,-1.5) -- (0,-3) node[blue, left,pos=0.5]{$\ast$};
\draw[-] (2,-1.5) -- (2,-3) node[blue, right,pos=0.5]{$q_3$};

\filldraw (0,0) circle (2pt)
(0,-1.5) circle (2pt)
(2,-1.5) circle (2pt)
(0,-3) circle (2pt)
(2,-3) circle (2pt);
\end{tikzpicture}
\end{center}
in which the $0$-dimensional moduli spaces are all single points that are labelled in blue.  The $1$-dimensional moduli spaces need to be given as well, and these can be drawn as
\begin{center}
\begin{tikzpicture}
\node (mod_a_b)at (0,0) {$\mathcal{M}(a,b) = $};
\node [label=below:{\color{blue}$q_1 \cdot p_1$}] (mod_a_b_1)at (1.5,0) {};
\node [label=below:{\color{blue}$q_3 \cdot p_2$}] (mod_a_b_2)at (6.5,0) {};

\node (mod_a_y)at (0,-2) {$\mathcal{M}(a,y) = $};
\node [label=below:{\color{blue}$\ast \cdot p_1$}] (mod_a_y_1)at (1.5,-2) {};
\node [label=below:{\color{blue}$q_1 \cdot p_2$}] (mod_a_y_2)at (6.5,-2) {};

\draw[|-|] (mod_a_b_1) -- (mod_a_b_2);
\draw[|-|] (mod_a_y_1) -- (mod_a_y_2);
\end{tikzpicture}
\end{center}

\noindent Now assume that $\imath$ is a neat embedding of $\Cat_{{\rm Ex}}$ relative ${\bf d} = (d_0,d_1) = (1,1)$. The cells needed to construct the CW complex $|\Cat_{{\rm Ex}}|$ are:
\begin{align*}
& \mathcal{C}(a) = [0,R] \times [-R,R] \times [0,R] \times [-R,R] \\
& \mathcal{C}(x) = [0,R] \times [-R,R] \times \{0\} \times [-\varepsilon, \varepsilon] = \mathcal{C}(c) \\
& \mathcal{C}(y) = \{0\} \times [-\varepsilon, \varepsilon] \times \{0\} \times [-\varepsilon, \varepsilon] = \mathcal{C}(b) {\rm .} 
\end{align*}
Each of these cells can be considered as a subset of $\E = \R_+ \times \R \times \R_+ \times \R$. Moreover,
\[
\partial \E = \big( \{0\} \times \R \times \R_+ \times \R \big) \bigcup \big( \R_+ \times \R \times \{ 0 \} \times \R \big)
\]
is a $3$-dimensional $\langle 2 \rangle$-manifold, which can be illustrated by flattening out the corner-plane $\{ 0 \} \times \R \times \{ 0 \} \times \R$ to give a homeomorphism $\partial \E \cong \R^3$ (c.f. \cite[Figure 3.3]{LipSarKhov}). Under this homeomorphism, the boundary $\partial \mathcal{C}(a)$ has subsets identified with certain moduli spaces that are depicted in Figure~\ref{fig:cancel_pic1}.
\begin{figure}[ht!]
  \centering
  \includegraphics[width=.6\linewidth , trim = 0 0cm 0 0, clip=true]{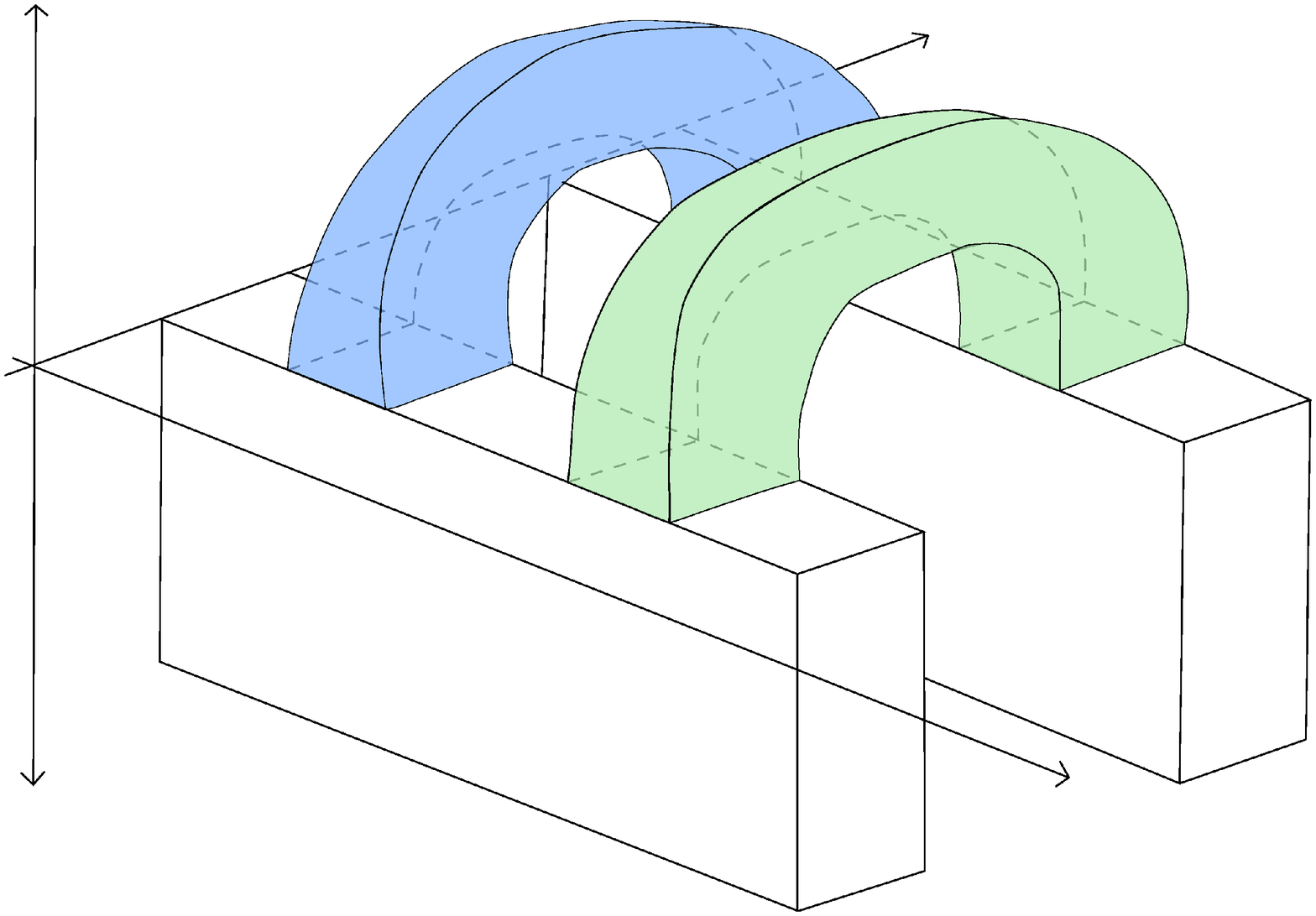}
  \caption{Example: $\Cat_{{\rm Ex}}$. Identifications of cells in $\partial \mathcal{C}(a)$.}
  \label{fig:cancel_pic1}
\end{figure}

In blue is the embedding $\Gamma_{a,b}$ and in green the embedding $\Gamma^\prime_{a,y}$ of $\mathcal{M}(a,b) \times \mathcal{C}(b)$ and $\mathcal{M}(a,y) \times \mathcal{C}(y)$, respectively. The rightmost (white) cube is $\mathcal{C}_c(a) \cong \mathcal{C}(c)$ and the leftmost (white) cube is $\mathcal{C}_x(a) \cong \mathcal{C}(x)$. The cell $\mathcal{C}(x)$ has parts of its boundary identified with both $\mathcal{C}_y(x)$ and $\mathcal{C}_b(x)$, and can be depicted on its own as in Figure \ref{fig:boundary_x}.
\begin{figure}[ht!]
  \centering
      \includegraphics[width=0.5\textwidth , trim = 0 7.5cm 0 7.5cm, clip = true]{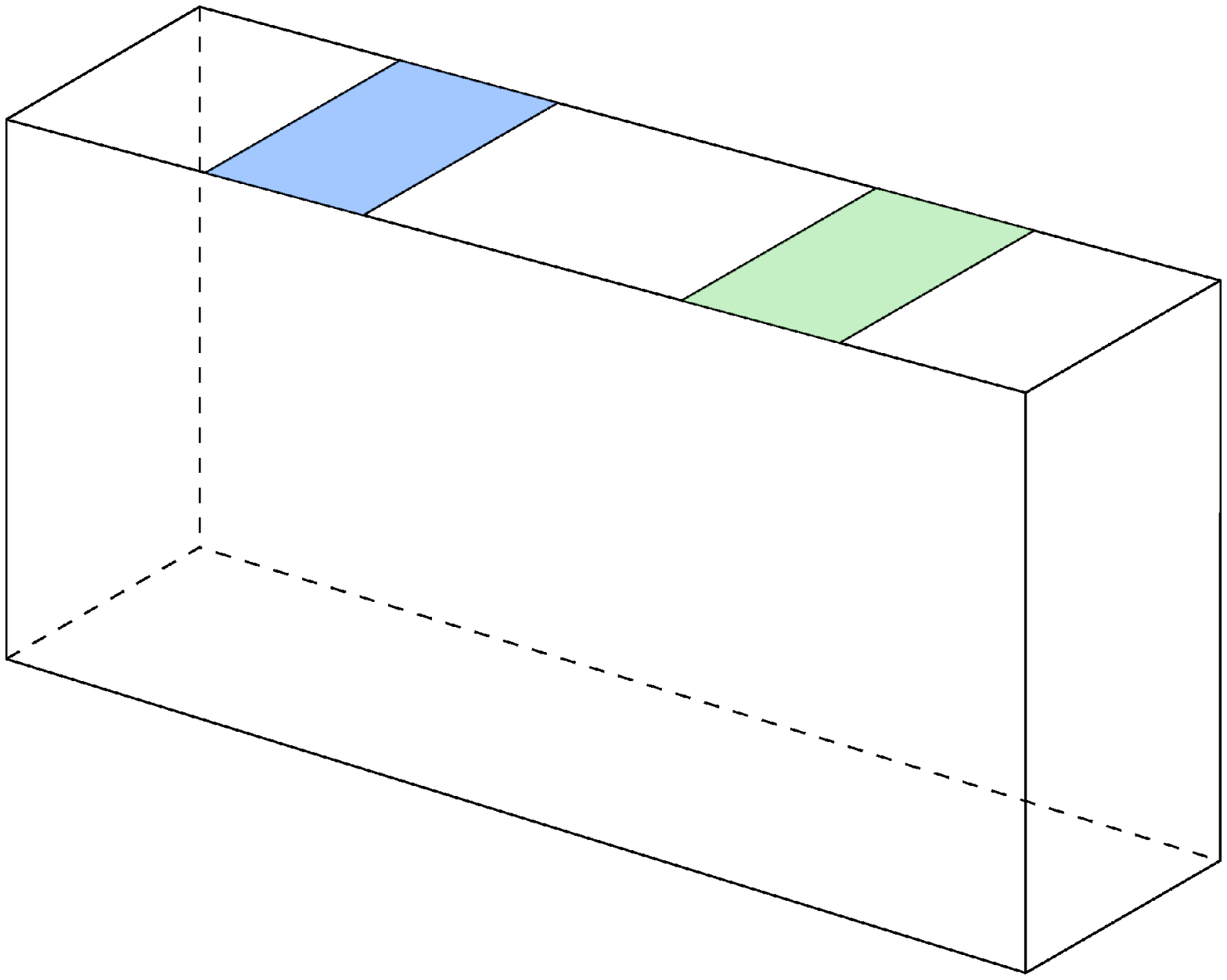}
  \caption{Example: $\Cat_{{\rm Ex}}$. The cell $\mathcal{C}(x)$.}
  \label{fig:boundary_x}
\end{figure}

The cell $\mathcal{C}(y)$ is green and the cell $\mathcal{C}(b)$ is blue. Further, the deformation $\Psi_t$ on $\mathcal{C}(x)$ (which is piecewise smooth on faces) sends $\partial \mathcal{C}(x) \setminus \mathcal{C}_y(x)$ through $\mathcal{C}(x)$ to $\mathcal{C}_y(x)$. This results in an embedding $\Gamma_{x,y}^{-1} \circ \Psi_1 |_{\mathcal{C}_b(x)} \circ \Gamma_{x,b}$ of $\mathcal{M}(x,b) \times \mathcal{C}(b)$ in the cell $\mathcal{C}_y(x) \cong \mathcal{C}(y)$. This embedding is highlighted in blue in Figure \ref{fig:collapsed_to_y}, where the images of each face are outlined.
\begin{figure}[ht!]
  \centering
      \includegraphics[width=0.5\textwidth , trim = 0 7.5cm 0 7.5cm, clip = true]{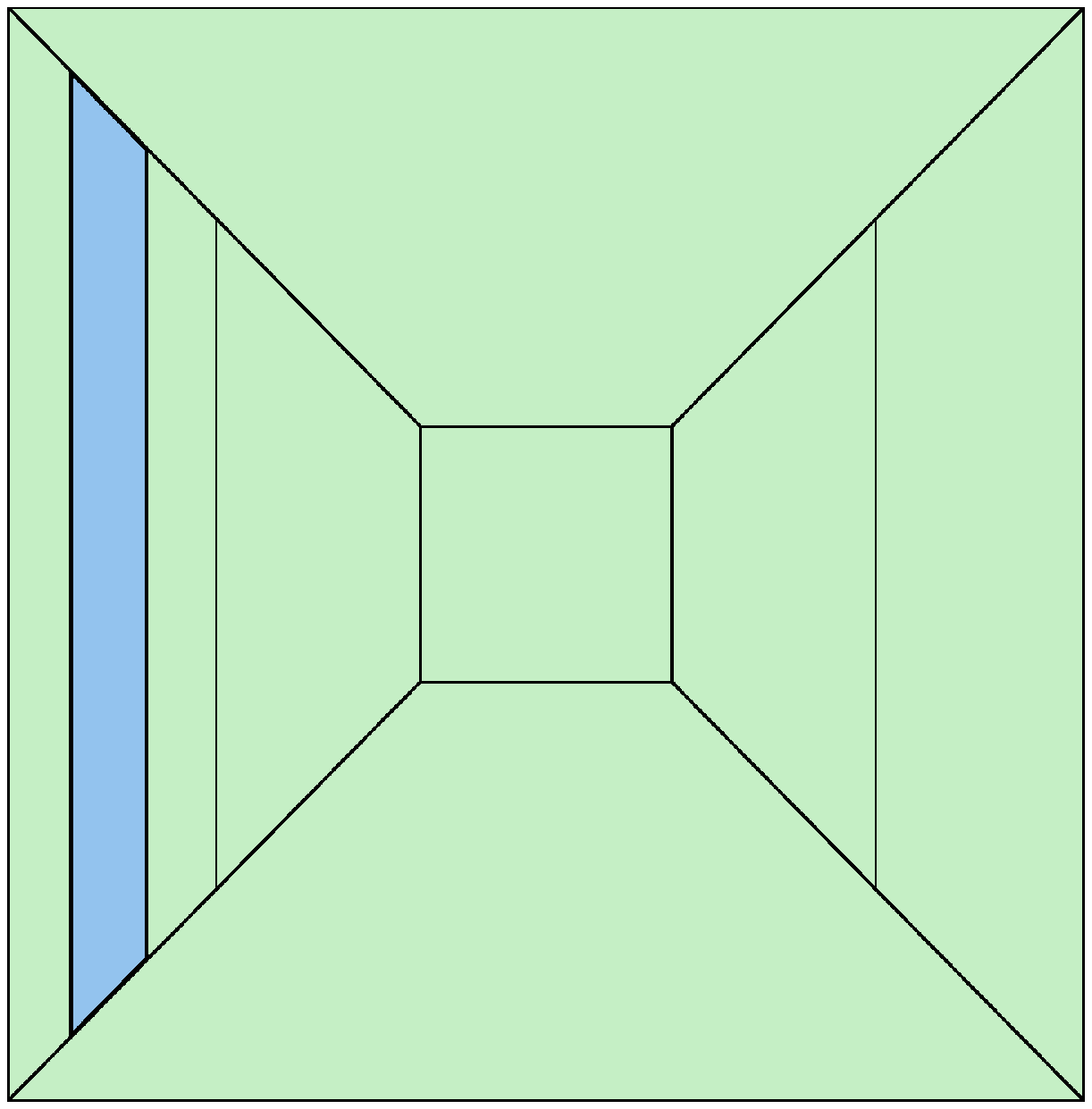}
  \caption{Example: $\Cat_{{\rm Ex}}$. The result of collapsing $\mathcal{C}(x)$ using $\Psi_1$.}
  \label{fig:collapsed_to_y}
\end{figure}

Now recall that the embedding $\Gamma_{x,b \times a,y}^\prime$ is defined in Equation \ref{eq:product_embedding} as the embedding $\Gamma_{a,y}$ of $\mathcal{M}(a,y) \times \mathcal{C}(y)$ with $\mathcal{C}_b(x)$ embedded into $\mathcal{C}(y)$ as above. This embedding, together with the embedding $\Gamma_{a,b}$ is depicted in Figure \ref{fig:cancel_pic2}.
\begin{figure}[ht!]
  \centering
  \includegraphics[width=.6\linewidth , trim = 0 0cm 0 0, clip=true]{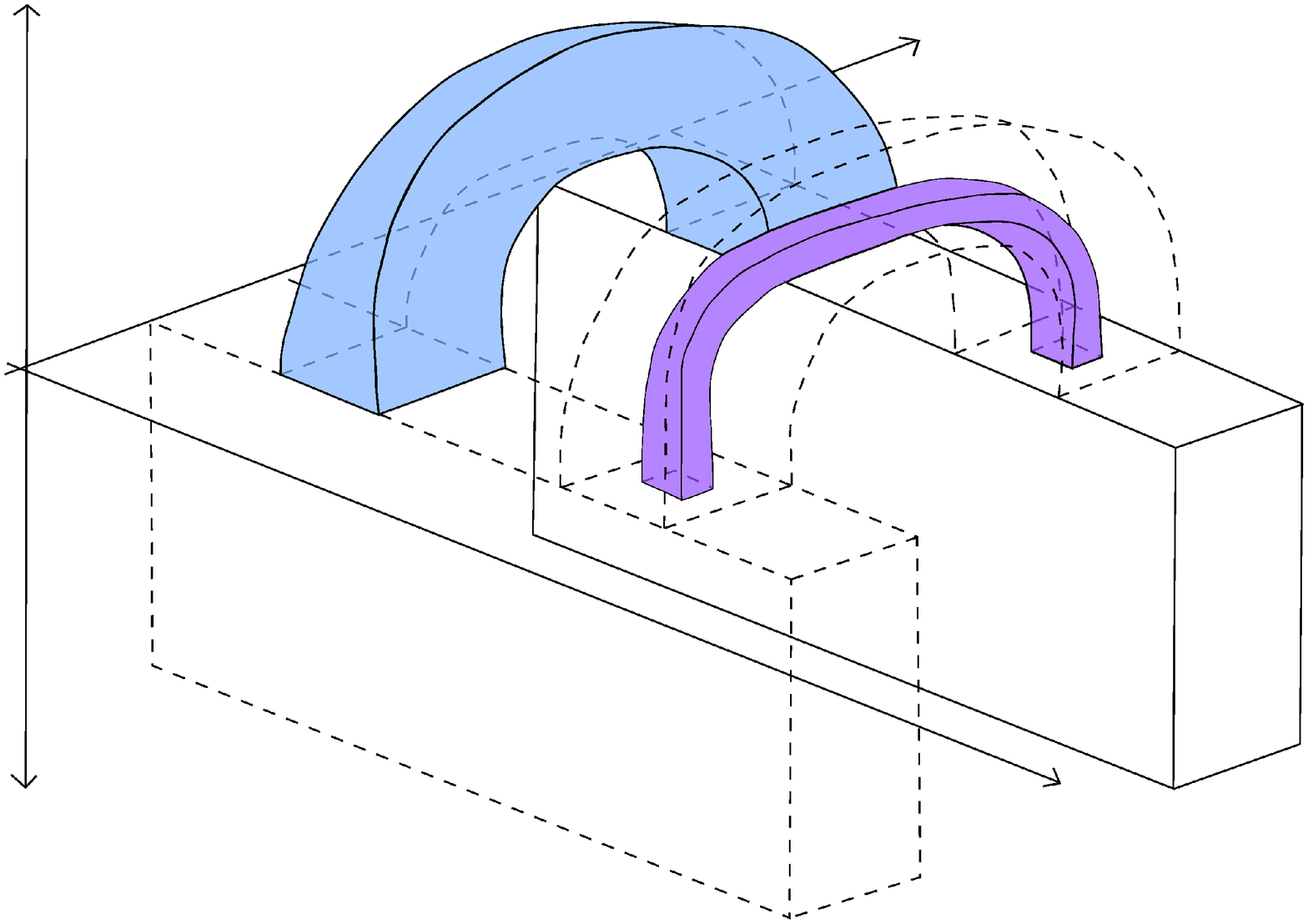}
  \caption{Example: $\Cat_{{\rm Ex}}$. The embedding $\Gamma_{a,b}$ (in blue) and $\Gamma^\prime_{x,b \times a,y}$ (in purple).}
  \label{fig:cancel_pic2}
\end{figure}

\noindent The cells $\mathcal{C}_x(a)$ and $\mathcal{C}_y(a)$ are indicated by dashed lines since they correspond to the objects that are being cancelled. The embedding $\Gamma_{a,b}$ is highlighted blue as before, and the embedding $\Gamma_{x,b \times a,y}^\prime$ is highlighted purple. Notice that the two framed intervals do not agree on their boundaries corresponding to
\[
\mathcal{M}(x,b) \times \mathcal{M}(a,x) {\rm ~ and ~ } \mathcal{M}(x,b) \times \big( \mathcal{M}(x,y) \times \mathcal{M}(a,x) \big) {\rm .}
\]
This is the purpose of the alteration of $\tilde{\Gamma}_{x,b \times a,y}$ in the proof of Lemma \ref{embedding1}. The embedding $\Gamma_{x,b \times a,y}$ is defined in Equation \ref{eq:product_embedding_with_gluing} by altering the embedding $\Gamma_{x,y \times a,y}^\prime$ in a collar neighbourhood of $\mathcal{M}(x,b) \times \mathcal{M}(a,y) \times \mathcal{C}(b)$. The alteration uses the deformation $\Psi_t$ and takes place inside $\mathcal{C}_x(a)$; it is highlighted red in Figure \ref{fig:cancel_pic3}.
\begin{figure}[ht!]
  \centering
  \includegraphics[width=.6\linewidth , trim = 0 0cm 0 0, clip=true]{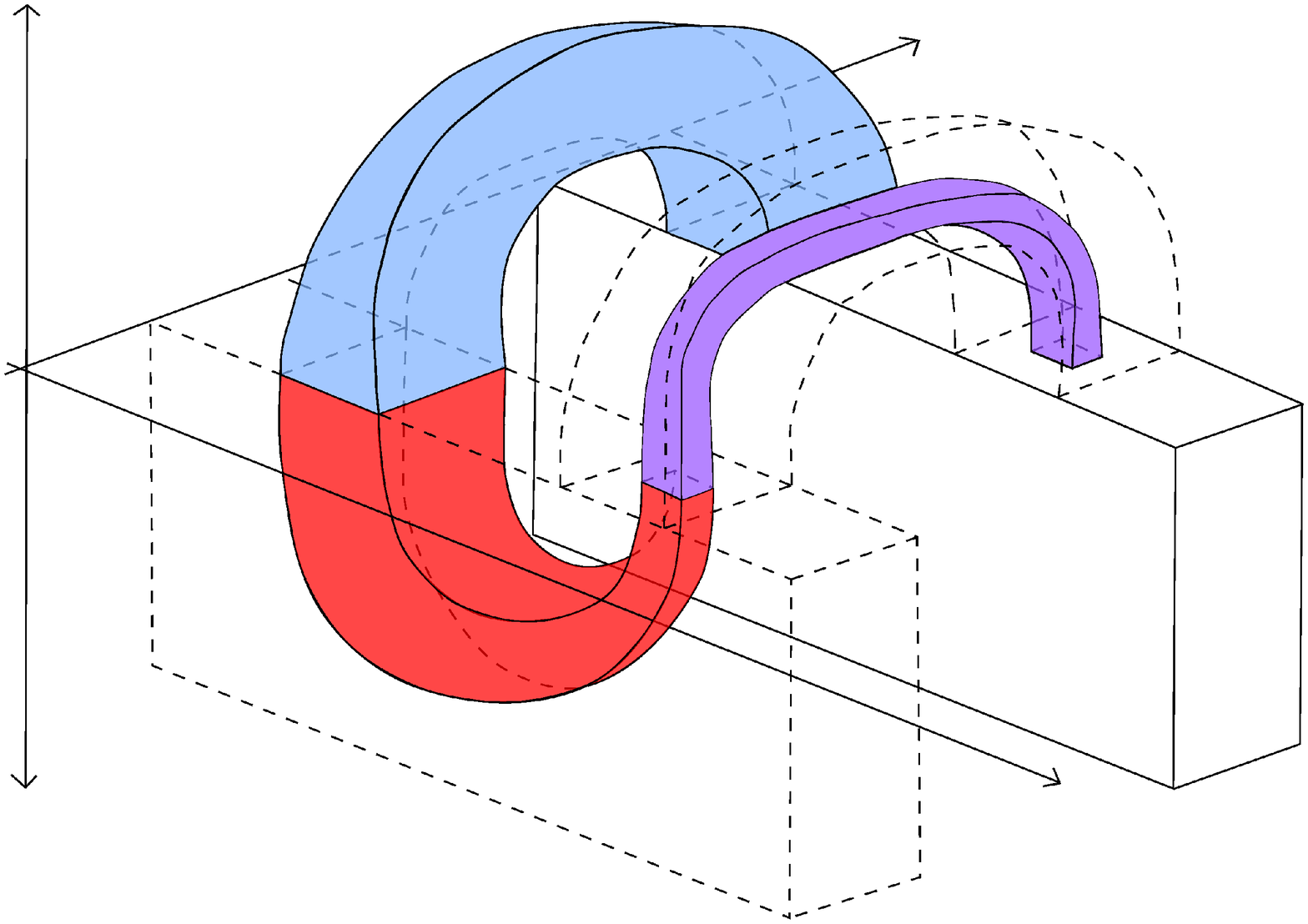}
  \caption{Example: $\Cat_{{\rm Ex}}$. The embeddings $\Gamma_{a,b}$ and $\Gamma_{x,b \times a,y}$.}
  \label{fig:cancel_pic3}
\end{figure}

In this figure, the embedding $\Gamma_{x,b \times a,y}$ of $\mathcal{M}(x,b) \times \mathcal{M}(a,y) \times \mathcal{C}(b)$ is depicted as the concatenation of both the red and purple framed embedded intervals.
\end{example}

\begin{lemma} \label{embedding2}
There is an embedding $\Gamma_{\bar{a}, \bar{b}} : \mathcal{M}(\bar{a},\bar{b}) \times \mathcal{C}(b) \rightarrow \partial \mathcal{C}(a)$.
\end{lemma}

\begin{proof}
Since $\mathcal{M}(\bar{a},\bar{b}) = \mathcal{M}(a,b)\cup_f \mathcal{M}(x,b)\times \mathcal{M}(a,y)$ (Definition \ref{cancelledcat}), we may use $\Gamma_{a,b}$ and $\Gamma_{x,b \times a,y}$ to embed $\mathcal{M}(\bar{a},\bar{b})$ into $\partial \mathcal{C}(a)$. The fact that these embeddings agree on the gluing $f$ given in Definition \ref{cancelledcat} was shown in Lemma \ref{embedding1}.
\end{proof}

\begin{example} \label{example_cancel_2}
Consider again the framed flow category $\Cat_{{\rm Ex}}$ from Example \ref{example_cancel_1}. The embeddings $\Gamma_{a,b}$ and $\Gamma_{x,b \times a,y}$ are depicted in Figure \ref{fig:cancel_pic3}.  The embedding $\Gamma_{x,b \times a,y}$ agrees with $\Gamma_{a,b}$ on
\[
( \mathcal{M}(x,b) \times \mathcal{M}(a,x) ) \times \mathcal{C}(b) \subset \partial (\mathcal{M}(a,b) \times \mathcal{C}(b)) {\rm .}
\]

\end{example}

Let
\[
\Pi[b:a] : \mathcal{C}(a) \rightarrow \E_{\bf d}[b:a]
\]
denote the projection map, and let
\[
\Xi[b:a] : [-\varepsilon , \varepsilon]^{d_n + \cdots + d_{m-1}} \rightarrow \cell(b)
\]
be the inclusion which takes value $0$ on all omitted coordinates of $\cell(b)$.  Observe that we can recover the embeddings
\[
\imath_{a,b} : \mathcal{M}(a,b) \times [-\varepsilon, \varepsilon]^{d_n + \cdots + d_{m-1}} \rightarrow \E_{\bf d}[b:a]
\]
as the composition
\[
\imath_{a,b} = \Pi[b:a] \circ \Gamma_{a,b} \circ ({\rm id}_{\M(a,b)} \times \Xi[b:a]) {\rm .}
\]
One might hope to be able to define embeddings of the moduli spaces $\mathcal{M}(\bar{a},\bar{b})$ in a similar way.
However, these embeddings would not obviously be \emph{neat embeddings} since boundary points of both $\mathcal{M}(a,b)$ and $\mathcal{M}(x,b) \times \mathcal{M}(a,y)$ that become interior points of $\mathcal{M}(\bar{a},\bar{b})$ would need to be identified with part of the interior of $\partial_n \mathcal{C}(a)$. This is the case in Examples \ref{example_cancel_1} and \ref{example_cancel_2}, where the embedding $\Gamma_{\bar{a}, \bar{b}}$ protrudes into the $1$-face of $\mathcal{C}(a)$ (the red framed interval in Figure \ref{fig:cancel_pic3}).  Instead, we alter these embeddings slightly as outlined in the proof of the following lemma.

\begin{lemma} \label{embedding3}
For each $\bar{a},\bar{b}$ in $ {\rm Ob}(\Cat_H)$, there are framed neat embeddings
\[
\bar{\imath}_{\bar{a},\bar{b}} : \mathcal{M}(\bar{a},\bar{b}) \times [-\varepsilon, \varepsilon]^{d_{n}+ \cdots + d_{m-1}} \rightarrow \E_{\bf d}[b:a]
\]
\end{lemma}

\begin{proof}
We shall give a homotopy of homeomorphisms $h_t : \partial \cell(a) \rightarrow \partial \cell(a)$ for $0 \leq t \leq 1$, which are smooth away from the corners of $\cell(a)$.  The purpose of this homotopy is to move from $h_0 = {\rm id}_{\cell(a)}$ to the map $h_1$ which redefines the cornered structure of $\cell(a)$ in a useful way.  In particular, if one pulls back the face structure of $\partial \cell(a)$ through $h_t$ one will obtain a cornered manifold $\cell(a)^t$ isomorphic to $\cell(a) = \cell(a)^0$.  The homotopy $h_t$ will be chosen so that the embedding $h_1 \circ (\Gamma_{a,b} \cup \Gamma_{x,b \times a,y})$ will be neat.

We shall define the embedding $\bar{\imath}_{\bar{a},\bar{b}}$ as the embedding
\[
\Pi[b:a] \circ h_1 \circ (\Gamma_{a,b} \cup \Gamma_{x,b \times a,y}) \circ(({\rm id}_{\M(a,b)} \cup {\rm id}_{\M(x,b) \times \M(a,y)}) \times \Xi[a:b]) {\rm .}
\]
Let us begin by defining the face structure of $\partial \cell(a)^1$.  It will be enough, for each point of $\partial \cell(a)^1$ to know whether or not it lies in the closure of the $k$-face for each $k < m$.

If $p \in \im(\Gamma_{a,x})$ (respectively $p \in \im(\Gamma_{a,y})$) then, since $\Gamma_{a,x}$ (resp.~$\Gamma_{a,y}$) is injective, we have that $\Gamma_{a,x}^{-1}(p)$ is a well-defined point of $\M(a,x) \times \cell(x)$ (resp.~$\Gamma_{a,y}^{-1}(p)$ is a well-defined point of $\M(a,y) \times \cell(y)$).  Projecting to the second coordinate gives a point $\wt{p} \in \cell(x)$ (resp.~$\wt{p} \in \cell(y)$).

If there is a point $q$ in the $k$-face $\partial_k \cell(x)$ such that $\Psi_t(q) = \wt{p}$ for some $0 \leq t \leq 1$ (resp.~such that $\Psi_1(q) = \wt{p}$) then we have that $p$ is in the $k$-face of $\cell(a)^1$.  For all other $p \in \im(\Gamma_{a,x})$ we have that $p$ is in the $k$-face of $\cell(a)^1$ iff it is in the $k$-face of $\cell(a)^0$.

That the embedding $\Gamma_{a,b} \cup \Gamma_{x,b \times a,y}$ is neat with respect to the face structure of $\cell(a)^1$ is clear from the definition of $\Gamma_{a,b} \cup \Gamma_{x,b \times a,y}$.

Finally we wish to see that we can realize the face structure of $\cell(a)^1$ as the pullback of the face structure of $\cell(a)^0 = \cell(a)$ through $h_1$ for some homotopy $h_s$.  We shall give the corresponding face structures $\cell(a)^s$, from which it will be clear that such an $h_s$ exists.

Let $f : \partial \cell(x) \rightarrow [0,1]$ be an injective height function smooth where it makes sense on the boundaries of $\cell(x)$, such that $f^{-1}(1)$ is a point, $f^{-1}(0) = \cell_y(x)$, $f^{-1}(s)$ is homeomorphic to $\partial \cell_y(x)$ for $0<s<1$, and the intersection of $f^{-1}[0,s]$ with any face of $\cell(x)$ is homeomorphic to a disc.
Now we use the same notation as used above when we were giving the face structure of $\cell(a)^1$.  If there is a point $q$ in the $k$-face $\partial_k \cell(x)$ with $f(q) \leq s$ such that $\Psi_t(q) = \wt{p}$ for some $0 \leq t \leq 1$ (resp.~such that $\Psi_1(q) = \wt{p}$) then we have that $p$ is in the $k$-face of $\cell(a)^s$.  For all other $p \in \im(\Gamma_{a,x})$ we have that $p$ is in the $k$-face of $\cell(a)^s$ iff it is in the $k$-face of $\cell(a)^0$.
\end{proof}

\begin{example}
Consider the recurring example of the framed flow category $\Cat_{\rm Ex}$. The embedding $\Gamma_{\bar{a},\bar{b}}$, as seen in Figure \ref{fig:cancel_pic3}, embeds $\mathcal{M}(\bar{a},\bar{b}) \times \mathcal{C}(b)$ in $\partial \mathcal{C}(a)$. However, the embedding protrudes into the $1$-face of $\partial \mathcal{C}(a)$.  Redefining the face structure of $\cell(a)$ as in Lemma \ref{embedding3} will push this interval into the $0$-face only, and the result is illustrated in Figure \ref{fig:cancel_pic4}.

\begin{figure}[ht!]
  \centering
  \includegraphics[width=.6\linewidth , trim = 0 0cm 0 0, clip=true]{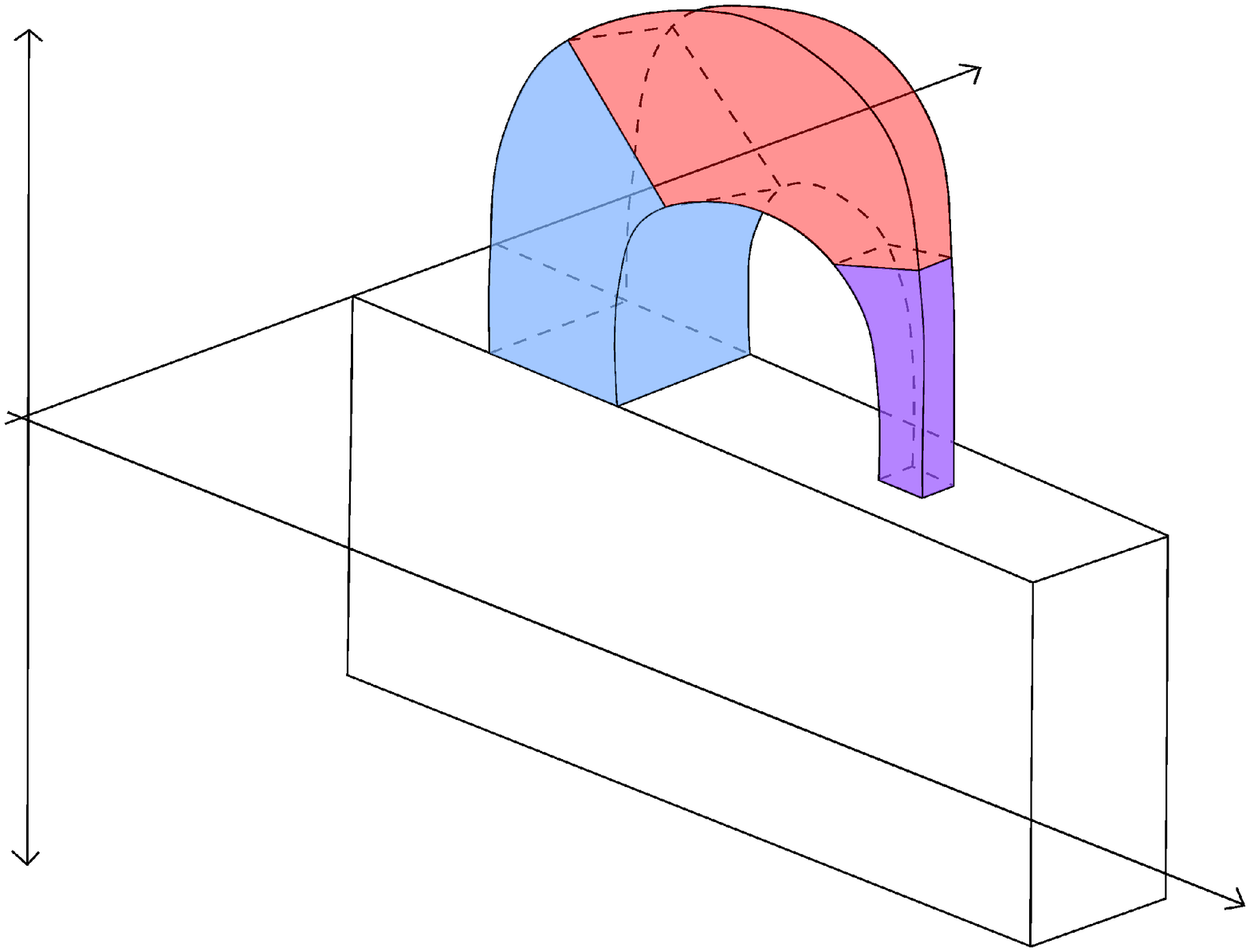}
  \caption{Example: $\Cat_{{\rm Ex}}$. The result of redefining the face structure of $\mathcal{C}(a)$.}
  \label{fig:cancel_pic4}
\end{figure}
\end{example}

So, if $(\Cat, \imath , \varphi)$ is a framed flow category containing two objects $x$ and $y$ with $\mathcal{M}(x,y) = \ast$, then there is a framed flow category $(\Cat_H , \bar{\imath} , \bar{\varphi})$ whose objects and moduli spaces are given in Definition \ref{cancelledcat}.  We are now in a position to state and prove the main theorem of this subsection.

\begin{theorem} \label{cancelthm}
Let $(\Cat, \imath , \varphi)$ be a framed flow category containing two objects $x$ and $y$ with $\mathcal{M}(x,y) = \ast$. The realisation $| \Cat |$ is stably homotopy equivalent to the realisation $| \Cat_H |$ of the cancelled category.
\end{theorem}

\begin{proof}
The realisation $|\Cat_H|$ is built up inductively from cells $\mathcal{C}(\bar{a})$ for objects $\bar{a}$ of $\Cat_H$ with increasing indices as prescribed by the Cohen-Jones-Segal construction. The way in which these cells are attached corresponds to the framed neat embeddings $\bar{\imath}_{\bar{a},\bar{b}}$ defined in Lemma \ref{embedding3}.

Firstly observe that the pair of skeleta $|\Cat|^{(i-2)}$ and $|\Cat_H|^{(i-2)}$ are identical.  Further attaching all cells $\mathcal{C}(\bar{a})$ for objects $\bar{a}$ with index equal to $i-1$, we have
\[
|\Cat|^{(i-1)} = |\Cat_H|^{(i-1)} \cup \mathcal{C}(y) \hookleftarrow |\Cat_H|^{(i-1)} {\rm .}
\] 

For objects $\bar{a}$ of $\Cat_H$ of index $m \geq |x|=i$ the cells $\mathcal{C}(a)$ are attached to $|\Cat|^{(m-1)}$ inductively via the Thom construction in the usual way, corresponding to the identifications
\[
\Gamma_{a,b} : \mathcal{M}(a,b) \times \mathcal{C}(b) \rightarrow \partial \mathcal{C}(a) {\rm .}
\]

The cells $\mathcal{C}(\bar{a})$ are attached to $|\Cat_H|^{(m-1)}$ corresponding to the identifications that come from the embeddings $\bar{\imath}_{\bar{a},\bar{b}}$ constructed in this subsection. To show that the CW complexes produced as a result of these methods are homotopy equivalent, consider the CW complex $X$ homotopy equivalent to $|\cC|$, and that is obtained from $|\Cat|$ by collapsing the cell $\mathcal{C}(x)$ via $\Psi_t$.

\[
X = |\cC| / ( \Psi_s(x) \sim \Psi_1(x), x \in (\partial \cell(x) \setminus {\rm int}(\cell_y (x))),0 \leq s \leq 1) {\rm .}
\]

Collapsing $\mathcal{C}(x)$ in $|\Cat|$ gives a description of $X$ as a CW complex with two fewer cells than $|\cC|$, and where the attaching maps of $X$ are given in terms of the attaching maps for $|\cC|$ and the map $\Psi_1$.

Indeed, the cell $\mathcal{C}(a)$ attaches to $X^{(m-1)}$ by applying the Thom construction to the embeddings $\Gamma_{\bar{a}, \bar{b}} = \Gamma_{a,b} \cup \Gamma_{x,b \times a,y}$ defined in Lemma \ref{embedding2}.

In Lemma \ref{embedding3} the neat embeddings $\bar{\imath}_{\bar{a},\bar{b}}$ of the framed flow category $\Cat_H$ are defined via perturbations of the embeddings $\Gamma_{\bar{a},\bar{b}}$ inside $\partial \mathcal{C}(a)$.  Finally, the Isotopy Extension Theorem ensures that there is a global isotopy in the Euclidean space $\E_{\bf d} [b:a]$ between the two embeddings which extends the isotopy in $\partial \mathcal{C}(a)$.  This gives a homotopy equivalence $|\Cat| \simeq X \simeq |\Cat_H|$ of CW complexes.
\end{proof}

\section{The framings of $1$-dimensional moduli spaces}
\label{sec:framings}

\subsection{Some simple stable homotopy types}
\label{subsec:chang_easy_examples}
In this subsection we consider three non-trivial stable homotopy types, the last two of which were demonstrated to be wedge summands of the Lipshitz-Sarkar space $\X^{\Kh}(K)$ for particular knots $K$ in \cite{LipSarSq}.  These spaces have low-dimensional representatives as based CW complexes $\CP^2$, $\RP^4 / \RP^1$, and $\RP^5 / \RP^2$.  We give framed flow categories which give rise to each of these spaces.  These framed flow categories are the simplest such possible in the sense that they have the minimal number of objects required to give the cohomology of the associated spaces, and their $0$-dimensional moduli spaces contain no cancelling pairs.

Let $m \geq 3$ and let $e^{m}$, $e^{m+1}$, and $e^{m+2}$ be three cells for which the dimension is indicated by the superscript, and let $b$ be a basepoint.

The group
\[ [\partial e^{m+2}, \{b\} \cup e^m] \cong [ S^{m+1}, S^m ] \cong \pi_{m+1}(S^m) \]
\noindent is isomorphic to $\Z/2$.

We can give representatives for the two elements of this group in the following way.  Let $K = S^1 \subset S^{m+1}$ be an embedded circle.  The normal bundle to the circle is then a trivial $D^m$-bundle over $S^1$.  There are two ways to frame this normal bundle up to homotopy equivalence (corresponding to the fact that $\pi_1(\SO(m))$ is a 2-element group).  Each framing then determines a map $S^{m+1} \rightarrow S^m$ by the Thom construction, and exactly one of the framings will induce the non-trivial element of $\pi_{m+1}(S^m)$.

In the Cohen-Jones-Segal construction, a framed flow category $\cC$ consisting of two objects $p^i$ and $p^{i+2}$ gives rise to the stable homotopy type of a CW complex consisting of a basepoint $b$ and one cell of each degree $m$ and $m+2$ for some $m >> 0$.  (This CW complex would then undergo (de)suspension so that its reduced cohomology would be supported in degrees $i$ and $i+2$).  Given that the reduced cohomology of $\Sigma^{i-2}\CP^2$ is 2-dimensional, this is the framed flow category with the smallest number of objects that might give rise to the stable homotopy type of $\Sigma^{i-2}\CP^2$.

The attaching map $S^{m+1} \cong \partial e^{m+2} \rightarrow \{b\} \cup e^m \cong S^m$ is given by applying the Thom construction to a disjoint union of framed embedded circles lying in $\partial e^{m+2}$.  These framed embedded circles are exactly the framed embedded moduli space $\M(p^{i+2}, p^i)$.  

In terms of the Lipshitz-Sarkar framing conventions, if $\M(p^{i+2}, p^i)$ consists of $r$ circle components framed $0$ and $s$ circle components framed $1$, then the induced attaching map $\partial e^{m+2} \rightarrow \{b\} \cup e^m$ is the wedge sum of $r$ homotopically non-trivial maps and $s$ homotopically trivial maps $S^{m+1} \rightarrow S^m$.

If the attaching map is non-trivial (equivalently if $r$ is odd), then the resulting stable homotopy type is that of $\Sigma^{i-2} \CP^2$ (see Example 5.6 in \cite{JLS}).  If the attaching map is trivial (equivalently if $r$ is even) then the resulting stable homotopy type is of course that of a wedge of Moore spaces $\Sigma^{i-2} (S^2 \vee S^4)$.

A similar analysis applies to the cases $\Sigma^{i-2}(\RP^4 / \RP^1)$ and $\Sigma^{i-3}(\RP^5 / \RP^2)$ (see Example 5.5 in \cite{JLS} for descriptions which are a single Whitney trick away from the following description).  In the case of $\Sigma^{i-2}(\RP^4 / \RP^1)$ we take a framed flow category $\cC$ consisting of three objects $p^i$, $p^{i+1}$, and $p^{i+2}$ (which again is the smallest number of objects that could give rise to the cohomology of $\Sigma^{i-2}(\RP^4 / \RP^1)$.  The resulting cell complex is built from the basepoint $b$ and three cells $e^m$, $e^{m+1}$, $e^{m+2}$.  

Since we know the differentials in the cochain complex, we know that the signed count of the number of points in $\M(p^{i+2}, p^{i+1})$ must be $2$ and the signed count of the number of points in $\M(p^{i+1}, p)$ must be $0$.  By the Whitney trick then we may assume that the signed count agrees with the absolute count.  This forces $\M(p^{i+2}, p^i)$ to be a boundaryless framed 1-manifold, in other words a disjoint union of framed 
embedded circles lying in $\partial e^{m+2}$.  As before, if the number of $0$-framed circles is odd then we obtain $\Sigma^{i-2}(\RP^4 / \RP^1)$, but if the number of $0$-framed circles is even then we obtain a wedge of Moore spaces $\Sigma^i (\RP^2 \vee S^0)$.

The case of $\Sigma^{i-3}(\RP^5 / \RP^2)$ arises from a flow category of three objects $p^i$, $p^{i+1}$, and $p^{i+2}$ in which $\M(p^{i+2}, p^{i+1})$ is empty, $\M(p^{i+1}, p^i)$ consists of two points with the same framing, and $\M(p^{i+2}, p^i)$ contains an odd number of $0$-framed circles (when the number of $0$-framed circles is even then we obtain a wedge of Moore spaces $\Sigma^{i-1}(S^3 \vee \RP^2)$).

\subsection{Framing conventions for 1-dimensional moduli spaces}
\label{subsec:stsq_recipe}
The formulae in \cite{LipSarSq} and \cite{JLS} require a way to encode the framing information of the $1$-dimensional moduli spaces. Particularly for intervals one has to be quite specific, and here we give a summary of \cite[\S 5]{JLS}.

Recall that the $1$-dimensional moduli spaces are framedly embedded into some $\R^{d_i}\times [0,\infty) \times \R^{d_{i+1}}$. Circle components stay away from the boundary of this Euclidean half-space. The framing together with a tangent direction of the circle represents an element of $H_1(\SO(d_i+d_{i+1}+1))$, a group which is $\Z/2$ provided $d_i+d_{i+1}\geq 2$, as we shall assume. We then assign this element as the framing information of the circle.

\begin{remark}
 As the tangential direction is taken into account, we get the following curious side effect: if a circle is trivially embedded into $\R^3$ with trivial framing, the resulting element of $H_1(\SO(3))$ is non-trivial. This is because the tangent direction with one normal direction performs one rotation in $\SO(2)$ with the second normal direction being constant.
Similarly, if we take a circle as the fibre of the Hopf fibration with framing coming from the pull-back of a framing at a point, the resulting element of $H_1(\SO(3))$ is trivial.
\end{remark}

For interval components $J$ we also obtain an element $fr(J)\in \Z/2$ which depends on a coherent system of paths joining the different framings of endpoints.

We assume that each $0$-dimensional moduli space is framed using the standard framing, that is, the positive framings are framed via $(e_1,\ldots,e_{d_i})$ in $\R^{d_i}$ (that is, using the standard basis), and negative framings are framed via $(-e_1,e_2,\ldots,e_{d_i})$ in $\R^{d_i}$.

Interval components of $1$-dimensional moduli spaces are therefore framed so that the boundary, when embedded in $\R^{d_i}\times\{0\}\times \R^{d_{i+1}}$ is framed via
\begin{align*}
&(e_1,\ldots,e_{d_i},e_{d_i+1},\ldots,e_{e_i+e_{i+1}}), (-e_1,e_2,\ldots,e_{d_i},e_{d_i+1},\ldots,e_{e_i+e_{i+1}}),\\ &(e_1,\ldots,e_{d_i},-e_{d_i+1},e_{d_i+2}\ldots,e_{e_i+e_{i+1}}), (-e_1,e_2,\ldots,e_{d_i},-e_{d_i+1},e_{d_i+2}\ldots,e_{e_i+e_{i+1}})
\end{align*}
which we denote by $\PP$, $\MP$, $\PM$, $\MM$, respectively.

\begin{definition}
A \em coherent system of paths joining $\PP$, $\MP$, $\PM$, $\MM$ \em is a choice of path $\overline{\varphi_1\varphi_2}$ in $\SO(m+1)$ from $\varphi_1$ to $\varphi_2$ for each pair of frames $\varphi_1,\varphi_2\in \{\PP,\MP,\PM,\MM\}$ satisfying the following cocycle conditions:
\begin{enumerate}
	\item For all $\varphi\in \{\PP,\MP,\PM,\MM\}$ the loop $\overline{\varphi\varphi}$ is null-homotopic;
	\item For all $\varphi_1,\varphi_2,\varphi_3\in \{\PP,\MP,\PM,\MM\}$ the path $\overline{\varphi_1\varphi_2}\cdot \overline{\varphi_2\varphi_3}$ is homotopic to $\overline{\varphi_1,\varphi_3}$ relative to the endpoints.
\end{enumerate}
\end{definition}

Coherent systems of paths exist, we will use the one described in \cite[Lm.3.1]{LipSarSq}. To describe it, we will refer to the first coordinate of $\R^{d_i}$ as the $e_1$-coordinate, to the first coordinate of $\R^{d_{i+1}}$ as the $e_2$ coordinate, and to the coordinate of $[0,\infty)$ as the $\bar{e}$-coordinate.

For $\varphi_1,\varphi_2\in \{\PP,\MP,\PM,\MM\}$ define $\overline{\varphi_1\varphi_2}$ as follows:
\begin{enumerate}
	\renewcommand\theenumi{\roman{enumi}}
	\item $\overline{\PP\MP}, \overline{\MP\PP}, \overline{\PM\MM}, \overline{\MM\PM}$: Rotate $180^\circ$ around the $e_2$-axis, such that the first vector equals $\bar{e}$ halfway through.
	\item $\overline{\PP\PM}, \overline{\PM\PP}$: Rotate $180^\circ$ around the $e_1$-axis, such that the second vector equals $\bar{e}$ halfway through.
	\item $\overline{\MP\MM}, \overline{\MM\MP}$: Rotate $180^\circ$ around the $e_1$-axis, such that the second vector equals $-\bar{e}$ halfway through.
	\item $\overline{\PP\MM}, \overline{\MM\PP}, \overline{\MP\PM}, \overline{\PM\MP}$: Rotate $180^\circ$ around the $\bar{e}$-axis, such that the second vector equals $-e_1$ halfway through.
\end{enumerate}

The framing of a $1$-dimensional moduli space $\M(z,x)$ is now encoded in a function
\[
 fr\colon \pi_0(\M(z,x))\to \Z/2{\rm ,}
\]
\noindent already given on circle components, and which is $0$ or $1$ on interval components depending as the interval is coherently framed or not.

\subsection{Gluing formulae for the Whitney trick}
\label{subsec:whitney_glue}
If $\cC$ is a framed flow category containing two objects $x,y$ with $\M(x,y)=\{P,M\}$ where $P$ is framed positively and $M$ negatively, then we can form the framed flow category $\cC_W$.  We have seen that $|\cC|$ and $|\cC_W|$ are stably homotopy equivalent.  For the purposes of computation, in this subsection we determine how the framings of the $1$-dimensional moduli spaces of $\cC_W$ can be determined from those of $\cC$.

\begin{figure}
	\centerline{
		{
			\psfrag{ldots}{$\ldots$}
			\psfrag{P}{$P$}
			\psfrag{M}{$M$}
			\psfrag{e1}{$e_1$}
			\psfrag{e2}{$e_2$}
			\psfrag{e}{$e$}
			\psfrag{p}{$p$}
			\includegraphics[height=3in,width=3.5in]{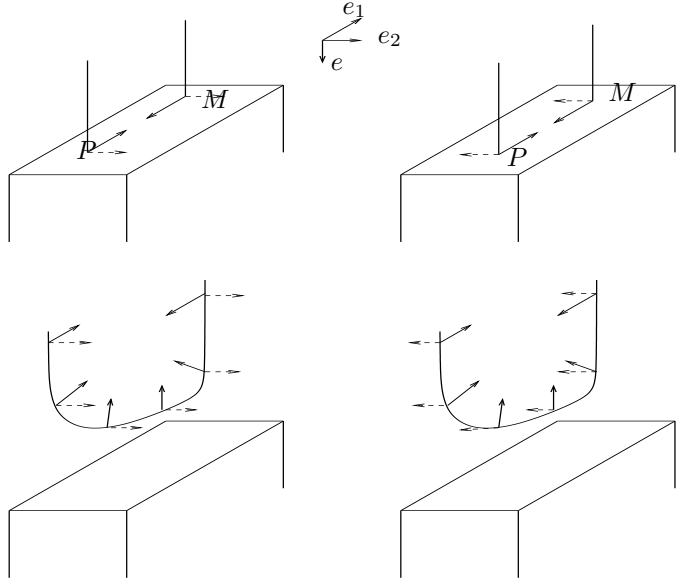}
		}}
		\caption{Here we show how to compute the framings of the glued intervals in $\M(\bar{a}, \bar{y})$ in the case that $a >_1 x$.  The top left big cuboid represents a positively framed point $p$ in $\M(a,x)$, while the top right represents a negatively framed point $p$ in $\M(a,x)$.  Also shown in top diagrams are the ends of the intervals in $\M(a,y)$ which have endpoints $(P,p)$ and $(M,p)$.  The horizontal plane is, as usual, a corner of Euclidean space.  The bottom left and bottom right diagrams show the framing on part of the new $1$-dimensional moduli space in $\M(\bar{a}, \bar{y})$ created by gluing the two ends together.}
		\label{fig:whitney_gluing_first}
	\end{figure}

\begin{proposition}
\label{prop:whitney_glue_formula_first}
Let $\cC$ and $\cC_W$ be as above, and suppose that $a \in \Ob(\cC)$ is such that $a >_1 x$.  Now, $\M(\bar{a},\bar{y})$ is obtained from $\M(a,y)$ by gluing $n$ pairs of endpoints of intervals of $\M(a,y)$ together where $n$ is the number of points in $\M(a,x)$.  The framing of the moduli spaces of $\M(\bar{a}, \bar{y})$ can then be calculated by summing the frames of the contributing moduli spaces of $\M(a,y)$, and adding $1$ for each point of $\M(a,x)$ at which there is a gluing.
\end{proposition}

\begin{proof}
The situation is illustrated in Figure \ref{fig:whitney_gluing_first}, where we consider a particular point of $\M(a,x)$ (which is either positively or negatively framed) and the gluing that corresponds to it.  We can see (in both the positively and the negative framed case) that the framing of the gluing region gives a path in $\SO(3)$ which is a rotation around the $e_2$ axis by $180^\circ$ in which halfway through the first vector is pointing in the $-e$ direction.  This corresponds to a non-standard frame path by the classification of Lipshitz-Sarkar, and hence we get a contribution of $+1$ for each gluing.
\end{proof}

\begin{proposition}
\label{prop:whitney_glue_formula_second}
Let $\cC$ and $\cC_W$ be as above, and suppose that $b \in \Ob(\cC)$ is such that $y >_1 b$.  Now, $\M(\bar{x},\bar{b})$ is obtained from $\M(x,b)$ by gluing $n$ pairs of endpoints of intervals of $\M(x,b)$ together where $n$ is the number of points in $\M(y,b)$.  The framing of the moduli spaces of $\M(\bar{x}, \bar{b})$ can then be calculated by summing the frames of the contributing moduli spaces of $\M(x,b)$, and adding $1$ for each \emph{positively framed} point of $\M(y,b)$ at which there is a gluing.
\end{proposition}

\begin{figure}
	\centerline{
		{
			\psfrag{ldots}{$\ldots$}
			\psfrag{P}{$P$}
			\psfrag{M}{$M$}
			\psfrag{e1}{$e_1$}
			\psfrag{e2}{$e_2$}
			\psfrag{e}{$e$}
			\psfrag{-1}{$-1$}
			\includegraphics[height=2.6in,width=3.3in]{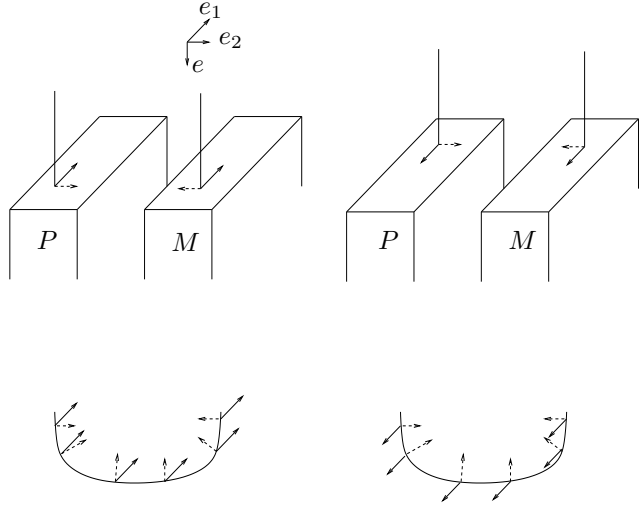} 
		}}
		\caption{Here we show how to compute the framings of the glued intervals in $\M(\bar{x}, \bar{b})$ in the case that $y >_1 b$.  Each big cuboid represents either the framed point $P$ or the framed point $M$ in $\M(x,y)$.  In the top left diagram we consider a positively framed point $p \in \M(y,b)$ and we show the ends of the intervals in $\M(x,b)$ which have endpoints $(p,P)$ and $(p,M)$.  The top right is similar but now $p$ is negatively framed.  The bottom left and bottom right diagrams show the framing on part of the new $1$-dimensional moduli space in $\M(\bar{x}, \bar{b})$ created by gluing the two ends together.}
		\label{fig:whitney_gluing_second}
	\end{figure}

\begin{proof}
In Figure \ref{fig:whitney_gluing_second} we consider a particular point $p$ of $\M(y,b)$ and the gluing that corresponds to $p$.  We can see that the framing of the gluing region gives a path in $\SO(3)$ which is a rotation around the $e_1$ axis by $180^\circ$ in which halfway through the second vector is pointing in the $-e$ direction.  This corresponds to a non-standard frame path when $p$ is positively framed, and a standard frame path when $p$ is negatively framed.  Hence we get a contribution of $+1$ for each gluing corresponding to a positively framed point $p$.
\end{proof}

\subsection{Gluing formulae for handle cancellation}
\label{subsec:gluing_formula}
If $\Cat$ is a framed flow category that contains two objects $x,y$ with $\M(x,y)=\{\ast\}$, we can form the framed flow category $\Cat_H$ through handle cancellation.  As in the previous subsection, we now investigate how to determine the framings on the $1$-dimensional moduli spaces of $\cC_H$.  This turns out to be more involved than in the case of the Whitney trick.  Indeed, the first difficulty that arises is that the new points in $0$-dimensional moduli spaces, namely the ones in $\M(x,b)\times \M(a,y)$ where $|a|=|x|=i+1$ and $|b|=|y|=i$, are not necessarily framed with one of the two standard frames, even if the original flow category has all $0$-dimensional moduli spaces framed using standard frames only.

To understand the embedding of $(B,A)\in \M(x,b)\times \M(a,y)$ we need to analyze the map $\Gamma_{x,b\times a,y}$. The relevant coordinates of the cells involved are $[0,R]\times [-R,R]^{d_i}$ for $C(x)$ and $C(a)$, and $\{0\}\times [-\varepsilon,\varepsilon]^{d_i}$ for $C(b)$ and $C(y)$. The remaining coordinates play no role after projection to $\R^{d_i}$. We then have the map $\Gamma_{x,b\times a,y}\colon \M(x,b)\times \M(a,y)\times \{0\}\times [-\varepsilon,\varepsilon]^{d_i} \to \{0\} \times [-R,R]^{d_i}$ which we can identify with the framed embedding into $\R^{d_i}$. This map is given by the compositions of $\imath_{x,b}$ with $\Psi_1$, $\imath_{x,y}^{-1}$ and $\imath_{a,y}$. Here $\Psi_1$ can be viewed as inversion of $\R^{d_i}$ around a small sphere with center $0$. 
Note that this sphere can be thought of as the boundary of the ball around $\imath_{x,y}(\ast)$ given by the framing. The points in $\M(x,b)$ are embedded away from this ball, then they are moved into this ball by $\Psi_1$, and then they are embedded near the points of $\M(a,y)$. So the only time that we may get a non-standard framing is when the map $\Psi_1$ is applied.

However, if we assume that the points of $\M(x,b)$ are embedded into $\R\times \{0\} \subset \R^{d_i}$, then the inversion $\Psi_1$ can be taken only to flip the first coordinate of the framing of these points. That is, standard frames are send to standard frames with the opposite sign.

So before we do the cancellation, we will change the embedding by an isotopy so that all moduli spaces between objects of degree $i$ and $i+1$ are embedded into $\R\times \{0\}\subset \R^{d_i}$ with standard frames.  Of course, this does not change the stable homotopy type of the realization by \cite[\S 3]{LipSarKhov}.

\begin{remark}
\label{remark_embedding}
Different isotopies into $\R\times \{0\}\subset \R^{d_i}$ can lead to different gluing formulas. We will therefore make a particular choice for the embeddings. In particular we assume that $\M(x,y)$ is sent to $0$ and all other points of moduli spaces $\M(a,b)$ with $|a|=i+1=|b|+1$ are embedded into $(-\infty,0)\times \{0\}$.
\end{remark}

If we denote the framing of $A\in \M(a,b)$ by $\varepsilon_A\in \{0, 1\}=\Z/2$, we can now easily see that $\varepsilon_{(B,A)} = 1+\varepsilon_B+\varepsilon_\ast+\varepsilon_A$, with the first summand coming from $\Psi_1$.

Let $a,b\in \Ob(\Cat)$ be two objects with $|a|=i+1$ and $|b|=i-1$, so that $\M(a,b)$ is a 1-dimensional moduli space. It is therefore a disjoint union of circles and compact intervals. Recall that the framing can be described by a function $fr\colon \pi_0(\M(a,b))\to \Z/2$. Provided that $|x|=i$ or $i+1$, the new moduli space is
\[
 \M(\bar{a},\bar{b})=\M(a,b)\cup \M(x,b)\times \M(a,y)
\]
with the gluing either along $\M(x,b)\times \M(a,x)$ if $|x|=i$ or along $\M(y,b)\times \M(a,y)$ if $|x|=i+1$.

To understand the framed embeddings of these moduli spaces, we need to take a closer look at $\cw(a)$. The relevant part is 
\[
 [0,R]\times [-R,R]^{d_{i-1}}\times [0,R]\times [-R,R]^{d_i}
\]
with the two relevant boundaries given by $\partial_{i-1}\cw(a)$ obtained by setting the first interval $[0,R]$ to $0$ and $\partial_i\cw(a)$ where the second interval $[0,R]$ is set to $0$. The map $\Gamma_{a,b}\colon \M(a,b) \times \cw(b) \to \partial_{i-1}\cw(a)$ can then be viewed as the framed embedding.

Let us first consider the case $|x|=i$.

If $I\subset \M(a,b)$ is an interval, its endpoints are given by points $(C,D)\in \M(c,b)\times \M(a,c)$ and $(C',D')\in \M(c',b)\times \M(a,c')$ for some $c,c'\in \Ob(\Cat)$ with $|c|=i=|c'|$. If $c\not=x\not=c'$, then $I$ remains an interval in $\M(\bar{a},\bar{b})$ with the same framed embedding, so the value of $fr(I)$ does not change. The same holds for a circle.

If $c$ or $c'$ are equal to $x$, the interval $I$ gets glued to at least one interval coming from $\M(a,y)$. Before we consider this case let us look at intervals in $\M(x,b)\times \M(a,y)$. Because $|x|=i$ these are of the form $\{B\}\times J$ where $B\in \M(x,b)$ and $J$ is an interval in $\M(a,y)$. The endpoints of $J$ are given by points $(C,D)\in \M(c,y)\times \M(a,c)$ and $(C',D')\in \M(c',y)\times \M(a,c')$ for some $c,c'\in \Ob(\Cat)$ with $|c|=i=|c'|$.

If $c\not= x\not= c'$, then $\{B\}\times J$ remains an interval in $\M(\bar{a},\bar{b})$, with endpoints given by $((B,C),D)$ and $((B,C'),D')$. Let us analyze the framing of $\{B\}\times J$ in this case.

The interval $\{B\}\times J$ is embedded by $\Gamma_{x,b\times a,y}$ into $\partial_{i-1}\cw(a)$ and is basically just parallel to $\Gamma_{a,y}(J)$. The framing in the $\R^{d_{i-1}}$ part is changed by $1+\varepsilon_B+\varepsilon_\ast$ because of $\Gamma_{x,b}$, $\Psi_1$ and $\Gamma_{x,y}^{-1}$. But changing the framing on a fixed path by reflecting the first coordinate will change a standard path to a non-standard path.

We therefore get
\begin{equation}
 \label{eqn:frame_J_i}
 fr(\{B\} \times J) = fr(J) + 1+\varepsilon_\ast+\varepsilon_B.
\end{equation}

Now assume that endpoints in the interval $J\subset \M(a,y)$ are glued along a point $(B,A)\in \M(x,b)\times \M(a,x)$. The picture in Figure \ref{fig_cube} shows $\partial_{i-1}\cw(a)$, where we only show the first dimension of $\R^{d_{i-1}}$ and $\R^{d_i}$, and where the back square is $\partial_{i-1}\cw(a)\cap \partial_i\cw(a)$. Protruding to the front is the $[0,R]$ direction. The interval $\{B\}\times J$ is parallel to $J$, and is glued to the interval $I$ along the dotted interval.
\begin{figure}[ht]
 \psfrag{I}{$I$}
 \psfrag{BJ}{$\{B\}\times J$}
 \psfrag{J}{$J$}
 \includegraphics[height=6cm,width=6cm]{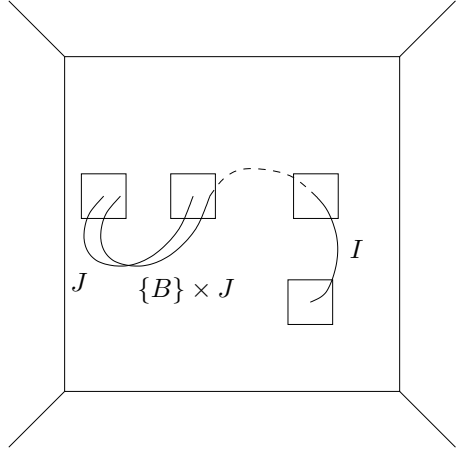}
 \caption{Gluing along one point.}
 \label{fig_cube}
\end{figure}

The dotted interval is the image of $\Psi_t(\{B\})$ within $\cw(x)$. 
Because of the convention that $B$ is embedded into $(-\infty,0)\times\{0\}\subset \R^{d_{i-1}}$, compare Remark \ref{remark_embedding}, this interval has framing $\varepsilon_B$.

The framing of the glued interval as an element of $\Z/2$ is therefore the sum of the framing of $I$ with the framing of $\{B\}\times J$ given by (\ref{eqn:frame_J_i}) and $\varepsilon_B$.

If both endpoints of $J$ are glued, we get two extra intervals whose contribution cancel each other. Using induction, we can summarize the gluings with $|x|=i=|y|+1$ as follows.

\begin{proposition}
\label{prop:gluing_form_i}
 Let $a,b\not=y$ be objects with $|a|=i+1=|x|+1=|b|+2$ and $K\subset \M(\bar{a},\bar{b})$. If $K$ is a circle, then either
\begin{enumerate}
 \item $K\subset \M(a,b)$, in which case $fr(K)$ is the same as before.
 \item $K=\{B\}\times K'\subset \M(x,b)\times \M(a,y)$ in which case $fr(K)=fr(K')$.
 \item $K$ is the result of gluing intervals $I_1,\ldots,I_k\subset \M(a,b)$ with intervals $\{B_1\}\times J_1,\ldots,\{B_k\}\times J_k\subset \M(x,b)\times \M(a,y)$ in which case
\[
 fr(K) = k(1+\varepsilon_\ast) + \sum_{i=1}^k (fr(I_i)+fr(J_i) + \varepsilon_{B_i}).
\]
\end{enumerate}
If $K$ is an interval, then either
\begin{enumerate}
\setcounter{enumi}{3}
 \item the endpoints of $K$ are of the form $(B,A)\in \M(c,b)\times \M(a,c)$, $(B',A')\in \M(c',b)\times \M(a,c')$ with $c,c'\not=x$. Then $K$ is the result of gluing intervals $I_1,\ldots,I_{k+1}\subset \M(a,b)$ with intervals $\{B_1\}\times J_1,\ldots,\{B_k\}\times J_k\subset \M(x,b)\times \M(a,y)$ in which case
\[
 fr(K) = k(1+\varepsilon_\ast) + \sum_{i=1}^k (fr(I_i)+fr(J_i) + \varepsilon_{B_i}) + fr(I_{k+1}).
\]
 \item the endpoints of $K$ are of the form $(B,A)\in \M(c,b)\times \M(a,c)$, $(B_k,C,D)\in \M(x,b)\times \M(c',y)\times \M(a,c')$ with $c,c'\not=x$. Then $K$ is the result of gluing intervals $I_1,\ldots,I_k\subset \M(a,b)$ with intervals $\{B_1\}\times J_1,\ldots, \{B_k\}\times J_k \subset \M(x,b)\times \M(a,y)$ in which case
\[
 fr(K) = k(1+\varepsilon_\ast) + \varepsilon_{B_k}+ \sum_{i=1}^k (fr(I_i)+fr(J_i)+\varepsilon_{B_i}).
\]
 \item the endpoints of $K$ are of the form $(B_1,C,D)\in \M(x,b)\times\M(c,y)\times \M(a,c)$ and $(B_{k+1},C',D')\in \M(x,b)\times \M(c',y)\times \M(a,c')$ with $c,c'\not=x$. Then $K$ is the result of gluing intervals $I_1,\ldots,I_k\subset \M(a,b)$ with intervals $\{B_1\}\times J_1,\ldots, \{B_{k+1}\}\times J_{k+1} \subset \M(x,b)\times \M(a,y)$ in which case
\begin{align*}
 fr(K) = &\, (k+1)(1+\varepsilon_\ast) + \varepsilon_{B_1}+\varepsilon_{B_{k+1}} + \\
  & \, \sum_{i=1}^{k+1} (fr(J_i)+\varepsilon_{B_i}) + \sum_{i=1}^k fr(I_i).
\end{align*}

\end{enumerate}

\end{proposition}

\begin{proof}
 If $K$ is a circle that was already a circle in either $\M(a,b)$ or $\M(x,b)\times \M(a,y)$, its framing is determined by an element of the fundamental group, and the triviality or non-triviality of this element is preserved. If this circle is the result of gluing various intervals, we need the same number of intervals coming from $\M(a,b)$ and from $\M(x,b)\times \M(a,y)$. Gluing of standard or non-standard paths is encoded by adding.
\end{proof}

Let us now consider the case $|x|=i+1=|y|+1$. Let $J\subset \M(x,b)$ be an interval with endpoints $(C,D)\in \M(c,b)\times\M(a,c)$ and $(C',D')\in \M(c',b)\times \M(x,c')$ with $c,c'\not=y$. For $A\in  \M(a,y)$ the interval $J\times \{A\}$ is embedded into $\R^{d_{i-1}}\times [0,\infty) \times \R^{d_i}$ as follows.
Again we can consider $\cw(b)=\{0\}\times [-\varepsilon,\varepsilon]^{d_{i-1}}\times\{0\}\times[\varepsilon,\varepsilon]^{d_i}$, that is, this contains the framing information in $\Gamma_{x,b\times a,y}$. Also note that $\Gamma_{x,y}^{-1}\circ \Psi_1(\partial_i\cw(x)-\cw_y(x))\subset \cw(y)$ is mapped to $\R^{d_{i-1}}\times \{0\} \times \R^{d_i}$ by $\Gamma_{x,b\times a,y}$ and contains the boundary of $J\times \{A\}$.

In Figure \ref{fig_reflect} we see two examples of framings of the interval $J$, one horizontally with $c=c'$ and one vertically with $c\not=c'$. Note that $\cw_c(x)$ are $\cw_{c'}(x)$ are embedded below $\cw_y(x)$ by Remark \ref{remark_embedding}. The map $\Psi_1$ now reflects $\partial_{i-1}\cw(x)$ and $\partial_i\cw(x)$ into $\cw_y(x)$ as in Figure \ref{fig_reflect}. We see that the intervals change their vertical orientation, and frames that pointed positively into $[0,R]$ within $\partial_{i-1}\cw(x)$ now point negatively into $[0,R]$ within $\cw_y(x)$.
\begin{figure}[ht]
 \includegraphics[height=5cm,width=12cm]{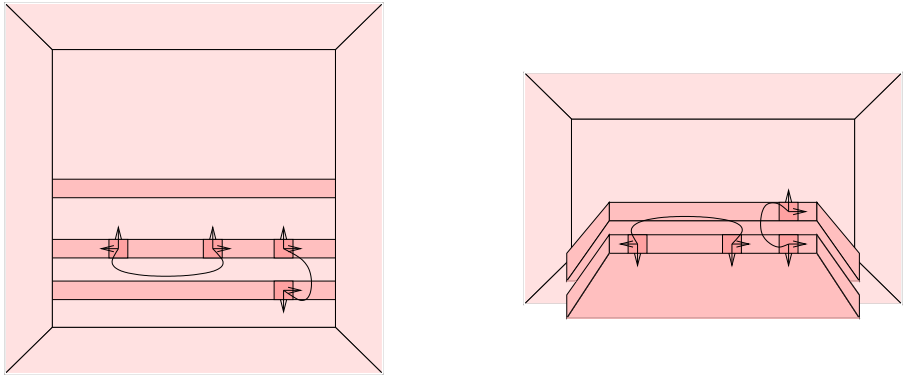}
 \caption{The left cube represents $\partial_{i-1}\cw(x)$ containing intervals $J$, with the back face containing $\partial_{i-1}\cw_t(x)$ for $t=y,c,c'$. The right cuboid represents $\cw_y(x)$ after applying $\Psi_1$. The cells $\cw_c(x)$ and $\cw_{c'}(x)$ are contained in the interior, and their vertical position and orientation has changed. The intervals are also contained in the interior.}
 \label{fig_reflect}
\end{figure}

This last change of orientation gets reversed though once we glue this cell into $\partial_{i-1}\cw(a)$. Also, reflecting the vertical direction preserves standard frames, as can be easily seen from our choice of standard frames. This also means that $\Gamma_{x,y}$ and $\Gamma_{a,y}$ have no further effect on the framing. We therefore get
\begin{equation}
 \label{eqn:frame_J_i+1}
 fr(J\times\{A\})=fr(J)
\end{equation}
for any $A\in \M(a,y)$.

If one of the endpoints is $(C,\ast)\in\M(y,b)\times \M(x,y)$, there is a gluing to an interval $I\in \M(a,b)$ with one endpoint $(C,A)\in\M(y,b)\times\M(a,y)$.
If $\varepsilon_{D'}\not=\varepsilon_\ast$ the situation is as in Figure \ref{fig_reflect3}.

\begin{figure}[ht]
 \includegraphics[height=5cm,width=12cm]{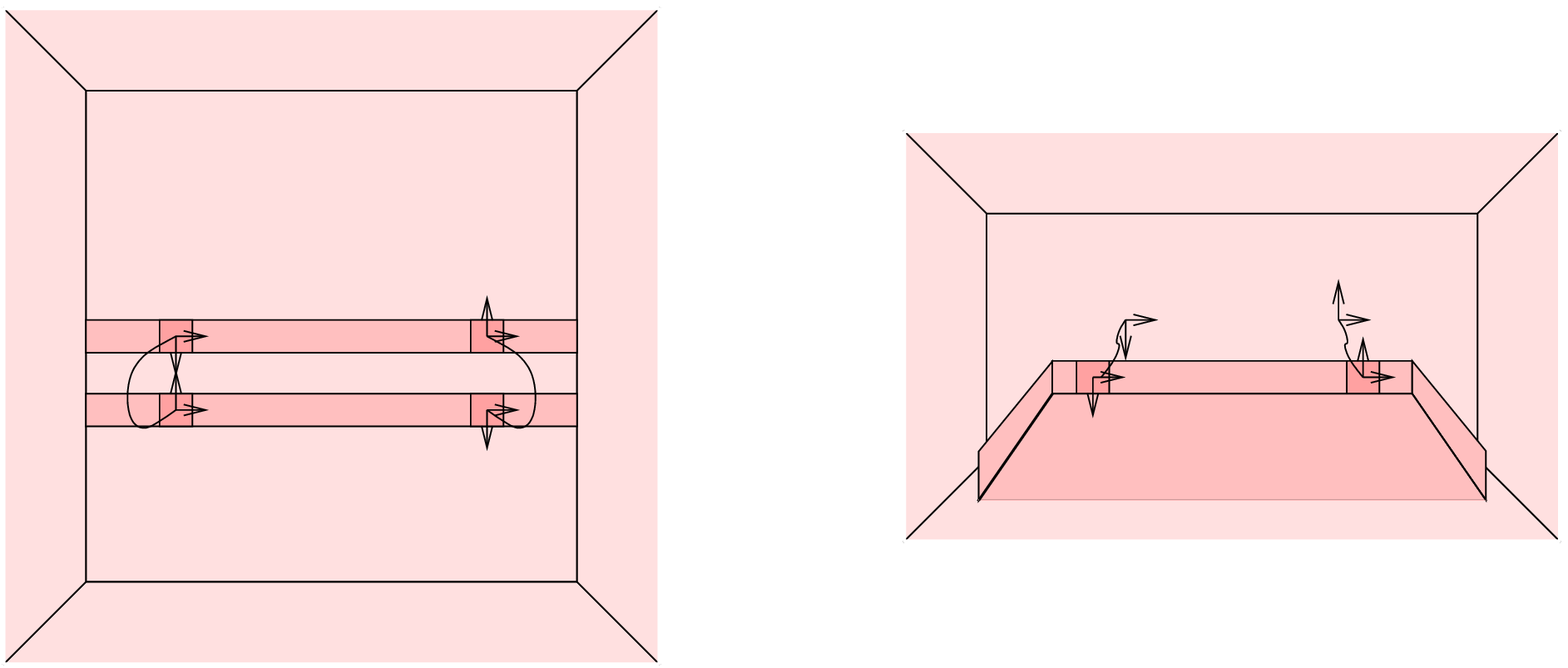}
 \caption{Applying $\Psi_1$ if there is one point of gluing, with $\varepsilon_\ast\not=\varepsilon_{D'}$.}
 \label{fig_reflect3}
\end{figure}

Note that if $J$ in $\partial_{i-1}\cw(x)$ is framed with the standard framing, then the framing of the interval on the right, going from $\PP$ to $\PM$, fixes the horizontal coordinate, while in the interval on the left we have to add a twist to get the standard path. This means that after applying $\Psi_1$, the framing on the right between $\PP$ and $\PP$ is standard, while the framing on the left between $\PM$ and $\PM$ is non-standard, as there is a twist.

There are two more cases to check, namely when $(C,\ast)$ is framed $\MP$ or $\MM$. But a similar argument shows that the change in framing of $J\times \{A\}$ is determined by $\varepsilon_\ast+\varepsilon_C$.

It remains to check the case $\varepsilon_{D'}=\varepsilon_\ast$ which is more difficult to visualize. The situation is as in Figure \ref{fig_rotate}.

\begin{figure}[ht]
 \psfrag{P}{$\Psi_1$}
 \includegraphics[height=5cm,width=12cm]{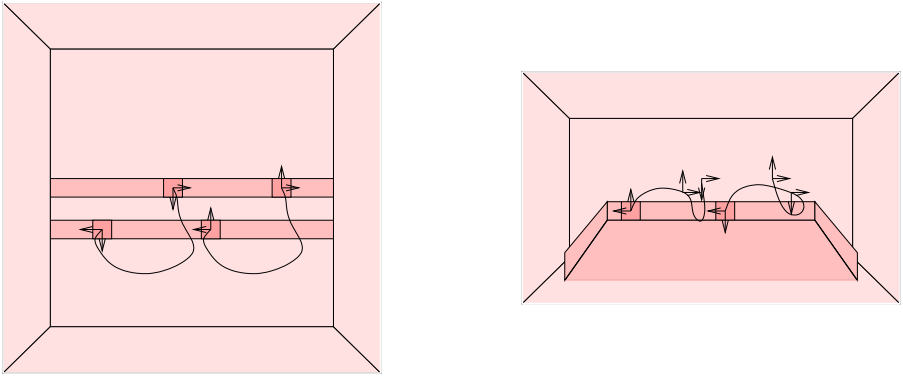}
 \caption{Applying $\Psi_1$ if there is one point of gluing, with $\varepsilon_\ast=\varepsilon_{D'}$.}
 \label{fig_rotate}
\end{figure}

We can think of the path starting at $(C,\ast)$ then droppiong sharpely vertically without changing its framing, and then forming the standard frame between a change in the first coordinate. During the sharp descent the vertical framing gets nearly tangential to the path, and if $\varepsilon_\ast=0$ it is pointing towards the point $(C,\ast)$. After reflecting by $\Psi_1$ the sharp descent then flips the second coordinate by pointing towards the point $(C,\ast)$.
If $\varepsilon_\ast=1$ the second coordinate gets flipped by pointing away from $(C,\ast)$. Using handmotions we can now see that the change in framing is also determined by $\varepsilon_\ast+\varepsilon_C$. That is, if $J$ gets glued at one endpoint, the change of frame is given by
\[
 fr(J\times \{A\}) = fr(J) + \varepsilon_\ast+\varepsilon_C.
\]

If we glue both endpoints, then $\Psi_1(J)$ has the same framings at the endpoints as $J$, but there is a reflection in the $[0,R]$-coordinate. This means standard paths go to non-standard paths. In that case we always have to add $1$ to the new frame formula.
\begin{proposition}
\label{prop:gluing_form_ii}
 Let $a\not=x,b$ be objects with $|a|=i+1=|x|=|b|+2$ and $K\subset \M(\bar{a},\bar{b})$. If $K$ is a circle, then either
\begin{enumerate}
 \item $K\subset \M(a,b)$, in which case $fr(K)$ is the same as before.
 \item $K=K'\times\{A\} \subset \M(x,b)\times \M(a,y)$ in which case $fr(K)=fr(K')$.
 \item $K$ is the result of gluing intervals $I_1,\ldots,I_k\subset \M(a,b)$ with intervals $J_1\times\{A_1\},\ldots,J_k\times\{A_k\} \subset \M(x,b)\times \M(a,y)$ in which case
\[
 fr(K) = k + \sum_{i=1}^k (fr(I_i)+fr(J_i)).
\]
\end{enumerate}
If $K$ is an interval, then either
\begin{enumerate}
\setcounter{enumi}{3}
 \item the endpoints of $K$ are of the form $(C,A)\in \M(c,b)\times \M(a,c)$, $(C',A')\in \M(c',b)\times \M(a,c')$ with $c,c'\not=x$. Then $K$ is the result of gluing intervals $I_1,\ldots,I_{k+1}\subset \M(a,b)$ with intervals $J_1\times\{A_1\},\ldots,J_k\times \{A_k\}\subset \M(x,b)\times \M(a,y)$ in which case
\[
 fr(K) = k + \sum_{i=1}^k (fr(I_i)+fr(J_i)) + fr(I_{k+1}).
\]
 \item the endpoints of $K$ are of the form $(C,A)\in \M(c,b)\times \M(a,c)$, $(C',D',A_k)\!\in \M(c',b)\times \M(x,c')\times \M(a,y)$ with $c,c'\not=x$. Then $K$ is the result of gluing intervals $I_1,\ldots,I_k\subset \M(a,b)$ with intervals $J_1\times \{A_1\},\ldots,J_k \times\{A_k\} \subset \M(x,b)\times \M(a,y)$ in which case
\[
 fr(K) = k + \varepsilon_{C'}+\varepsilon_{D'} + \sum_{i=1}^k (fr(I_i)+fr(J_i)).
\]
 \item the endpoints of $K$ are of the form $(C,D,A_1)\in \M(c,b)\times \M(x,c)\times\M(a,y)$ and $(C',D',A_{k+1})\in \M(c',b)\times \M(x,c')\times \M(a,y)$ with $c,c'\not=x$. Then $K$ is the result of gluing intervals $I_1,\ldots,I_k\subset \M(a,b)$ with intervals $J_1\times\{A_1\},\ldots,J_{k+1} \times \{A_{k+1}\} \subset \M(x,b)\times \M(a,y)$ in which case
\begin{align*}
 fr(K) = &\, 1+k + \varepsilon_{C'}+\varepsilon_{D'}+\varepsilon_C+\varepsilon_D + \sum_{i=1}^k (fr(I_i)+fr(J_i)) + fr(J_{k+1}).
\end{align*}
\end{enumerate}

\end{proposition}

\begin{proof}
Following the discussion above we only remark that if we glue an interval $J$ at one endpoint, then $\varepsilon_\ast+\varepsilon_C=1+\varepsilon_{C'}+\varepsilon_{D'}$. Also, the case where there is no gluing corresponds to the last case with $k=0$, noticing that $1+\varepsilon_{C'}+\varepsilon_{D'}+\varepsilon_C+\varepsilon_D =0$, leading to (\ref{eqn:frame_J_i+1}).
\end{proof}

\section{Examples}
\label{sec:examples}

\subsection{The (3,4)-torus knot}
\label{subsec:torus34}
In \cite[\S 6.1]{JLS} we calculated a second Steenrod square on the $\Z/2$ Khovanov cohomology of the $(3,4)$-torus knot
\[ \Sq^2 : \Kh^{2,11}(T_{3,4}; \Z/2) \rightarrow \Kh^{4,11}(T_{3,4}; \Z/2) {\rm .} \]
It followed that the Lipshitz-Sarkar stable homotopy type in quantum degree $11$ is
\[ \X^{\Kh}_{11}(T_{3,4}) \simeq \Sigma^{-1}\RP^5 / \RP^2 {\rm .} \]

This calculation was done using the glued diagram expressing $T_{3,4}$ as the pretzel knot $P(-2,3,3)$, and taking advantage of the small flow category thus available.  After this the flow category was reduced again to a flow category $\Cat$ with $19$ objects using upward- and downward-closed subcategories corresponding to contractible spaces. The objects and $0$-dimensional moduli spaces of $\Cat$ are depicted in Figure \ref{page1}, with the $1$-dimensional moduli spaces listed in Figure \ref{mod_spaces_1}.

\begin{figure}[ht!]
	\includegraphics[scale=0.475]{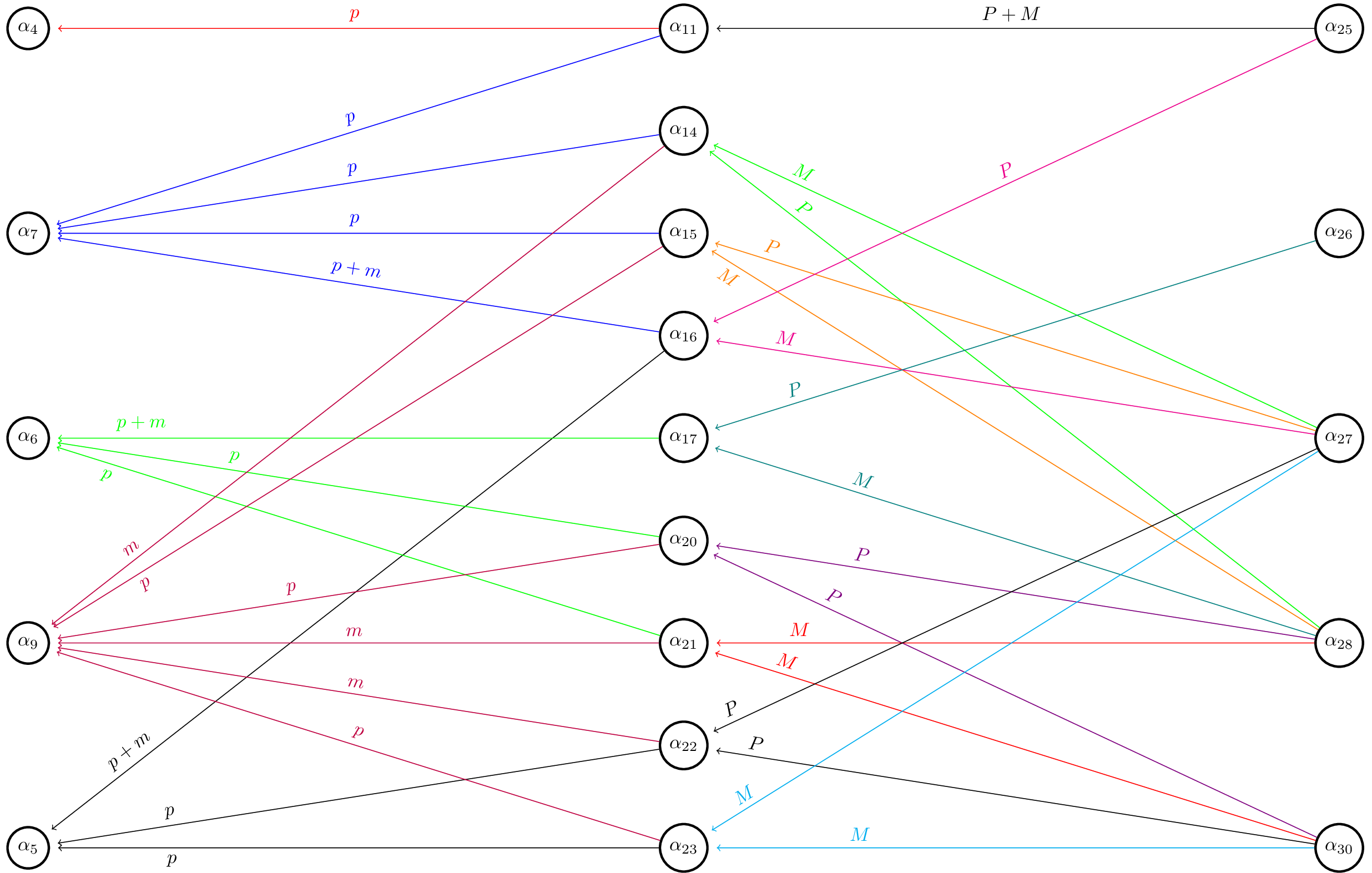}
	\caption{The flow category $\Cat$.}
	\label{page1}
\end{figure}
The $0$-dimensional moduli spaces are points where $P$ or $p$ indicates a positive framing, while $M$ or $m$ indicates a negative framing. For the $1$-dimensional moduli spaces a red $0$ indicates the standard framing of the interval, while a red $1$ indicates the non-standard framing. These values were determined in \cite[\S 6.1]{JLS}.
\begin{figure}
	\includegraphics[scale=1]{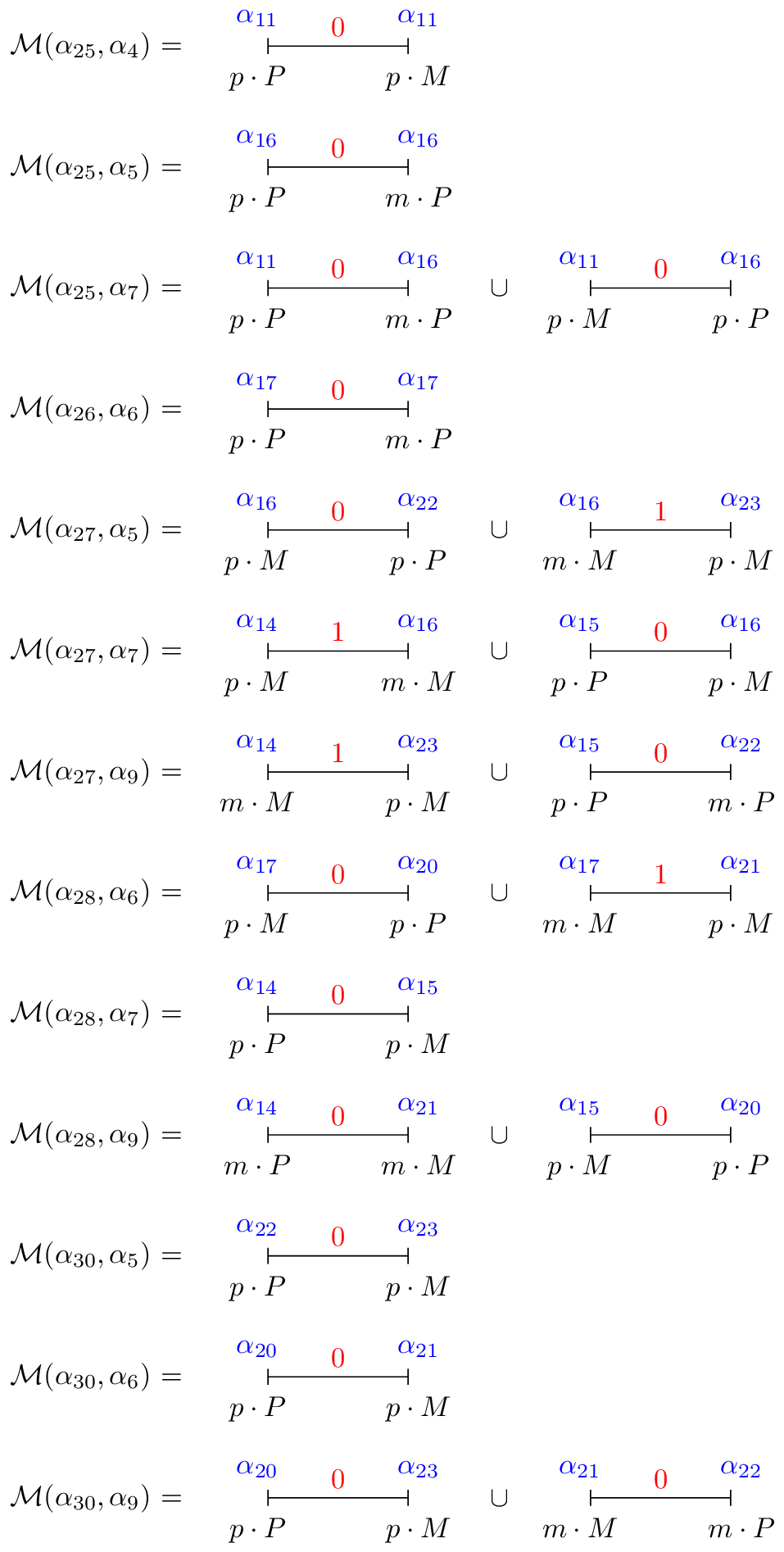}
	\caption{$1$-dimensional moduli spaces for $\Cat$.}
	\label{mod_spaces_1}
\end{figure}
Together with a choice of topological boundary matching, the Steenrod square can be calculated from this data, as was done in \cite[\S 6.1]{JLS}.

We would like to see more directly that the space associated to $\cC$ is $\Sigma^{-1}\RP^5 / \RP^2$.  The moduli spaces for $\Cat$ can be cancelled using a combination of handle cancellation and the Whitney trick until there are no $0$-dimensional moduli spaces left that are either single points, or that contain two points with opposite sign. There are \emph{a priori} many ways to do this, but we shall proceed as follows.

Firstly, there are four $0$-dimensional moduli spaces in $\Cat$ (see Figure \ref{page1}) where the Whitney trick can be applied. These are $\mathcal{M}(\alpha_{25},\alpha_{11})$, $\mathcal{M}(\alpha_{16},\alpha_{7})$, $\mathcal{M}(\alpha_{16},\alpha_{5})$ and $\mathcal{M}(\alpha_{17},\alpha_{6})$ (and we shall cancel them in this order).
By cancelling $\mathcal{M}(\alpha_{25},\alpha_{11})$, the $1$-dimensional moduli spaces $\mathcal{M}(\alpha_{25},\alpha_{4})$ and $\mathcal{M}(\alpha_{25},\alpha_{7})$ both change. Originally, $\mathcal{M}(\alpha_{25},\alpha_{4})$ is the interval
\begin{center}
\begin{tikzpicture}
\node[] at (-6,0) {$\mathcal{M}(\alpha_{25},\alpha_{4})$ = };
\node[label=below:$p \cdot P$, label=above:{\color{blue}$\alpha_{11}$}] (25_4_1)at (-4,0) {};
\node[label=below:$p \cdot M$, label=above:{\color{blue}$\alpha_{11}$}] (25_4_2)at (-2,0) {};

\draw[|-|] (25_4_1)--(25_4_2) node[sloped, above,pos=0.5, red]{$0$};
\end{tikzpicture}
\end{center}
After performing the Whitney trick, where $x= \alpha_{25}$ and $y = \alpha_{11}$, the two endpoints of the interval are identified. The result is the circle
\begin{center}
\begin{tikzpicture}
\node[] at (-6,0) {$\mathcal{M}(\bar{\alpha}_{25},\bar{\alpha}_{4})$ = };
\node[] at (-4.3,0) {\color{red} 1};
\draw (-4.3,0) circle (0.5cm);
\end{tikzpicture}
\end{center}
framed with a $1$ using Proposition \ref{prop:whitney_glue_formula_second}. Notice that a circle with a $1$-framing in a moduli space $\M(\alpha,\beta)$ means that the attaching map of $C(\alpha)$ contains a wedge summand which is attached to $C(\beta)$ in a way homotopic to a constant map. As there are no objects of higher degree, we can simply remove this circle from the moduli space without changing the stable homotopy type of the framed flow category. Similarly, two circles with a $0$-framing in a moduli space can also be removed.

The moduli space $\mathcal{M}(\alpha_{25},\alpha_{7})$ originally consists of the two intervals
\begin{center}
\begin{tikzpicture}
\node[] at (-6,-7.5) {$\mathcal{M}(\alpha_{25},\alpha_{7})$ = };
\node[label=below:$p \cdot P$, label=above:{\color{blue}$\alpha_{11}$}] (25_7_1)at (-4,-7.5) {};
\node[label=below:$m \cdot P$, label=above:{\color{blue}$\alpha_{16}$}] (25_7_2)at (-2,-7.5) {};
\node[label=below:$p \cdot M$, label=above:{\color{blue}$\alpha_{11}$}] (25_7_3)at (0,-7.5) {};
\node[label=below:$p \cdot P$, label=above:{\color{blue}$\alpha_{16}$}] (25_7_4)at (2,-7.5) {};

\node[] at (-1,-7.5) {$\cup$};

\draw[|-|] (25_7_1)--(25_7_2) node[sloped, above,pos=0.5, red]{$0$};
\draw[|-|] (25_7_3)--(25_7_4) node[sloped, above,pos=0.5, red]{$0$};
\end{tikzpicture}
\end{center}
After performing the Whitney trick with $x= \alpha_{25}$ and $y = \alpha_{11}$, these two intervals are glued together along the boundaries corresponding to $\mathcal{M}(\alpha_{11},\alpha_7) \times \mathcal{M}(\alpha_{25},\alpha_{11})$. The result is the single interval
\begin{center}
\begin{tikzpicture}
\node[] at (-6,-7.5) {$\mathcal{M}(\bar{\alpha}_{25},\bar{\alpha}_{7})$ = };
\node[label=below:$m \cdot P$, label=above:{\color{blue}$\alpha_{16}$}] (25_7_1)at (-4,-7.5) {};
\node[label=below:$p \cdot P$, label=above:{\color{blue}$\alpha_{16}$}] (25_7_2)at (-2,-7.5) {};

\draw[|-|] (25_7_1)--(25_7_2) node[sloped, above,pos=0.5, red]{$1$};
\end{tikzpicture}
\end{center}
which is framed with a $1$ using Proposition \ref{prop:whitney_glue_formula_second}.

After performing Whitney tricks in $\M(\alpha_{16},\alpha_7)$, $\M(\alpha_{16},\alpha_5)$ and $\M(\alpha_{17},\alpha_6)$ the following $1$-dimensional moduli spaces have changed:
\begin{align*}
 \mathcal{M}(\alpha_{25},\alpha_{4}) & = \emptyset & \M(\alpha_{25},\alpha_{5}) & = \emptyset \\
 \mathcal{M}(\alpha_{25},\alpha_{7}) & = \emptyset & \M(\alpha_{26},\alpha_{6}) & = \emptyset
\end{align*}
because we can remove circles with a $1$-framing, and

\begin{tikzpicture}[
solidnode/.style={circle, minimum size = 0.5mm, fill=black!, very thick}
]

\node[] at (-6,-4) {$\mathcal{M}(\alpha_{27},\alpha_{5})$ = };
\node[label=below:$p \cdot P$, label=above:{\color{blue}$\alpha_{22}$}] (27_5_1)at (-4.3,-4) {};
\node[label=below:$p \cdot M$, label=above:{\color{blue}$\alpha_{23}$}] (27_5_2)at (-2.3,-4) {};

\node[] at (0,-4) {$\mathcal{M}(\alpha_{27},\alpha_{7})$ = };
\node[label=below:$p \cdot M$, label=above:{\color{blue}$\alpha_{14}$}] (27_7_1)at (1.7,-4) {};
\node[label=below:$p \cdot P$, label=above:{\color{blue}$\alpha_{15}$}] (27_7_2)at (3.7,-4) {};

\node[] at (-6,-5.5) {$\mathcal{M}(\alpha_{28},\alpha_{6})$ = };
\node[label=below:$p \cdot P$, label=above:{\color{blue}$\alpha_{20}$}] (28_6_1)at (-4.3,-5.5) {};
\node[label=below:$p \cdot M$, label=above:{\color{blue}$\alpha_{21}$}] (28_6_2)at (-2.3,-5.5) {};


\draw[|-|] (27_5_1)--(27_5_2) node[sloped, above,pos=0.5, red]{$0$};

\draw[|-|] (28_6_1)--(28_6_2) node[sloped, above,pos=0.5, red]{$0$};

\draw[|-|] (27_7_1)--(27_7_2) node[sloped, above,pos=0.5, red]{$0$};
\end{tikzpicture}

We now proceed by cancelling the one-point $0$-dimensional moduli spaces using handle cancellation. Recall that in Definition \ref{cancelledcat}, moduli spaces $\mathcal{M}(a,b)$ in the original category become $\mathcal{M}(\bar{a},\bar{b}) = \mathcal{M}(a,b) \times \big( \mathcal{M}(x,b) \times \mathcal{M}(a,y) \big)$ in the cancelled category. So if $\mathcal{M}(\alpha_i,\alpha_j) = \ast$ denotes the pair being cancelled in $\Cat_1$, the first three pairs can be cancelled in the following order, where $\mathcal{M}(\alpha_i, \alpha_l) \times \mathcal{M}(\alpha_k, \alpha_j) = \emptyset$ so $\mathcal{M}(\bar{\alpha}_k, \bar{\alpha}_l) =  \mathcal{M}(\alpha_k,\alpha_l)$ for each $k,l \notin \{ i,j \}$ (see Definition \ref{cancelledcat}):
\begin{center}
\begin{enumerate}
\item $\mathcal{M}(\alpha_{11}, \alpha_4) = p$
\item $\mathcal{M}(\alpha_{25}, \alpha_{16}) =P$
\item $\mathcal{M}(\alpha_{26}, \alpha_{17}) = P$
\end{enumerate}
\end{center}
The flow category $\Cat_1$, obtained as a result of cancelling these three pairs via handle cancellation, is depicted in Figure \ref{page2}.
\begin{figure}[ht!]
	\includegraphics[scale=0.5]{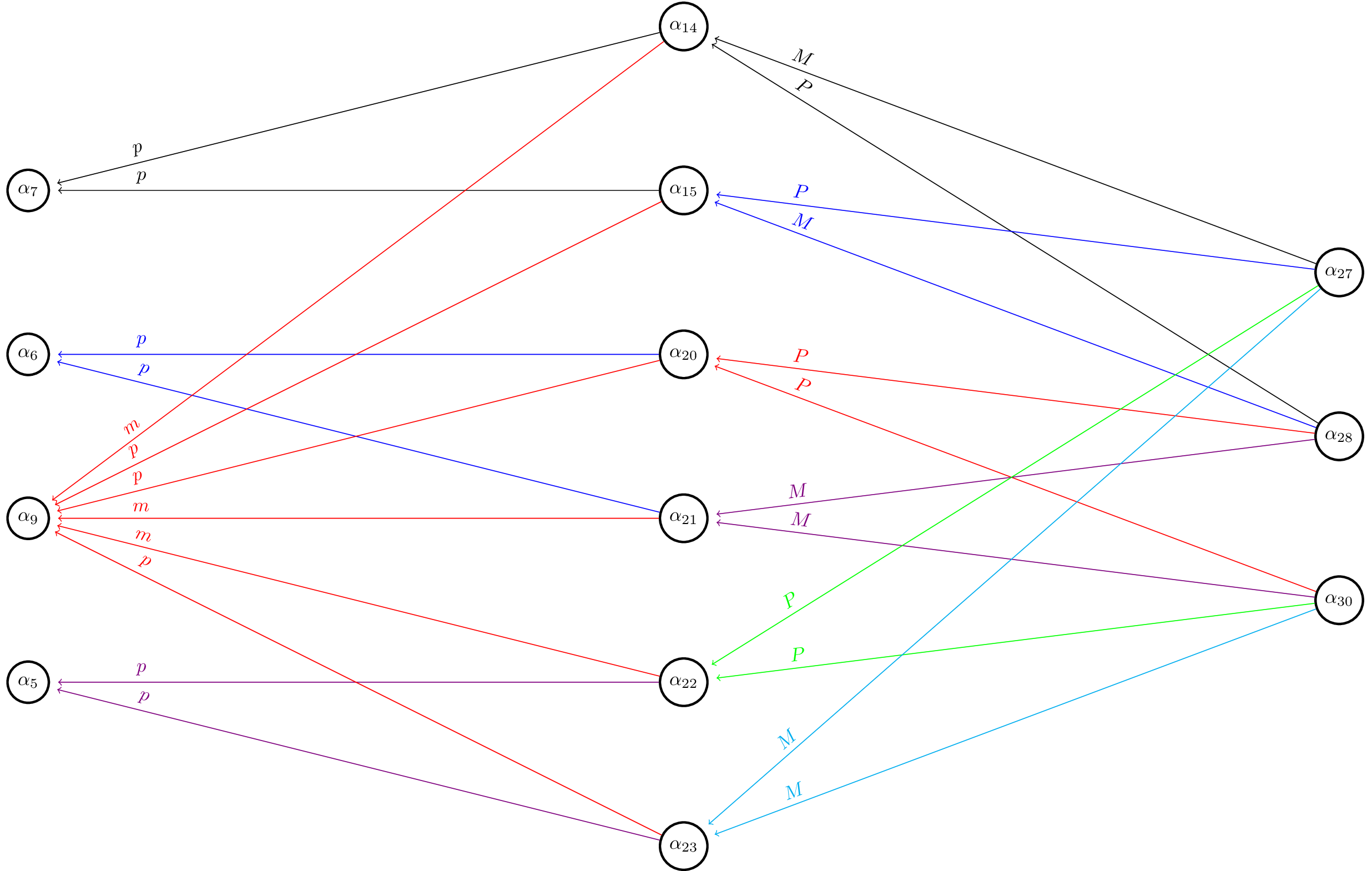}
	\caption{The flow category $\Cat_1$.}
	\label{page2}
\end{figure}
Note that there have been no changes in the $1$-dimensional moduli spaces during these cancellations.

Continuing the process of cancellation further, let us now cancel the moduli space $\mathcal{M}(\alpha_{14},\alpha_7) = p$ using handle cancellation. The only $0$-dimensional moduli space that changes is
\[
\mathcal{M}(\bar{\alpha}_{15},\bar{\alpha}_9) = \mathcal{M}({\alpha}_{15},{\alpha}_9) \sqcup \big( \mathcal{M}({\alpha}_{14},{\alpha}_9) \times \mathcal{M}({\alpha}_{15},{\alpha}_7) \big) = p \sqcup \tilde{p}
\]
where $\tilde{p}$ is used to distinguish the point from $p$. The resultant flow category $\Cat_2$ is depicted in Figure \ref{page3}.
\begin{figure}[ht!]
	\includegraphics[scale=0.5]{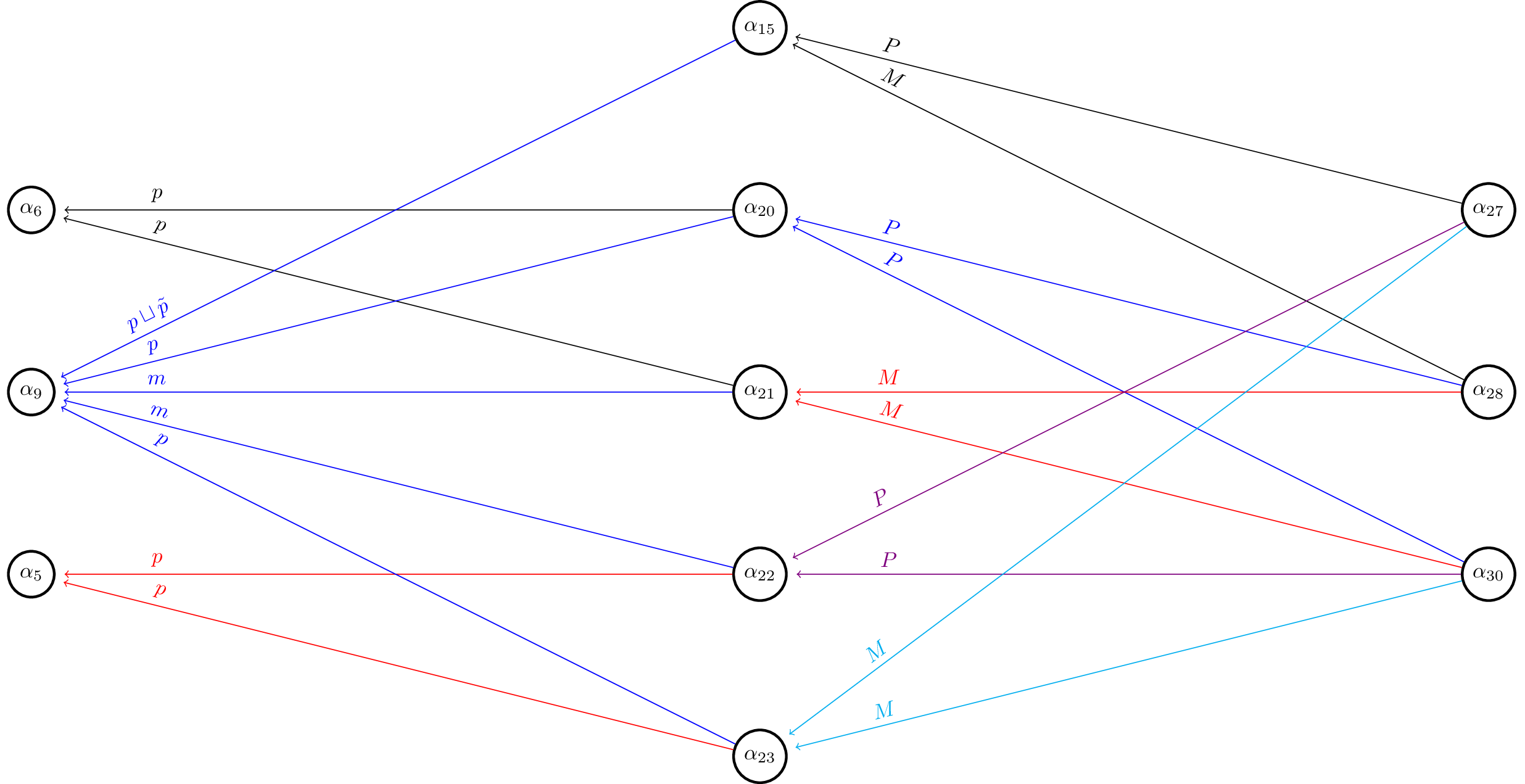}
	\caption{The flow category $\Cat_2$.}
	\label{page3}
\end{figure}
Two $1$-dimensional moduli spaces change as a result of this cancellation, namely $\mathcal{M}({\alpha}_{27},{\alpha}_9)$ and $\mathcal{M}({\alpha}_{28},{\alpha}_9)$. The corresponding $1$-dimensional moduli spaces for $\Cat_2$ are:
\begin{center}
\begin{tikzpicture}
\node[] at (-6,-10.5) {$\mathcal{M}(\bar{\alpha}_{27},\bar{\alpha}_{9})$ = };
\node[label=below:$\tilde{p} \cdot P$, label=above:{\color{blue}$\alpha_{15}$}] (27_9_1)at (-4,-10.5) {};
\node[label=below:$p \cdot M$, label=above:{\color{blue}$\alpha_{23}$}] (27_9_2)at (-2,-10.5) {};
\node[label=below:$p \cdot P$, label=above:{\color{blue}$\alpha_{15}$}] (27_9_3)at (0,-10.5) {};
\node[label=below:$m \cdot P$, label=above:{\color{blue}$\alpha_{22}$}] (27_9_4)at (2,-10.5) {};

\node[] at (-1,-10.5) {$\cup$};

\node[] at (-6,-12) {$\mathcal{M}(\bar{\alpha}_{28},\bar{\alpha}_{9})$ = };
\node[label=below:$\tilde{p} \cdot M$, label=above:{\color{blue}$\alpha_{15}$}] (28_9_1)at (-4,-12) {};
\node[label=below:$m \cdot M$, label=above:{\color{blue}$\alpha_{21}$}] (28_9_2)at (-2,-12) {};
\node[label=below:$p \cdot M$, label=above:{\color{blue}$\alpha_{15}$}] (28_9_3)at (0,-12) {};
\node[label=below:$p \cdot P$, label=above:{\color{blue}$\alpha_{20}$}] (28_9_4)at (2,-12) {};

\node[] at (-1,-12) {$\cup$};

\draw[|-|] (27_9_1)--(27_9_2) node[sloped, above,pos=0.5, red]{$0$};
\draw[|-|] (27_9_3)--(27_9_4) node[sloped, above,pos=0.5, red]{$0$};

\draw[|-|] (28_9_1)--(28_9_2) node[sloped, above,pos=0.5, red]{$1$};
\draw[|-|] (28_9_3)--(28_9_4) node[sloped, above,pos=0.5, red]{$0$};
\end{tikzpicture}
\end{center}
The framings of these new moduli spaces are calculated using (2) of Proposition \ref{prop:gluing_form_i} where $k=1$, $\varepsilon_\ast = \varepsilon_{p} = 0$ and $\varepsilon_{B_1} = \varepsilon_{m} = 1$ for both moduli spaces.

Working in the flow category $\Cat_2$, we shall now cancel $\mathcal{M}(\alpha_{20}, \alpha_6) = p$ followed by $\mathcal{M}(\alpha_{30}, \alpha_{23}) = M$. Firstly, cancelling $\mathcal{M}(\alpha_{20}, \alpha_6)$ results in the change of one $0$-dimensional moduli space, $\mathcal{M}(\alpha_{21}, \alpha_9)$, which becomes
\[
\mathcal{M}(\bar{\alpha}_{21},\bar{\alpha}_9) = \mathcal{M}({\alpha}_{21},{\alpha}_9) \sqcup \big( \mathcal{M}({\alpha}_{20},{\alpha}_9) \times \mathcal{M}({\alpha}_{21},{\alpha}_6) \big) = m \sqcup \tilde{m} {\rm .}
\]
Further, cancelling $\mathcal{M}(\alpha_{30}, \alpha_{23})$ results in the change of two $0$-dimensional moduli spaces, $\mathcal{M}(\alpha_{27}, \alpha_{22})$ and $\mathcal{M}(\alpha_{27}, \alpha_{21})$, which become
\begin{align*}
& \mathcal{M}(\bar{\alpha}_{27},\bar{\alpha}_{22}) = \mathcal{M}({\alpha}_{27},{\alpha}_{22}) \sqcup \big( \mathcal{M}({\alpha}_{30},{\alpha}_{22}) \times \mathcal{M}({\alpha}_{27},{\alpha}_{23}) \big) = P \sqcup M {\rm .} \\
& \mathcal{M}(\bar{\alpha}_{27},\bar{\alpha}_{21}) = \mathcal{M}({\alpha}_{27},{\alpha}_{21}) \sqcup \big( \mathcal{M}({\alpha}_{30},{\alpha}_{21}) \times \mathcal{M}({\alpha}_{27},{\alpha}_{23}) \big) = \emptyset \sqcup P = P {\rm .} 
\end{align*}
The flow category $\Cat_3$ obtained as a result of these two cancellations is depicted in Figure \ref{page4}.
\begin{figure}[ht!]
	\includegraphics[scale=0.5]{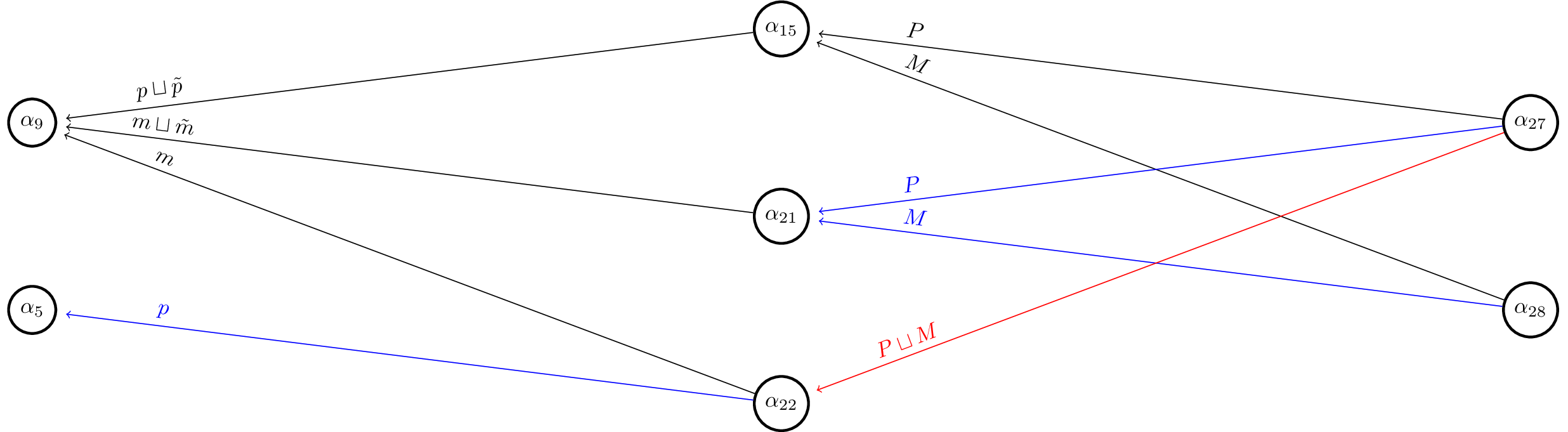}
	\caption{The flow category $\Cat_3$.}
	\label{page4}
\end{figure}
The $1$-dimensional moduli spaces of $\Cat_3$, obtained as alterations of previous moduli spaces, are:
\begin{center}
\begin{tikzpicture}
\node[] at (-6,-7.5) {$\mathcal{M}(\alpha_{27},\alpha_{5})$ = };
\node[label=below:$p \cdot P$, label=above:{\color{blue}$\alpha_{22}$}] (27_5_1)at (-4,-7.5) {};
\node[label=below:$p \cdot M$, label=above:{\color{blue}$\alpha_{22}$}] (27_5_2)at (-2.5,-7.5) {};

\node[] at (-6,-8.9) {$\mathcal{M}(\alpha_{27},\alpha_{9})$ = };
\node[label=below:$\tilde{p} \cdot P$, label=above:{\color{blue}$\alpha_{15}$}] (27_9_1)at (-4,-8.9) {};
\node[label=below:$\tilde{m} \cdot P$, label=above:{\color{blue}$\alpha_{21}$}] (27_9_2)at (-2.5,-8.9) {};
\node[label=below:$p \cdot P$, label=above:{\color{blue}$\alpha_{15}$}] (27_9_3)at (-1,-8.9) {};
\node[label=below:$m \cdot P$, label=above:{\color{blue}$\alpha_{22}$}] (27_9_4)at (0.5,-8.9) {};
\node[label=below:$m \cdot P$, label=above:{\color{blue}$\alpha_{21}$}] (27_9_5)at (2,-8.9) {};
\node[label=below:$m \cdot M$, label=above:{\color{blue}$\alpha_{22}$}] (27_9_6)at (3.5,-8.9) {};

\node[] at (-1.7,-8.9) {$\cup$};
\node[] at (1.3,-8.9) {$\cup$};

\draw[|-|] (27_5_1)--(27_5_2) node[sloped, above,pos=0.5, red]{$1$};

\draw[|-|] (27_9_1)--(27_9_2) node[sloped, above,pos=0.5, red]{$0$};
\draw[|-|] (27_9_3)--(27_9_4) node[sloped, above,pos=0.5, red]{$0$};
\draw[|-|] (27_9_5)--(27_9_6) node[sloped, above,pos=0.5, red]{$0$};

\node[] at (-6,-10.3) {$\mathcal{M}(\alpha_{28},\alpha_{9})$ = };
\node[label=below:$\tilde{p} \cdot M$, label=above:{\color{blue}$\alpha_{15}$}] (28_9_1)at (-4,-10.3) {};
\node[label=below:$m \cdot M$, label=above:{\color{blue}$\alpha_{21}$}] (28_9_2)at (-2.5,-10.3) {};
\node[label=below:$p \cdot M$, label=above:{\color{blue}$\alpha_{15}$}] (28_9_3)at (-1,-10.3) {};
\node[label=below:$\tilde{m} \cdot M$, label=above:{\color{blue}$\alpha_{21}$}] (28_9_4)at (0.5,-10.3) {};

\node[] at (-1.7,-10.3) {$\cup$};


\draw[|-|] (28_9_1)--(28_9_2) node[sloped, above,pos=0.5, red]{$1$};
\draw[|-|] (28_9_3)--(28_9_4) node[sloped, above,pos=0.5, red]{$1$};
\end{tikzpicture}
\end{center}
where the framings are calculated using a combination of (2) from Proposition \ref{prop:gluing_form_i} (for the cancellation of $\mathcal{M}(\alpha_{20}, \alpha_6)$) and (2) from Proposition \ref{prop:gluing_form_ii} (for the cancellation of $\mathcal{M}(\alpha_{30}, \alpha_{23})$).

Notice that there is now a moduli space which can be cancelled using the Whitney trick again; namely, $\mathcal{M}({\alpha}_{27},{\alpha}_{22}) = \{P \sqcup M \}$. 
Two $1$-dimensional moduli spaces are changed as a result of this cancellation; they become:
\begin{center}
\begin{tikzpicture}
\node[] at (-6,-7.5) {$\mathcal{M}(\alpha_{27},\alpha_{5})$ = };
\node[] at (-4.3,-7.5) {\color{red} 0};
\draw (-4.3,-7.5) circle (0.5cm);

\node[] at (-6,-9) {$\mathcal{M}(\alpha_{27},\alpha_{9})$ = };
\node[label=below:$\tilde{p} \cdot P$, label=above:{\color{blue}$\alpha_{15}$}] (27_9_1)at (-4,-9) {};
\node[label=below:$\tilde{m} \cdot P$, label=above:{\color{blue}$\alpha_{21}$}] (27_9_2)at (-2,-9) {};
\node[label=below:$p \cdot P$, label=above:{\color{blue}$\alpha_{15}$}] (27_9_3)at (0,-9) {};
\node[label=below:$m \cdot P$, label=above:{\color{blue}$\alpha_{21}$}] (27_9_4)at (2,-9) {};

\node[] at (-1,-9) {$\cup$};

\draw[|-|] (27_9_1)--(27_9_2) node[sloped, above,pos=0.5, red]{$0$};
\draw[|-|] (27_9_3)--(27_9_4) node[sloped, above,pos=0.5, red]{$0$};
\end{tikzpicture}
\end{center}
where the framings are calculated using Proposition \ref{prop:whitney_glue_formula_second}.

Since $\mathcal{M}(\alpha,\alpha_5) = \emptyset$ for all objects $\alpha$ different from $\alpha_{22}$, we can cancel $\mathcal{M}(\alpha_{22}, \alpha_5) = p$ with no effect on the other moduli spaces. At the same time, let us cancel the $0$-dimensional moduli space $\mathcal{M}(\alpha_{27}, \alpha_{15}) = P$. The latter cancellation changes a single $0$-dimensional moduli space, which is now
\[
\mathcal{M}(\bar{\alpha}_{28},\bar{\alpha}_{21}) = \mathcal{M}({\alpha}_{28},{\alpha}_{21}) \sqcup \big( \mathcal{M}({\alpha}_{27},{\alpha}_{21}) \times \mathcal{M}({\alpha}_{28},{\alpha}_{15}) \big) = M \sqcup P
\]
and the result is the flow category $\Cat_4$ depicted in Figure \ref{page5}.
\begin{figure}[ht!]
	\includegraphics[scale=0.5]{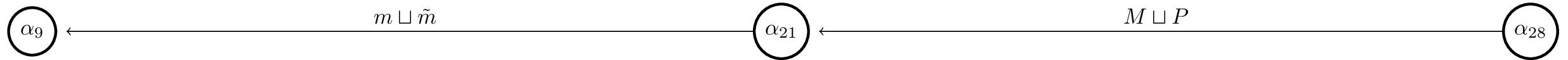}
	\caption{The flow category $\Cat_4$.}
	\label{page5}
\end{figure}
We see that the single $1$-dimensional moduli space of $\Cat_4$ consists of two intervals:
\begin{center}
\begin{tikzpicture}

\node[] at (-6,-13.5) {$\mathcal{M}(\alpha_{28},\alpha_{9})$ = };
\node[label=below:$m \cdot M$, label=above:{\color{blue}$\alpha_{21}$}] (28_9_1)at (-4,-13.5) {};
\node[label=below:$\tilde{m} \cdot P$, label=above:{\color{blue}$\alpha_{21}$}] (28_9_2)at (-2,-13.5) {};
\node[label=below:$m \cdot M$, label=above:{\color{blue}$\alpha_{21}$}] (28_9_3)at (0,-13.5) {};
\node[label=below:$\tilde{m} \cdot P$, label=above:{\color{blue}$\alpha_{21}$}] (28_9_4)at (2,-13.5) {};

\node[] at (-1,-13.5) {$\cup$};


\draw[|-|] (28_9_1)--(28_9_2) node[sloped, above,pos=0.5, red]{$1$};
\draw[|-|] (28_9_3)--(28_9_4) node[sloped, above,pos=0.5, red]{$1$};
\end{tikzpicture}
\end{center}
whose framings are calculated using (2) of Proposition \ref{prop:gluing_form_ii}.

The Whitney trick can be used on the moduli space $\mathcal{M}(\alpha_{28}, \alpha_{21}) = M \sqcup P$, resulting in the final flow category, $\Cat_{Fin}$.
The sole $1$-dimensional moduli space which remains is:
\begin{center}
\begin{tikzpicture}
\node[] at (-6,-13.5) {$\mathcal{M}(\alpha_{28},\alpha_{9})$ = };
\node[] at (-4.3,-13.5) {\color{red} 0};
\draw (-4.3,-13.5) circle (0.5cm);
\end{tikzpicture}
\end{center}
where the framing is calculated using Proposition \ref{prop:whitney_glue_formula_second}.

Now $|\cC_{Fin}|$ is a CW complex consisting of three cells and a basepoint.   In fact $\cC_{Fin}$ is the simplest framed flow category giving rise to $\Sigma^{-1} \RP^5 / \RP^2$, as discussed in Subsection \ref{subsec:chang_easy_examples}.  The Lipshitz-Sarkar stable homotopy type is therefore determined immediately as $\Sigma^{-1} \RP^5 / \RP^2$.

\subsection{A pretzel link}
\label{subsec:pretzel_link}
\input{subsec_pretzel_link.tex}

\subsection{The disjoint union of two trefoils}
\label{subsec:trefoils}
Let us consider the disjoint union of two (right-handed) trefoils, denoted $L$.
\begin{figure}[ht]
 \includegraphics[width=3cm,height=2cm]{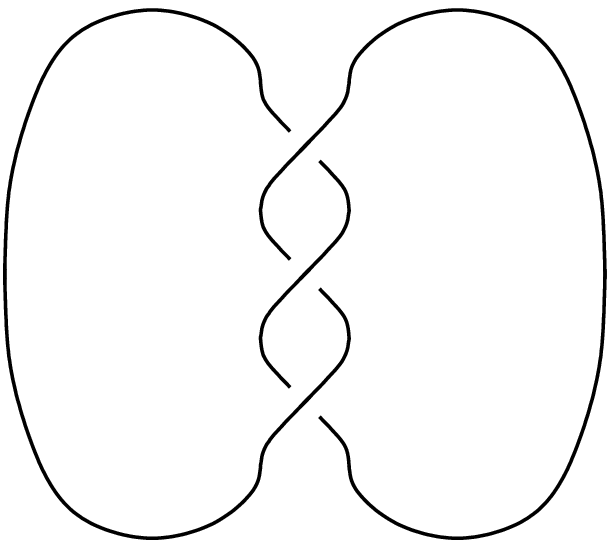}\hspace{0.3cm}\includegraphics[width=3cm,height=2cm]{fig_tref}
\caption{The disjoint union of two trefoil knots.}
\end{figure}

The Lipshitz-Sarkar stable homotopy type of this link is easily derived from that of the trefoil using \cite[Thm.1]{LawLipSar}. Also, the stable homotopy type of the trefoil $T$ is just a wedge of Moore spaces determined by the Khovanov cohomology, but it is also easily derived from the diagram using one elementary tangle of index 3 and \cite[Thm.1.2]{JLS}. In fact, the sock flow category is
\[
 \begin{array}{ccccccc}
  \includegraphics[width=1cm, height=0.6cm]{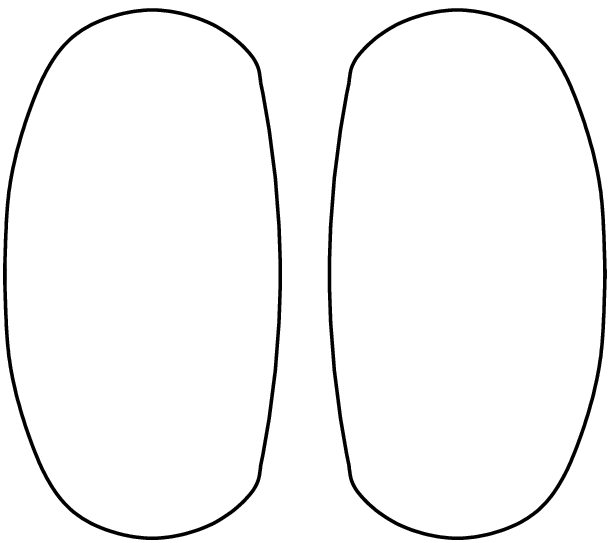} & \stackrel{+}{\longleftarrow} & \includegraphics[width=1cm, height=0.6cm]{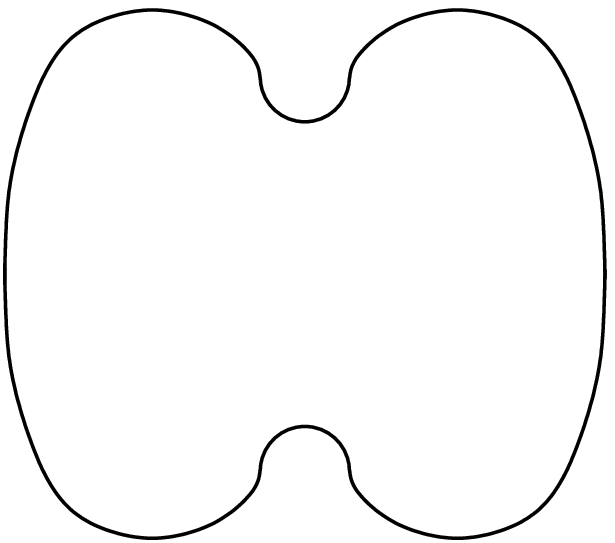} & \stackrel{+-}{\longleftarrow} & \includegraphics[width=1cm, height=0.6cm]{fig_tref1} & \stackrel{++}{\longleftarrow} & \includegraphics[width=1cm, height=0.6cm]{fig_tref1}\\
 0 & & 1 & & 2 & & 3
 \end{array}
\]
In quantum degree $q=7$ we only get two objects, the circle above $2$ with a $+$ label and the circle above $3$ with a $-$ label, and $2$ positively framed points between them. In particular,
\[
 \X^\Kh_7(T) \simeq M(\Z/2,2).
\]
By \cite[Thm.1]{LawLipSar} $\X^\Kh_{14}(L)$ contains $M(\Z/2,2)\wedge M(\Z/2,2)$ as a wedge summand, an elementary Chang complex (see \cite[\S 11]{baues}) and not decomposable into Moore spaces. Nevertheless, let us consider the flow category $\cL^\Kh_{14}(L_{(3,3)})$. There are in fact only $10$ objects, and the objects based at $(1,3)$ and $(3,1)$ can be cancelled with the objects based at $(0,3)$ and $(3,0)$. The remaining $6$ objects are depicted in Figure \ref{fig_flowcattrefs}.
\begin{figure}[ht]
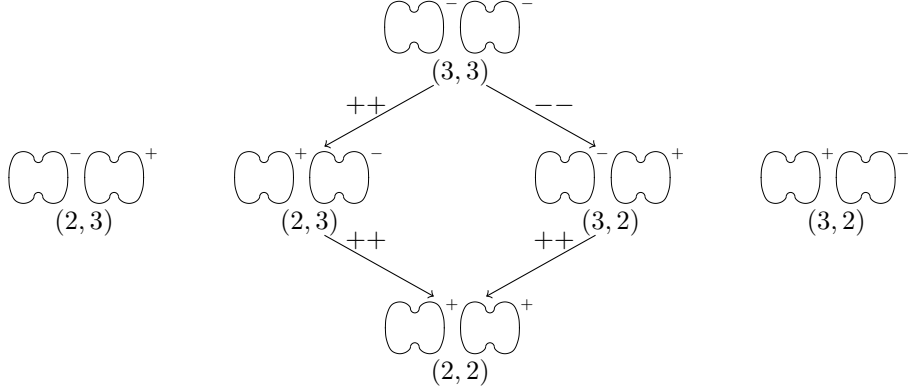

\begin{tikzpicture}
 \node at (0,0) {$\stackrel{\includegraphics[width=0.8cm]{fig_tref1}^-\includegraphics[width=0.8cm]{fig_tref1}^-}{(3,3)}$};
 \node at (-5,-2) {$\stackrel{\includegraphics[width=0.8cm]{fig_tref1}^-\includegraphics[width=0.8cm]{fig_tref1}^+}{(2,3)}$};
 \node at (-2,-2) {$\stackrel{\includegraphics[width=0.8cm]{fig_tref1}^+\includegraphics[width=0.8cm]{fig_tref1}^-}{(2,3)}$}; 
 \node at (5,-2) {$\stackrel{\includegraphics[width=0.8cm]{fig_tref1}^+\includegraphics[width=0.8cm]{fig_tref1}^-}{(3,2)}$};
 \node at (2,-2) {$\stackrel{\includegraphics[width=0.8cm]{fig_tref1}^-\includegraphics[width=0.8cm]{fig_tref1}^+}{(3,2)}$};
 \node at (0,-4) {$\stackrel{\includegraphics[width=0.8cm]{fig_tref1}^+\includegraphics[width=0.8cm]{fig_tref1}^+}{(2,2)}$};

 \draw [shorten >=0.8cm,shorten <=0.4cm,->] (0,-0.4) -- node [above] {$++$} (-2.5,-1.8);
 \draw [shorten >=0.8cm,shorten <=0.4cm,->] (0,-0.4) -- node [above] {$--$} (2.5,-1.8);
 \draw [shorten >=0.8cm,shorten <=0.4cm,<-] (0,-3.6) -- node [above] {$++$} (-2.5,-2.2);
 \draw [shorten >=0.8cm,shorten <=0.4cm,<-] (0,-3.6) -- node [above] {$++$} (2.5,-2.2);
\end{tikzpicture}
\caption{The flow category $\cL^\Kh_{14}(L_{(3,3)})$ after the cancellation of $4$ objects.}
\label{fig_flowcattrefs}
\end{figure}

Notice that the outside objects in the middle row are isolated from all other objects. This means they contribute two copies of $S^5$ to the stable homotopy type. To describe the $1$-dimensional moduli spaces, let us introduce more notation. We call the top object $\gamma$, the bottom object $\alpha$ the middle objects $\beta_1$ and $\beta_2$ where $\beta_1$ is based at $(2,3)$ while $\beta_2$ is based at $(3,2)$. Write 
\begin{align*}
\M(\gamma,\beta_1) &= \{P_0,P_1\} &  \M(\gamma,\beta_2) = \{\tilde{P}_0,\tilde{P}_1\} \\
\M(\beta_1,\alpha) &= \{p_0,p_1\} &  \M(\beta_2,\alpha) = \{\tilde{p}_0,\tilde{p}_1\}.
\end{align*}
Following \cite[\S 5.2]{JLS}, we get four intervals in $\M(\gamma,\alpha)$ as follows:

\begin{tikzpicture}
\node[label=below:$p_0\,P_0$, label=above:{\color{blue}$\beta_1$}] (5_5_1)at (0,0) {};
\node[label=below:$\tilde{p}_0\,\tilde{P}_0$, label=above:{\color{blue}$\beta_2$}] (5_5_2)at (1.6,0) {};

\draw[|-|] (5_5_1)--(5_5_2);

\node[label=below:$p_0\,P_1$, label=above:{\color{blue}$\beta_1$}] (5_5_3)at (3.3,0) {};
\node[label=below:$\tilde{p}_1\,\tilde{P}_0$, label=above:{\color{blue}$\beta_2$}] (5_5_4)at (4.9,0) {};

\draw[|-|] (5_5_3)--(5_5_4);

\node[label=below:$p_1\,P_0$, label=above:{\color{blue}$\beta_1$}] (5_5_5)at (6.6,0) {};
\node[label=below:$\tilde{p}_0\,\tilde{P}_1$, label=above:{\color{blue}$\beta_2$}] (5_5_6)at (8.2,0) {};

\draw[|-|] (5_5_5)--(5_5_6);

\node[label=below:$p_1\,P_1$, label=above:{\color{blue}$\beta_1$}] (5_5_7)at (9.9,0) {};
\node[label=below:$\tilde{p}_1\,\tilde{P}_1$, label=above:{\color{blue}$\beta_2$}] (5_5_8)at (11.5,0) {};

\draw[|-|] (5_5_7)--(5_5_8);
\end{tikzpicture}

Since $\alpha$ is based at $(2,2)$ and all four intervals correspond to a $(P,P)$-cell in the obstruction complex (see \cite[\S 5.2]{JLS}), all intervals are framed $0$. We cannot perform a Whitney trick or a handle cancellation directly, but consider the flow category $\Cat$ depicted in Figure \ref{fig_supcat}.

\begin{figure}[ht]
 \begin{tikzpicture}[scale =1,
roundnode/.style={circle, draw=black!, very thick}
]
\draw [very thick] (-3,0) circle [radius=0.3];
\draw [very thick] (3,0) circle [radius=0.3];
\draw [very thick] (-3,-2) circle [radius=0.3];
\draw [very thick] (0,-2) circle [radius=0.3];
\draw [very thick] (3,-2) circle [radius=0.3];
\draw [very thick] (0,-4) circle [radius=0.3];

\node at (-3,0) {$\gamma$};
\node at (3,0) {$\tau$};
\node at (-3,-2) {$\beta_1$};
\node at (0,-2) {$\beta_2$};
\node at (3,-2) {$\sigma$};
\node at (0,-4) {$\alpha$};

\draw [->, shorten >=0.4cm, shorten <=0.4cm] (-3,0)--(-3,-2) node[sloped,below,pos=0.4]{$P_0\,P_1$};
\draw [->, shorten >=0.4cm, shorten <=0.4cm] (-3,0)--(0,-2) node[sloped,above,pos=0.3]{$\tilde{P}_0\,\tilde{P}_1$};
\draw [->, shorten >=0.4cm, shorten <=0.4cm, gray] (3,0)--(-3,-2) node[sloped,above,pos=0.3]{$-$};
\draw [->, shorten >=0.4cm, shorten <=0.4cm, gray] (3,0)--(0,-2) node[sloped,above,pos=0.7]{$+$};
\draw [->, shorten >=0.4cm, shorten <=0.4cm, gray] (3,0)--(3,-2) node[sloped,above,pos=0.3]{$+$};
\draw [->, shorten >=0.4cm, shorten <=0.4cm] (-3,-2)--(0,-4) node[sloped,below,pos=0.4]{$p_0\,p_1$};
\draw [->, shorten >=0.4cm, shorten <=0.4cm] (0,-2)--(0,-4) node[sloped,above,pos=0.5]{$\tilde{p}_0\,\tilde{p}_1$};
\end{tikzpicture}
\caption{The flow category $\Cat$.}
\label{fig_supcat}
\end{figure}
The moduli space $\M(\gamma,\alpha)$ is as in $\cL^\Kh_{14}(L_{(3,3)})$, and $\M(\tau,\alpha)$ is given by two intervals
\[
\begin{tikzpicture}
\node[] at (-1.5,0) {$\mathcal{M}(\tau,\alpha) =$};
\node[label=below:$p_0\,-$, label=above:{\color{blue}$\beta_1$}] (5_5_1)at (0,0) {};
\node[label=below:$\tilde{p}_1\,+$, label=above:{\color{blue}$\beta_2$}] (5_5_2)at (1.6,0) {};
\node[label=below:$p_1\,-$, label=above:{\color{blue}$\beta_1$}] (5_5_3)at (3.3,0) {};
\node[label=below:$\tilde{p}_0\,+$, label=above:{\color{blue}$\beta_2$}] (5_5_4)at (4.9,0) {};

\node[] at (2.4,0) {$\cup$};

\draw[|-|] (5_5_3)--(5_5_4) node[sloped, above,pos=0.5, red]{$\varepsilon_2$};
\draw[|-|] (5_5_1)--(5_5_2) node[sloped, above,pos=0.5, red]{$\varepsilon_1$};
\end{tikzpicture}
\]
for some $\varepsilon_1,\varepsilon_2\in \Z/2$. 

If we cancel the pair $\tau,\sigma$ we obtain the relevant part of the flow category $\cL^\Kh_{14}(L)$. But if we cancel the pair $\tau,\beta_2$ we get the flow category $\overline{\Cat}$ given in Figure \ref{fig_supcatcan}.

\begin{figure}[ht]
 \begin{tikzpicture}[scale =1,
roundnode/.style={circle, draw=black!, very thick}
]
\draw [very thick] (0,0) circle [radius=0.3];
\draw [very thick] (-3,-2) circle [radius=0.3];
\draw [very thick] (3,-2) circle [radius=0.3];
\draw [very thick] (0,-4) circle [radius=0.3];

\node at (0,0) {$\bar{\gamma}$};
\node at (-3,-2) {$\bar{\beta_1}$};
\node at (3,-2) {$\bar{\sigma}$};
\node at (0,-4) {$\bar{\alpha}$};

\draw [->, shorten >=0.4cm, shorten <=0.4cm] (0,0)--(-3,-2) node[sloped,above,pos=0.5]{$P_0\,P_1\, M_0\,M_1$};
\draw [->, shorten >=0.4cm, shorten <=0.4cm] (0,0)--(3,-2) node[sloped,above,pos=0.3]{$\hat{P}_0\,\hat{P}_1$};
\draw [->, shorten >=0.4cm, shorten <=0.4cm] (-3,-2)--(0,-4) node[sloped,below,pos=0.4]{$p_0\,p_1$};
\end{tikzpicture}
\caption{The flow category $\overline{\Cat}$.}
\label{fig_supcatcan}
\end{figure}

We denote the new points in $\M(\bar{\gamma},\bar{\beta_1})$ by $M_0 = (-,\tilde{P}_0)$ and $M_1=(-,\tilde{P}_1)$, and both are framed negatively.
Using Proposition \ref{prop:gluing_form_ii}(5) we get that $\M(\bar{\gamma},\bar{\alpha})$ still consists of $4$ intervals given by

\begin{tikzpicture}
\node[label=below:$p_0\,P_0$, label=above:{\color{blue}$\beta_1$}] (5_5_1)at (0,0) {};
\node[label=below:$p_1\,M_0$, label=above:{\color{blue}$\beta_1$}] (5_5_2)at (1.6,0) {};

\draw[|-|] (5_5_1)--(5_5_2) node[sloped, above,pos=0.5, red]{$\varepsilon_2$};

\node[label=below:$p_0\,P_1$, label=above:{\color{blue}$\beta_1$}] (5_5_3)at (3.3,0) {};
\node[label=below:$p_0\,M_0$, label=above:{\color{blue}$\beta_1$}] (5_5_4)at (4.9,0) {};

\draw[|-|] (5_5_3)--(5_5_4) node[sloped, above,pos=0.5, red]{$\varepsilon_1$};

\node[label=below:$p_1\,P_0$, label=above:{\color{blue}$\beta_1$}] (5_5_5)at (6.6,0) {};
\node[label=below:$p_1\,M_1$, label=above:{\color{blue}$\beta_1$}] (5_5_6)at (8.2,0) {};

\draw[|-|] (5_5_5)--(5_5_6) node[sloped, above,pos=0.5, red]{$\varepsilon_2$};

\node[label=below:$p_1\,P_1$, label=above:{\color{blue}$\beta_1$}] (5_5_7)at (9.9,0) {};
\node[label=below:$p_0\,M_1$, label=above:{\color{blue}$\beta_1$}] (5_5_8)at (11.5,0) {};

\draw[|-|] (5_5_7)--(5_5_8) node[sloped, above,pos=0.5, red]{$\varepsilon_1$};
\end{tikzpicture}

After performing the Whitney trick with $P_0,M_0$ and using Proposition \ref{prop:whitney_glue_formula_second} the $1$-dimensional moduli space changes to
\[
\begin{tikzpicture}
\node[label=below:$p_0\,P_1$, label=above:{\color{blue}$\beta_1$}] (5_5_1)at (0,0) {};
\node[label=below:$p_0$] (5_5_2)at (1.6,0) {};
\node[label=below:$p_1$] (5_5_3)at (3.2,0) {};
\node[label=below:$p_1\,M_1$, label=above:{\color{blue}$\beta_1$}] (5_5_6)at (4.8,0) {};
\node[label=below:$p_1\,P_1$, label=above:{\color{blue}$\beta_1$}] (5_5_7)at (6.4,0) {};
\node[label=below:$p_0\,M_1$, label=above:{\color{blue}$\beta_1$}] (5_5_8)at (8,0) {};

\draw[|-|] (5_5_1)--(1.6,0);
\draw[|-|] (1.6,0)--(3.2,0) node[sloped, above,pos=0.5, red]{$\varepsilon_2+\varepsilon_1+\varepsilon_2$};
\draw[|-|] (3.2,0)--(5_5_6);
\draw[|-|] (5_5_7)--(5_5_8) node[sloped, above,pos=0.5, red]{$\varepsilon_1$};
\end{tikzpicture}
\]
and after the Whitney trick with $P_1,M_1$ we get a circle with label $\varepsilon_1+\varepsilon_2+\varepsilon_1+\varepsilon_2=0$ in $\M(\bar{\gamma},\bar{\alpha})$. So the cell of $\gamma$ is attached to the cell of $\alpha$ in a non-trivial way, leading to
\[
 \X^\Kh_{14}(L) \simeq \Sigma^2 (\RP^2\wedge \RP^2) \vee S^5 \vee S^5 {\rm .}
\]

\begin{remark}
 It is possible to get more directly from $\cL^\Kh_{14}(L_{(3,3)})$ to the last flow category using a \emph{handle slide}.  In particular, we would not have to increase the number of objects to identify the flow category as a standard Chang complex in the sense of \cite[\S 11]{baues}.  We shall give the construction of a handle slide on a framed flow category in an upcoming paper.
\end{remark}

\bibliographystyle{myamsalpha}
\bibliography{master_morse}
\end{document}